\documentclass[12pt,reqno]{amsart}
\usepackage[colorlinks=true, pdfstartview=FitH, linkcolor=blue, citecolor=blue, urlcolor=blue]{hyperref}
\usepackage{fullpage}
\usepackage{amsmath,amssymb,amsthm,graphicx}
\usepackage{cite}
\usepackage{tikz}

\numberwithin{equation}{section}
\theoremstyle{plain}
\newtheorem{theorem}{Theorem}[section]
\newtheorem{lemma}[theorem]{Lemma}
\newtheorem{prop}[theorem]{Proposition}

\theoremstyle{definition}
\newtheorem{defn}[theorem]{Definition}

\newtheorem{remark}[theorem]{Remark}
\newtheorem*{ex}{Example}

\newcommand{\N}{\mathbb{N}}
\newcommand{\R}{\mathbb{R}}
\newcommand{\Z}{\mathbb{Z}}

\newcommand{\xmid}{\parallel} 

\newcommand{\ep}{\varepsilon}

\renewcommand{\mod}[1]{{\ifmmode\text{\rm\ (mod~$#1$)}\else\discretionary{}{}{\hbox{ }}\rm(mod~$#1$)\fi}}
\DeclareMathOperator{\sdif}{\bigtriangleup} 
\DeclareMathOperator{\lcm}{lcm} 
\DeclareMathOperator{\erf}{erf} 
\DeclareMathOperator{\erfi}{erfi} 

\newcommand{\Ex}{\mathbf{E}} 
\newcommand{\Mn}{\mathbf{Mn}} 
\newcommand{\Mx}{\mathbf{Mx}} 

\newcommand{\den}{\mathfrak{d}} 
\newcommand{\inv}{\mathfrak{I}} 
\newcommand{\pri}{\mathfrak{P}} 
\newcommand{\prm}{\mathfrak{Q}} 
\newcommand{\iV}{\mathfrak{V}} 
\newcommand{\iW}{\mathfrak{W}} 
\newcommand{\lV}{\log\mathfrak{V}} 
\newcommand{\OT}{\mathbf{T}} 
\newcommand{\SF}{\mathbf{S}} 
\newcommand{\UF}{\mathbf{U}} 

\newcommand{\msg}{\mathfrak{S}} 

\newcommand{\sif}{\mathfrak{f}} 
\newcommand{\siF}{\mathfrak{F}} 
\newcommand{\TQ}{\mathcal{Q}} 
\newcommand{\tq}{\mathfrak{q}} 

\newcommand{\cD}{\mathcal D}
\newcommand{\cE}{\mathcal E}
\newcommand{\cH}{\mathcal H}

\newcommand{\theoremoverall}{
Given $m\in\N$, let $f_1(m), f_2(m), \ldots$ be the distinct invariant factor orders of $(\Z/m\Z)^\times$ listed in increasing order. For any positive integer~$D$, almost all numbers~$m$ satisfy $f_i(m) \in \sif_i$ for $i = 1,\ldots,D$. In other words, the set of numbers~$m$ for which $f_i(m) \notin \sif_i$ has density~$0$ for any $i\in\N$, as does the set of numbers~$m$ for which the length of the invariant factor decomposition of $(\Z/m\Z)^\times$ has fewer than~$D$ distinct orders.}

\newcommand{\theoremexgeneral}{%
If $d \in \sif_i$ for some $i\in\N$, then
\[ \Ex_n(\inv(m;d)) = \mu_i\log\log n + O_d \bigl( (\log\log n)^{1/2} \bigr) . \]}
\newcommand{\theoremexA}{%
If $i>1$ is an integer such that $\#\sif_{i-1} = \#\sif_i = \#\sif_{i+1} = 1$, then for $d \in \sif_i$,
\[ \Ex_n(\inv(m;d)) = \mu_i\log\log n + O_d(1) . \]}
\newcommand{\theoremexB}{%
If $i>1$ is an integer such that $\#\sif_{i-1} = 2$ and $\#\sif_i = \#\sif_{i+1} = 1$, then for $d \in \sif_i$,
\[ \Ex_n(\inv(m;d)) = \mu_i\log\log n-\bigg(\frac{\sigma_{i-1}}{\sqrt{2\pi}}+o(1)\bigg)(\log\log n)^{1/2} . \]}
\newcommand{\theoremexC}{%
If $i$ is an integer such that  $\#\sif_i = 2$, then for $d \in \sif_i$,
\[ \Ex_n(\inv(m;d)) = \bigg(\frac{\sigma_i}{\sqrt{2\pi}}+o(1)\bigg)(\log\log n)^{1/2} . \]}
\newcommand{\theoremexD}{%
If $i>1$ is an integer such that $\#\sif_{i-1} = 2$ and $\sif_i=1$ and $\#\sif_{i+1} = 2$, then for $d \in \sif_i$,
\[ \Ex_n(\inv(m;d)) = \mu_i\log\log n-\bigg(\frac{\sigma_{i-1}+\sigma_{i+1}}{\sqrt{2\pi}}+o(1)\bigg)(\log\log n)^{1/2} . \]}
\newcommand{\theoremexE}{%
If $i>1$ is an integer such that $\#\sif_{i-1} = \#\sif_i = 1$ and $\#\sif_{i+1} = 2$, then for $d \in \sif_i$,
\[ \Ex_n(\inv(m;d)) = \mu_i\log\log n-\bigg(\frac{\sigma_{i+1}}{\sqrt{2\pi}}+o(1)\bigg)(\log\log n)^{1/2} . \]
The same statement holds when $i=1$ (so that $d=2$).}
\newcommand{\theoremexrare}{%
If $d\notin\siF$, then for any $k \in \N$ and $r > 0$,
\[ \Ex_n(\inv(m;d)^k) \ll_{d,r,k} (\log\log n)^{-r} . \]}

\newcommand{\theoremdistgeneral}{%
Let $d \in \N$. If $d\in\siF$ is a universal factor order,
the limiting distribution of $\inv(m;d)$ under appropriate rescaling is
fully characterized by the number of elements in $\sif_{i-1}$, $\sif_i$, and $\sif_{i+1}$,
except for the special cases $d=2\in\sif_1$ and $d=12\in\sif_3$.
If $d\notin\siF$ is a rare factor order,
then the limiting distribution is fully characterized by this fact.}

\newcommand{\theoremdistA}{%
If $i>1$ is an integer such that $\#\sif_{i-1} = \#\sif_i = \#\sif_{i+1} = 1$, then for $d \in \sif_i$,
\[ \lim_{n\to\infty} P_n\bigg(\frac{\inv(m;d) - \mu_i\log\log n}{\sigma_i(\log\log n)^{1/2}} \le x\bigg) = \Phi(x) , \]
where $\Phi(x)$ is the standard normal cumulative distribution function.}

\newcommand{\theoremdistB}{%
If $i>1$ is an integer such that $\#\sif_{i-1} = 2$ and $\#\sif_i = \#\sif_{i+1} = 1$, then for $d \in \sif_i$,
\[
\lim_{n\to\infty} P_n\bigg(\frac{\inv(m;d) - \mu_i\log\log n}{\sigma_i(\log\log n)^{1/2}} \le x\bigg) = \Phi\biggl(x;-\frac{\nu_i}{\sqrt{\sigma_i^2+\sigma_{i;2,1}^2}}\biggr),
\]
a skew-normal cumulative distribution function as in Definition~\ref{defn:skew-normal}.
The characteristic function of this distribution is $e^{-t^2/2}\big(1-\eta(\frac{\nu_i}{2\sigma_i} t)\big)$.}

\newcommand{\theoremdistC}{%
If $i$ is an integer such that  $\#\sif_i = 2$, then for $d \in \sif_i$,
\[ \lim_{n\to\infty} P_n\bigg(\frac{\inv(m;d)}{\sigma_i(\log\log n)^{1/2}} \le x\bigg)
= \begin{cases} \Phi(x), & \text{if } x \ge 0, \\ 0, & \text{otherwise}. \end{cases} \]
Furthermore,
\[ \lim_{n\to\infty} P_n\bigg(\frac{\sum_{d \in \sif_i} \inv(m;d)}{\sigma_i(\log\log n)^{1/2}} \le x\bigg) = \Phi_+(x) . \]}

\newcommand{\theoremdistD}{%
If $i>3$ is an integer such that $\#\sif_{i-1} = 2$ and $\sif_i=1$ and $\#\sif_{i+1} = 2$, then for $d \in \sif_i$,
\begin{multline*}
\lim_{n\to\infty} P_n\bigg(\frac{\inv(m;d) - \mu_i\log\log n}{\sigma_i(\log\log n)^{1/2}} \le x\bigg) \\
= \UF\biggl( 4x;
  2\frac{\sqrt{\sigma_i^2+\sigma_{i;1,2}^2+\sigma_{i;2,1}^2+\sigma_{i;2,2}^2}}{\sigma_i},
  2\frac{\sqrt{\sigma_i^2+\sigma_{i;1,2}^2-\sigma_{i;2,1}^2-\sigma_{i;2,2}^2}}{\sigma_i},
  2\frac{\sqrt{\sigma_i^2-\sigma_{i;1,2}^2+\sigma_{i;2,1}^2-\sigma_{i;2,2}^2}}{\sigma_i} \biggr),
\end{multline*}
where $\UF(x;\sigma_1,\sigma_2,\sigma_3)$ is given in Definition~\ref{defn:U_function} below. 
The characteristic function of this distribution is
\begin{align*}
e^{-t^2/2}\bigg(1 - \eta\bigg(\frac{\nu_i}{2\sigma_i} t\bigg) \bigg) \bigg( 1 - \eta\bigg(\frac{\nu_{i+1}}{2\sigma_i} t\bigg) \bigg) .
\end{align*}}

\newcommand{\theoremdistDalt}{%
If $i = 3$ (so that $d=12$), then 
\[ \lim_{n\to\infty} P_n\bigg(\frac{\inv(m;12) - \frac14\log\log n}{(\frac 12\log\log n)^{1/2}}
\le x\bigg) = \Mn_F(x) , \] 
where~$F$ is the singular four-variable normal distribution centred at the origin with covariance matrix
\[ M = \left[ \begin{array}{cccc}
1 & 5/8 & 1/2 & 1/8 \\ 5/8 & 1 & -1/8 & 1/4 \\ 1/2 & -1/8 & 1 & 3/8 \\ 1/8 & 1/4 & 3/8 & 1/2 
\end{array} \right] , \]
and $\Mn$ represents the minimum of those four random variables (see Definition~\ref{defn:min_max_distributions}).
}

\newcommand{\theoremdistE}{%
If $i>1$ is an integer such that $\#\sif_{i-1} = \#\sif_i = 1$ and $\#\sif_{i+1} = 2$, then for $d \in \sif_i$,
\[
\lim_{n\to\infty} P_n\bigg(\frac{\inv(m;d) - \mu_i\log\log n}{\sigma_i(\log\log n)^{1/2}} \le x\bigg) = \Phi\biggl(x;-\frac{\nu_{i+1}}{\sqrt{\sigma_i^2+\sigma_{i;1,2}^2}}\biggr).
\]
The characteristic function of this skew-normal distribution is
$e^{-t^2/2}\big(1-\eta(\frac{\nu_{i+1}}{2\sigma_i} t)\big)$.
The same statement holds when $i=1$ (so that $d=2$).}

\newcommand{\theoremdistrare}{%
If $d\notin\siF$, then almost all integers $m$ satisfy $\inv(m;d)=0$.}

\newcommand{\definitionphisequence}{%
We define the {\em total $\varphi$-sequence} to be the infinite sequence
\[
(\tq_i) = (2,3,4,5,8,7,9,16,11,13,17,32,19,27,25,23,29,\ldots)
\]
which consists of every prime power, ordered by the values the Euler $\varphi$-function takes on them, with ties broken by their size.
More precisely, the total $\varphi$-sequence is ordered so that the following two conditions are satisfied:
\begin{enumerate}
\item	If $i<j$ then $\varphi(\tq_i) \le \varphi(\tq_j)$;
\item	If $i < j$ and $\varphi(\tq_i) = \varphi(\tq_j)$, then $\tq_i < \tq_j$. 
\end{enumerate}
We define a {\em (finite) $\varphi$-sequence} to be a finite contiguous section of the total $\varphi$-sequence, that is, a finite sequence of the form $(\tq_i, \tq_{i+1}, \dots, \tq_{j-1}, \tq_j)$ for some positive integers $i\le j$.
}

\newcommand{\firsttenUFOsets}{%
(\sif_i) = \bigl( \{2\}, \{4,6\}, \{12\}, \{24,60\}, \{120\}, \{360,840\}, \{2520\}, \\
\{5040\}, \{55440\}, \{720720\}, \ldots \bigr)
} 

\newcommand{\definitionUFOsets}{%
We define the \emph{universal factor order sets} $(\sif_i)_{i=1}^\infty$ to be
\[ \sif_i = \begin{cases}
\{ \lcm[\tq_1, \tq_2, \dots, \tq_i ] \},
& \text{if } \varphi(\tq_i) < \varphi(\tq_{i+1}) , \\
\{ \lcm[ \tq_1, \tq_2, \dots, \tq_i ], \lcm[ \tq_1, \tq_2, \dots, \tq_{i-1}, \tq_{i+1} ] \}, & 
\text{if } \varphi(\tq_i) = \varphi(\tq_{i+1}) . \end{cases} \]
We set $\siF = \bigcup_{i=1}^\infty \sif_i$.
A positive integer~$d$ is said to be a \emph{universal factor order} if $d \in \siF$
and is otherwise a \emph{rare factor order}.
}

\newcommand{\definitioncountingfunctionsinv}{%
Define $\inv(m)$ to be the total number of invariant factors (counted with multiplicity) of the multiplicative group $(\Z/m\Z)^\times$. For any positive integer~$d$, define $\inv(m;d)$ to be the number of times $\Z_d$ appears as an invariant factor of $(\Z/m\Z)^\times$.
}

\newcommand{\definitioncountingfunctionspri}{%
Let $\pri(m;p^\alpha)$ count the number of times $p^\alpha$ appears as an elementary divisor of the multiplicative group $(\Z/m\Z)^\times$, so that $\prm(m;p^\alpha)= \sum_{\beta=\alpha}^\infty \pri(m;p^\beta)$.
}

\newcommand{\definitioncountingfunctionsprm}{%
Let $\prm(m;p^\alpha)$ count the number of elementary divisors of the multiplicative group $(\Z/m\Z)^\times$ that are of the form $p^\beta$ where $\beta\ge\alpha$. In particular, $\prm(m;p)$ is the number of $p$-groups in the elementary divisor decomposition of $(\Z/m\Z)^\times$.
}

\newcommand{\definitionS}{%
We define
\[ \SF(x;\sigma_1,\sigma_2,\sigma_3) = \lim_{y\to-\infty} \lim_{T\to\infty} \frac 1{2\pi}
\int_{-T}^T \frac{e^{-ity}-e^{-itx}}{it} e^{-(\sigma_1^2+\sigma_2^2+\sigma_3^2)t^2/2} 
\eta\bigg(\frac{\sigma_2 t}{\sqrt 2}\bigg)\eta\bigg(\frac{\sigma_3 t}{\sqrt 2}\bigg) \, dt . \]
We remark that a formula for the derivative of this function with respect to $x$ can be found in~\cite[Lemma~2.2]{skew_normal_sum}.
}

\newcommand{\definitionU}{%
We define $\Sigma = \sqrt{\sigma_1^2+\sigma_2^2+\sigma_3^2}$ and
\[
\UF(x;\sigma_1,\sigma_2,\sigma_3) = \Phi\Big(\frac x\Sigma\Big) 
+ 2\OT\Big(\frac x\Sigma, \frac{\sigma_2}{\sqrt{\sigma_1^2+\sigma_3^2}}\Big)
+ 2\OT\Big(\frac x\Sigma, \frac{\sigma_3}{\sqrt{\sigma_1^2+\sigma_2^2}}\Big)
+ \SF(x;\sigma_1,\sigma_2,\sigma_3),
\]
where $\OT$ is as in Definition~\ref{Owen T def}.
}

\newcommand{\definitionnu}{%
Given positive integers $i$, $k$, and $\ell$, define
\begin{align*}
\nu_i = \sqrt{ \frac1{\varphi(\tq_i)} - \frac1{\varphi(\tq_i)^2} }
\quad\text{and}\quad
\sigma_{i;k,\ell} = \sqrt{ \frac1{\varphi(\tq_i)^k} + \frac1{\varphi(\tq_{i+1})^\ell} - \frac2{\varphi(\tq_i)\varphi(\tq_{i+1})} } .
\end{align*}
}

\newcommand{\definitionmusigma}{%
The mu-sequence
\[
( \mu_i ) = \bigg( \frac 12,0,\frac 14,0,\frac 1{12},0,\frac 1{24},\frac 1{40},\frac 1{60},\frac 1{48},0,\frac 1{144},0,\frac 1{180},\frac 1{220},\frac 3{308},\ldots \bigg)
\]
and the sigma-sequence
\[
( \sigma_i ) = \bigg( \sqrt{\frac 12}, \sqrt{\frac 12} ,\sqrt{\frac 12}, \sqrt{\frac 38}, \sqrt{\frac 13}, \sqrt{\frac 5{18}}, \sqrt{\frac 14}, \sqrt{\frac 15}, \sqrt{\frac 16}, \sqrt{\frac{13}{96}}, \sqrt{\frac{15}{128}}, \sqrt{\frac 19}, \sqrt{\frac{17}{162}}, \ldots \bigg) 
\]
are defined, for all $i\in\N$, by
\begin{align*}
\mu_i = \frac1{\varphi(\tq_i)} - \frac1{\varphi(\tq_{i+1})}
\quad\text{and}\quad
\sigma_i = \sqrt{ \frac1{\varphi(\tq_i)} + \frac1{\varphi(\tq_{i+1})} - \frac2{\varphi(\tq_i)\varphi(\tq_{i+1})} } ,
\end{align*}
where $\varphi$ is Euler's totient function.
}

\newcommand{\Enotation}{%
We use of the general notation
\[
\Ex_n(f(m)) = \frac 1n \sum_{m=1}^n f(m)
\]
to denote the expectation of the function $f(m)$ when $m$ is chosen uniformly at random from $\{1,2,\dots,n\}$.
}

\newcommand{\definitionMnMx}{%
For any $n$-dimensional probability distribution $F$, define two functions
\begin{align*}
\Mx_F(x) &= F(x,\ldots,x) , \\
\Mn_F(x) &= \sum_{k=1}^n (-1)^{k-1} \sum_{\substack{S \subset \{1,\ldots,n\} \\ \# S = k}} F(\vec x_S) ,
\end{align*}
where for any $S \subset \{1,\ldots,n\}$ and $x \in \R$,
the vector $\vec x_S \in \bar\R^n$ has $i$th component equal to $x$ if $i \in S$
has $i$th component equal to $\infty$ otherwise.
}

\newcommand{\definitionomegafunctions}{%
For any set~$S$ of primes, we define $\omega(m;S) = \#\{ p\mid m \colon p\in S \}$. We use a shorthand notation for the important special case
\[
\omega(m;q,a) = \omega \bigl( m; \{ p \equiv a \mod q \} \bigr) = \#\{ p\mid m\colon p\equiv a\mod q\}.
\]
Further, for any two sets~$S$ and~$T$ of primes, we define $\omega(m;S-T) = \omega(m;S)-\omega(m;T)$.
}

\newcommand{\definitiontypical}{%
Given a real number $y \ge 1$, an integer $m \in \N$ is said to be {\em $y$-typical} if for any two prime powers~$q_1$ and~$q_2$:
\begin{enumerate}
\item	If $\varphi(q_1) < \varphi(q_2) \le y$, then $\prm(m;q_1) \ge \prm(m;q_2)$;
\item	If $\varphi(q_1) \le y < \varphi(q_2)$, then $\prm(m;q_1) > \prm(m;q_2)$.
\end{enumerate}
}

\newcommand{\definitiontwototient}{%
A {\em two-totient $\varphi$-sequence} is a $\varphi$-sequence $\{\tq_i,\dots,\tq_j\}$ satisfying:
\begin{itemize}
\item $\{\varphi(\tq_i),\dots,\varphi(\tq_j)\}$ has exactly two distinct elements;
\item If $k<i$ then $\varphi(\tq_k)<\varphi(\tq_i)$;
\item If $k>j$ then $\varphi(\tq_k)>\varphi(\tq_j)$.
\end{itemize}
In other words, for every two consecutive prime-power totients, there is a unique two-totient $\varphi$-sequence that is the inverse image of those two totient values under the Euler $\varphi$-function.
}

\newcommand{\definitionVW}{%
We define
\[
\iV(x) = \lcm\{ p^\alpha \colon \varphi(p^\alpha) \le x \} \quad\text{and}\quad \iW(x) = \#\{ p^\alpha \colon \varphi(p^\alpha) \le x \} .
\]
}

\newcommand{\definitionABCD}{%
For any real-valued additive function $f \colon \N \to \R$, define
\begin{align*}
A_f(x) &= \sum_{p \le x} \frac{f(p)}p, & 
C_f(x) &= \max\{ |f(p^\alpha)| \colon p^\alpha \le x \}, \\
B_f(x) &= \bigg(\sum_{p \le x} \frac{|f(p)|^2}p\bigg)^{1/2}, &
D_f(x) &= \bigg(\sum_{p^\alpha \le x} \frac{f(p^\alpha)^2}{p^\alpha}\bigg)^{1/2}.
\end{align*}
}

\newcommand{\definitionDd}{%
The \emph{Dirichlet density} of a set~$S$ of primes is
\[ \den(S) = \lim_{x\to\infty} \frac 1{\log\log x} \sum_{p \le x} \frac{1_S(p)}p \]
provided the limit exists.
}

\newcommand{\definitionEk}{%
To avoid excessive nested parentheses, we use the convention that $\Ex_n(f(m))^k$ is to be interpreted as $\Ex_n \bigl( (f(m))^k \bigr)$, the expectation of the $k$th power, as opposed to $\bigl( \Ex_n (f(m)) \bigr)^k$, the $k$th power of the expectation.
}

\newcommand{\definitionPn}{%
If $S(m)$ is any assertion about its argument~$m$, and~$n$ is any positive integer, we let $P_n(S(m)) = \frac 1n \#\{ m \in \N,\, m \le n \colon  S(m)$ is true$\}$ denote the probability that an integer~$m$ chosen uniformly from the set $\{1,\ldots,n\}$ satisfies $S(m)$.
}

\title{The universal profile of the invariant factors of~$(\Z/m\Z)^\times$}
\author{Greg Martin}
\address{University of British Columbia\\ Department of Mathematics \\ Room 121\\ 1984 Mathematics Road\\Vancouver, BC Canada V6T 1Z2}
\email{gerg@math.ubc.ca}
\author{Reginald~M.~Simpson}
\address{University of British Columbia\\ Department of Mathematics \\ Room 121\\ 1984 Mathematics Road\\Vancouver, BC Canada V6T 1Z2}
\email{regsim@math.ubc.ca}

\begin{document}

\maketitle
\tableofcontents

\section{Introduction}

The {\em multiplicative group} of the positive integer~$m$ is the group of units $(\Z/m\Z)^\times$ in the ring $\Z/m\Z$, which encodes the structure of multiplication among the invertible residue classes modulo~$m$. Many classical arithmetic functions can be defined in terms of this group, such as Euler's totient function $\varphi(m)$ (the order of the multiplicative group of~$m$), the Carmichael lambda-function $\lambda(m)$ (the maximum order of the elements of the multiplicative group of~$m$), and the Dirichlet characters $\chi(m)$ (arithmetic lifts of the irreducible representations of the multiplicative group of~$m$).

As a finite abelian group, the multiplicative group of~$m$ can be uniquely decomposed into a direct product
of cyclic groups $\Z_{\lambda_1} \times \cdots \times \Z_{\lambda_K}$ for some integer~$K$, with orders $\lambda_1,\ldots,\lambda_K$ such that $\lambda_1 \mid \lambda_2 \mid \cdots \mid \lambda_K$. This representation, which is called the {\em invariant factor decomposition} of the multiplicative group, is also connected to well-known arithmetic functions. In particular, the order~$\lambda_K$ of the largest cyclic group in the invariant factor decomposition of $(\Z/m\Z)^\times$ is precisely $\lambda(m)$. Furthermore, it is not hard to show (see Lemma~\ref{lemma:total_invariant_factors} below) that the length $K$ of the invariant factor decomposition of $(\Z/m\Z)^\times$ is essentially equal to $\omega(m)$, the number of distinct prime factors of~$m$. The limiting distribution of various statistics of this family of multiplicative groups is therefore closely connected to classical number theory results such as the Hardy--Ramanujan theorem and the Erd\H os--Kac theorem.

In this article, we will demonstrate that the invariant factors of $(\Z/m\Z)^\times$ are strongly restricted in their possible orders, at least for almost all numbers~$m$ (that is, on a subset of~$\N$ of natural density~$1$). These special orders are defined by a specific sequence of sets $(\sif_i)$, each one having either one or two elements, which we call the \emph{sequence of universal factor order sets}. The first ten such sets are
\begin{multline} \label{first ten UFOs}
\firsttenUFOsets
\end{multline}
When we define this sequence precisely (Definition~\ref{defn:standard_invariant_sequence} below), we will see that the elements of each universal factor order set always divide the elements of its successor. In particular, we note that the first even number not found in any $\sif_i$ is~$8$.

Chang and the first author~\cite{SIFMG} showed that for almost all~$m$, the multiplicative group $(\Z/m\Z)^\times$ has $\Z_2$ as an invariant factor (indeed, they gave an asymptotic formula for the counting function of the exceptions). Our first theorem is a vast generalization of this result, precisely classifying the orders of cyclic groups appearing in almost all multiplicative groups.

\begin{theorem}
\label{theorem:order_restriction_finite}
\theoremoverall
\end{theorem}

For a concrete implication of this theorem, we see that the set of positive integers whose multiplicative group has at least one invariant factor $\Z_8$ has density~$0$, while the set of positive integers whose multiplicative group has at least one invariant factor $\Z_{720720}$ (indeed, for which $\Z_{720720}$ is the tenth-smallest distinct invariant factor order!)~has density~$1$. We also see that almost all integers have one of $\Z_{360}$ and $\Z_{840}$ as an invariant factor (and no integer has both, since neither of $360$ and $840$ divides the other); we will see later in Theorem~\ref{theorem:dist_2} that each of these two possibilities occurs on a set of density~$\frac12$.

Theorem~\ref{theorem:order_restriction_finite} already reveals a startling uniformity in the invariant factor decompositions within the family of multiplicative groups; but much more can be said. For instance, beyond asking simply which invariant factors arise, we can also ask about the multiplicities of these invariant factors:

\begin{defn}
\label{defn:group_counting_functions_inv}
\definitioncountingfunctionsinv
\end{defn}

We can compute statistics of the distribution of these functions $\inv(m;d)$. For example, we will show in Theorem~\ref{theorem:ex_general} below that the average order (or expected value) of the number of invariant factors $\Z_{720720}$ is
\begin{equation} \label{720720 example}
\frac 1x \sum_{m \le x} \inv(m;720720) \sim \frac 1{48}\log\log x.
\end{equation}
By comparison, the average order of the total number of invariant factors is the average order of $\omega(m)$ (see Lemma~\ref{lemma:total_invariant_factors} below), which is $\log\log x$.
Thus on average, a proportion $\frac1{48}$ of the invariant factor decomposition of a multiplicative group consists of copies of $\Z_{720720}$. Section~\ref{expectation intro section} contains many such expected-value results for the quantities $\inv(m;d)$.

Remarkably, we can go even farther than examining the average orders of these counting functions $\inv(m;d)$: we can establish their precise limiting distributions. In Section~\ref{distribution intro section}, we show that normalized versions of the $\inv(m;d)$ have limiting distributions that follow Erd\H{o}s--Kac-type laws---but with the very unusual property that while some universal factor orders~$d$ yield normal distributions, others yield skew-normal or even more complicated distributions.

\subsection{Definitions of important sequences}

We need several definitions and pieces of notation to state these stronger theorems precisely. Before we can define the full sequence $\{\sif_i\}$ of universal factor order sets, we must first define a particular ordering of the set of prime powers. (We remark that all of the definitions and notation in Section~1 are gathered in Appendix~\ref{appendix:notation} for the reader's ease of reference.)

\begin{defn}
\label{defn:phi_sequence}
\definitionphisequence
\end{defn}

(We remark that all of the notation in Section~1 is gathered in Appendix~\ref{appendix:notation} for the reader's ease of reference.)

We can now present the definition of the sequence $\{\sif_i\}$ whose first elements we saw in equation~\eqref{first ten UFOs}.
By definition, each set $\sif_i$ is either a singleton or a doubleton.

\begin{defn}
\label{defn:standard_invariant_sequence}
\definitionUFOsets
\end{defn}

\noindent
It turns out that doubletons are never adjacent in the sequence $(\sif_i)_{i=1}^\infty$ (see Lemma~\ref{lemma:equal_totient2} below); we will take this fact for granted in the introduction for the sake of exposition.

It will also be helpful to define two often-used sequences of constants derived from the total $\varphi$-sequence:

\begin{defn}
\label{defn:mu_sigma_sequences}
\definitionmusigma
We note that $\sum_{i=1}^\infty \mu_i$ is a telescoping series whose value is $1/\varphi(\tq_1) = 1/\varphi(2) = 1$.
\end{defn}

\subsection{Expectations of invariant factor counting functions}  \label{expectation intro section}

In this section we are interested in the average orders (or expectations) of the counting functions $\inv(m;d)$ defined in Definition~\ref{defn:group_counting_functions_inv}.
\Enotation
We can describe the average orders of $\inv(m;d)$ for every universal factor order~$d$ in one succinct statement (which includes equation~\eqref{720720 example} as a special case):

\begin{theorem}
\label{theorem:ex_general}
\theoremexgeneral
\end{theorem}

\noindent
(We remark that all of the theorems in Section~1 are gathered in Appendix~\ref{appendix:theorems} for the reader's ease of reference.)

In contrast, we can show that rare factor orders~$d$ indeed rarely occur as invariant factors of multiplicative groups. Theorem~\ref{theorem:order_restriction_finite} implies, for example, that almost no multiplicative groups include~$\Z_8$ as an invariant factor; we can improve this statement by showing that $\Ex_n(\inv(m;8)) = o(1)$, so that this family of groups contain vanishingly few invariant factors~$\Z_8$ on average as well. Even more strongly, we can show the same for all moments of $\inv(m;8)$, and even give a quantitative rate of decay of these expectations. The following theorem exhibits such a bound, indeed for all rare factor orders.

%
%
\begin{theorem}
\label{theorem:ex_rare}
\theoremexrare
\end{theorem}


We can compare these expectations to that of the total number $\inv(m)$ of invariant factors, which is asymptotically $\log\log m$ (see Lemma~\ref{lemma:total_invariant_factors} below). The result is that on average, a proportion $\mu_i$ of the invariant factors of multiplicative groups are equal to $\Z_d$ whenever $d\in\sif_i$. Moreover, since $\sum_{i=1}^\infty \mu_i = 1$, these represent almost all of the invariant factors present. Theorems~\ref{theorem:ex_general} and~\ref{theorem:ex_rare} therefore describe a ``limiting histogram'' of invariant factor orders among the family of multiplicative groups $(\Z/m\Z)^\times$. Figure~\ref{bar chart} displays this limiting histogram; we will see in the next section that it describes the limiting distribution of invariant factors in a far stronger sense than merely ``on average''.

\begin{figure}[ht]
\includegraphics[width=4in]{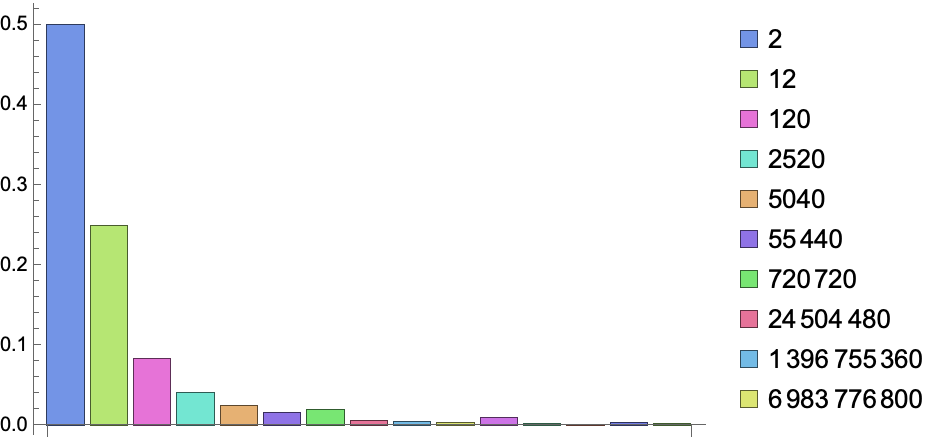}
\caption{The universal profile of invariant factors for the multiplicative groups $(\Z/m\Z)^\times$.}
\label{bar chart}
\end{figure}

We can see that some of the universal factor orders are missing from this histogram; indeed, it follows easily from Definitions~\ref{defn:standard_invariant_sequence} and~\ref{defn:mu_sigma_sequences} that $\mu_i$ equals~$0$ precisely when $\sif_i$ is a doubleton set. Thus Theorem~\ref{theorem:ex_general} is true but somewhat misleading in these cases: since $\mu_i = 0$, the ``error term'' is actually the main term (see Theorem~\ref{theorem:ex_2} below). In other cases, the error term can be improved significantly and second-order main terms appear.

We now describe these refinements to Theorem~\ref{theorem:ex_general}. The specifics turn out to depend on whether the adjacent sets $\sif_{i-1},\sif_i,\sif_{i+1}$ are singletons or doubletons, so there are several cases (enumerated in Proposition~\ref{prop:two-totient sequences} below) to consider. We will preface each case with a specific numerical example to help orient the reader. In our first example, the error term in the average order is improved dramatically:

\begin{ex}
The number of invariant factors of order $5040$ is on average
\[ \Ex_n(\inv(m;5040)) = \frac 1{40}\log\log n + O(1) . \]
\end{ex}

\noindent
(As a special case of the observations made before Figure~\ref{bar chart}, we see that on average, the cyclic group $\Z_{5040}$ comprises $2.5\%$ of the total number of invariant factors.) The general statement, which involves the notation $\mu_i$ from Definition~\ref{defn:mu_sigma_sequences}, concerns the case where all the sets $f_i$ involved are singletons.


\begin{theorem}
\label{theorem:ex_11}
\theoremexA
\end{theorem}

In our next example, a secondary main term arises:

\begin{ex}
The number of invariant factors of order $2520$ is on average
\begin{equation}  \label{2520 example}
\Ex_n(\inv(m;2520)) = \frac 1{24}\log\log n-\bigg(\sqrt{\frac5{36\pi}}+o(1)\bigg)(\log\log n)^{1/2} .
\end{equation}
\end{ex}

The general statement, of which this example is a special case, concerns the case where the set preceding $f_i$ is a doubleton but the two other relevant sets are singletons:
The main term still involves $\mu_i$ (as is consistent with Theorem~\ref{theorem:ex_general}), but the secondary term involves the notation $\sigma_i$ from Definition~\ref{defn:mu_sigma_sequences}.

\begin{theorem}
\label{theorem:ex_21}
\theoremexB
\end{theorem}

We encourage the reader to begin forming the intuition that the presence of the doubleton $\sif_{i-1}$ to the left of the singleton $\sif_i$ ``steals" a few of the invariant factors from the sequence of copies of~$\Z_d$. This intuition will be connected to Theorem~\ref{theorem:ex_2} below and will also be significant in the next section, where we discuss limiting distributions of these counting functions $\inv(m;d)$.


A similar outcome occurs when the relevant doubleton is on the other side:

\begin{ex}
The number of invariant factors of order $720720$ and $2$ are on average
\begin{align*}
\Ex_n(\inv(m;720720)) &= \frac 1{48}\log\log n-\bigg( \sqrt{\frac{15}{256\pi}}+o(1)\bigg)(\log\log n)^{1/2} \\
\Ex_n(\inv(m;2)) &= \frac 12\log\log n-\bigg(\frac1{2\sqrt\pi}+o(1)\bigg)(\log\log n)^{1/2}.
\end{align*}
\end{ex}
\noindent
We remark that the first of these two formulas is a refinement of the example~\eqref{720720 example} we have already seen, while the second formula implies that $\Z_2$ is the most common invariant factor (indeed comprising nearly half of the invariant factors) on average.

These two examples are special cases of the following general theorem, concerning the case where the set following $\sif_i$ is a doubleton but the two other relevant sets are singletons:

\begin{theorem}
\label{theorem:ex_12}
\theoremexE
\end{theorem}

\noindent
Again it is beneficial to think that the presence of the doubleton $\sif_{i+1}$ to the right of the singleton $\sif_i$ results in the loss of some of the copies of $\Z_d$.

This phenomenon is exacerbated if there are doubletons on both sides:

\begin{ex}
The number of invariant factors of order $120$ and $12$ are on average
\begin{align} \label{120 example}
\Ex_n(\inv(m;120)) &= \frac 1{12}\log\log n-\bigg(\sqrt{\frac 3{16\pi}}+\sqrt{\frac5{36\pi}}+o(1)\bigg)(\log\log n)^{1/2} \\
\Ex_n(\inv(m;12)) &= \frac 14\log\log n-\bigg(\frac1{2\sqrt\pi}+\sqrt{\frac 3{16\pi}}+o(1)\bigg)(\log\log n)^{1/2}.
\label{12 example}
\end{align}
\end{ex}

\noindent
These two examples are special cases of the following result, where (informally) doubletons on both sides of $\sif_i$ each steal some of the copies of $\Z_d$ that would otherwise appear:

\begin{theorem}
\label{theorem:ex_22}
\theoremexD
\end{theorem}

\noindent
Taken together, Theorems~\ref{theorem:ex_11}--\ref{theorem:ex_22} cover all of the ``singleton cases'', namely the universal factor orders $d\in\sif_i$ where $\#\sif_i=1$.


The last refinement of Theorem~\ref{theorem:ex_general} concerns a situation where the expected value is truly of size $(\log\log n)^{1/2}$ rather than $\log\log n$:

\begin{ex}
The number of invariant factors of order $360$ and $840$ are on average
\begin{align} \label{360 840 example}
\Ex_n(\inv(m;360)) &= \bigg(\sqrt{\frac5{36\pi}}+o(1)\bigg)(\log\log n)^{1/2} \\
\Ex_n(\inv(m;840)) &= \bigg(\sqrt{\frac5{36\pi}}+o(1)\bigg)(\log\log n)^{1/2} . \notag
\end{align}
\end{ex}

\noindent
We can check that this example is consistent with Theorem~\ref{theorem:ex_general}: we have $\sif_6 = \{360,840\}$ and $\mu_6 = \varphi(\tq_6)^{-1} - \varphi(\tq_7)^{-1} = \frac 16 - \frac 16 = 0$. Therefore Theorem~\ref{theorem:ex_general}, while correctly asserting that $\Ex_n(\inv(m;360)) = O\bigl( (\log\log n)^{1/2} \bigr)$, hides a more precise average order in this case (and similarly for $d=840$). The fact that $\mu_6=0$ also explains why neither $d=360$ nor $d=840$ appears in the limiting histogram in Figure~\ref{bar chart}, which depicts proportions of invariant factors after division by $\log\log n$.

It is also worth remarking that the doubleton $\sif_6 = \{360,840\}$ really can be seen here to be stealing some of the invariant factors from its neighbours $\sif_7=\{2520\}$ and $\sif_5 = \{120\}$: note that the two expressions in equation~\eqref{360 840 example} are exactly the same size as negative secondary main terms in equations~\eqref{2520 example} and~\eqref{120 example}.

There is a corresponding general statement which covers all ``doubleton cases'', namely the universal factor orders $d\in\sif_i$ where $\#\sif_i=2$:

\begin{theorem}
\label{theorem:ex_2}
\theoremexC
\end{theorem}

\noindent
As mentioned earlier, it follows easily from Definitions~\ref{defn:standard_invariant_sequence} and~\ref{defn:mu_sigma_sequences} that $\#\sif_i = 2$ if and only if $\mu_i = 0$; thus Theorem~\ref{theorem:ex_2} is also consistent with, but more precise than, Theorem~\ref{theorem:ex_general} in this case.

With the intuition we have been forming, it is beneficial to recast the sequence~\eqref{first ten UFOs} of universal factor order sets as a flowchart where the singletons and doubletons play different roles, as in Figure~\ref{flowchart figure}. The invariant factor decomposition of a typical multiplicative group $(\Z/m\Z)^\times$ begins with about $\frac12\log\log m$ copies of~$\Z_2$ (as Figure~\ref{bar chart} indicates); however, before continuing with about $\frac14\log\log m$ copies of~$\Z_{12}$, there is a smaller transition zone (on the $(\log\log m)^{1/2}$ scale) consisting of copies of either~$\Z_4$ or~$\Z_6$. The same phenomenon occurs every time there is a doubleton set $\sif_i$ between two singleton sets. When there is no such doubleton set, the transition is immediate---for example, the roughly $\frac1{24}\log\log m$ copies of~$\Z_{2520}$ are followed immediately by roughly $\frac1{40}\log\log m$ copies of~$\Z_{5040}$ for this typical multiplicative group $(\Z/m\Z)^\times$.

\begin{figure}[ht]
\tikzstyle{singleton} = [rectangle, minimum width=10mm, minimum height=18pt, text centered, draw=black, fill=green!20]
\tikzstyle{doubleton} = [rectangle, minimum width=5mm, minimum height=12pt, text centered, draw=black, fill=blue!20, font=\scriptsize]
  \begin{tikzpicture}
    \node (2) [singleton] at (0.2,0) {2};
    \node (4) [doubleton] at (1.3,0.3) {4};
    \node (6) [doubleton] at (1.3,-0.3) {6};
    \node (12) [singleton] at (2.43,0) {12};
    \node (24) [doubleton] at (3.6,0.3) {24};
    \node (60) [doubleton] at (3.6,-0.3) {60};
    \node (120) [singleton] at (4.77,0) {120};
    \node (360) [doubleton] at (6.0,0.3) {360};
    \node (840) [doubleton] at (6.0,-0.3) {840};
    \node (2520) [singleton] at (7.25,0) {2520};
    \node (5040) [singleton] at (8.7,0) {5040};
    \node (55440) [singleton] at (10.25,0) {55440};
    \node (720720) [singleton] at (12,0) {720720};
    \node (and) at (13.2,0.3) {};
    \node (so) at (13.5,0) {$\cdots$};
    \node (on) at (13.2,-0.3) {};
    
    \draw[->] (2) -- (4); \draw[->] (2) -- (6);
    \draw[->] (4) -- (12); \draw[->] (6) -- (12);
    \draw[->] (12) -- (24); \draw[->] (12) -- (60);
    \draw[->] (24) -- (120); \draw[->] (60) -- (120);
    \draw[->] (120) -- (360); \draw[->] (120) -- (840);
    \draw[->] (360) -- (2520); \draw[->] (840) -- (2520);
    \draw[->] (2520) -- (5040);
    \draw[->] (5040) -- (55440);
    \draw[->] (55440) -- (720720);
    \draw[->] (720720) -- (and); \draw[->] (720720) -- (on);
  \end{tikzpicture}
  \caption{The universal factor orders in the doubleton sets~$\sif_i$ serve as transitions between the more numerous universal factor orders in the singleton~sets.}
  \label{flowchart figure}
\end{figure}
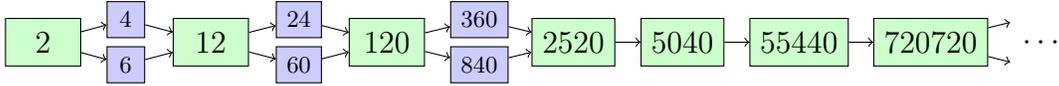

While all of the theorems in this section have concerned the {\em average} behaviour of the invariant factor order counting functions $\inv(m;d)$, we have begun describing these phenomena as the {\em typical} behaviour. This shift in language is intentional: in the next section, we will see that the functions $\inv(m;d)$ have limiting distribution laws for every universal factor order~$d$. Strengthening this section's results in that way is quite analogous to strengthening the average order of the function $\omega(m)$ to the Erd\H os--Kac theorem which gives its normalized limiting distribution. We will indeed prove theorems of Erd\H os--Kac type for each of these counting functions $\inv(m;d)$, with one unexpected development: the limiting distributions that arise are not always normal distributions---they are sometimes skew-normal distributions or even more complicated~ones.

%

\subsection{Limiting distributions of invariant factor counting functions}  \label{distribution intro section}

In this section we exhibit all of the limiting distributions of the counting functions $\inv(m;d)$ where $d\in\siF$ is a universal factor order (as well as the trivial limiting distributions where~$d\notin\siF$ is a rare factor order). As with the theorems above concerning average orders, the theorems in this section depend on whether the adjacent sets $\sif_{i-1},\sif_i,\sif_{i+1}$ are singletons or doubletons. Fortunately, the shape of the limiting distribution depends only upon the sizes of these adjacent sets, except in the special cases $i=1$ (so that $d=2$) and $i=3$ (so that $d=12$). We will again preface each case with a specific numerical example for the reader's benefit.


For these theorems, some consistent notation for common probability distributions will be necessary.

\begin{defn} \label{defn:normal}
We use the usual error function $\erf(x) = \frac2{\sqrt\pi} \int_0^x e^{-t^2} \,dt$ and define its relative $\eta(x) = \erf(ix) = i\erfi(x)$. We then define the standard normal probability density function
\[
\phi(x) = \frac{e^{-{x^2}/2}}{\sqrt{2\pi}}
\]
and the standard normal cumulative distribution function
\[
\Phi(x) = \int_{-\infty}^x \phi(t) \, dt = \frac12 \biggl( \erf\biggl( \frac x{\sqrt2}  \biggr) + 1 \biggr).
\]
The standard normal distribution has characteristic function $e^{-t^2/2}$.
\end{defn}

\begin{remark}\
\begin{enumerate}
\item To clarify a possible notational confusion: in this paper $\varphi$ always denotes the Euler phi-function and $\varphi$-sequences,
while $\phi$ and $\Phi$ always denote the density function and the cumulative distribution function, respectively, of the standard normal distribution (or, when given a second argument, of the skew-normal distribution in Definition~\ref{defn:skew-normal} below).
\item Typically the characteristic function of a generic random variable~$X$ is denoted by~$\varphi_X$; however, to avoid further overloading this Greek letter, in this paper we denote such characteristic functions by~$\chi_X$.
\item All the probability-related functions in this section are collected in one place in Appendix~\ref{appendix:notation} for ease of reference.
\end{enumerate}
\end{remark}

We will use notation for probabilities themselves when sampling uniformly from an initial segment of the positive integers.
\definitionPn

With this notation in hand, we can now give a much stronger version of the example before Theorem~\ref{theorem:ex_11}:

\begin{ex}
The number of invariant factors of order $5040$ follows an Erd\H os--Kac law with mean $\frac1{40}\log\log n$ and variance $\frac15\log\log n$. More precisely, for every $x\in\R$,
\[
\lim_{n\to\infty} P_n\bigg(\frac{\inv(m;5040) - \frac 1{40}\log\log n}{(\frac 15\log\log n)^{1/2}} \le x\bigg) = \Phi(x).
\]
\end{ex}

\noindent
The general statement, which strengthens Theorem~\ref{theorem:ex_11}, involves the notation~$\mu_i$ and~$\sigma_i$ from Definition~\ref{defn:mu_sigma_sequences} and concerns the case where all the sets~$f_i$ involved are singletons.

\begin{theorem}
\label{theorem:dist_11}
\theoremdistA
\end{theorem}

\noindent
Under the assumptions of Theorem~\ref{theorem:dist_11}, where all relevant sets~$\sif_i$ are singletons, it turns out that the invariant factor order counting function $\inv(m;d)$ can be closely approximated by a single additive function of~$m$ (see equation~\eqref{J as an additive function} below). Thus the Erd\H os--Kac statement can be proved in a standard way in this case.

However, when some relevant sets~$\sif_i$ are doubletons, it turns out that $\inv(m;d)$ has a more complicated behaviour that involves multiple additive functions---it might, for example, be approximately the minimum of two additive functions of~$m$. In such cases, while the limiting distributions of the individual additive functions follow Erd\H{o}s--Kac laws, the joint distribution of the additive functions on which $\inv(m;d)$ depends will be either a nonsingular two-variable normal distribution or a singular four-variable normal distribution. As a consequence, the resulting distribution for $\inv(m;d)$ will be more complicated than a normal distributon.

To describe one such result, we need the following definitions.

\begin{defn} \label{Owen T def}
Define Owen's $T$-function
\[
\OT(h,a) = \frac1{2\pi} \int_0^a e^{-(1+t^2)h^2/2} \frac{dt}{1+t^2}.
\]
If~$h$ and~$a$ are nonnegative, then $\OT(h,a)$ is the probability that $X\ge h$ and $0\le Y\le aX$, where~$X$ and~$Y$ are independent standard normal variables (see~\cite{owen_T}). We note the obvious symmetries $\OT(-h,a) = \OT(h,a)$ and $\OT(h,-a)=-\OT(h,a)$.
\end{defn}

\begin{defn} \label{defn:skew-normal}
Define the skew-normal distribution with parameter~$\alpha$ (see~\cite{skew_normal_class}) by its probability density function
\[
\phi(x;\alpha) = 2\phi(x)\Phi(\alpha x)
\]
and its cumulative distribution function
\[
\Phi(x;\alpha) = \Phi(x) - 2\OT(x,\alpha).
\]
This family of distributions includes the standard normal distribution since $\phi(x;0)=\phi(x)$ and $\Phi(x;0) = \Phi(x)$. The mean of the distribution corresponding to $\phi(x;\alpha)$ is $\alpha\sqrt{2/(\alpha^2+1)\pi}$, and its characteristic function is $e^{-t^2/2} \bigl( 1 + \eta\bigl( \alpha t/\sqrt{2(\alpha^2+1)} \bigr) \bigr)$.
\end{defn}

This skew-normal distribution appears in our next example:

\begin{ex}
The number of invariant factors of order $2520$ follows a ``skew-Erd\H os--Kac law'':
\[ \lim_{n\to\infty} P_n\bigg(\frac{\inv(m;2520) - \frac 1{24}\log\log n}{\frac 12(\log\log n)^{1/2}} \le x\bigg)
= \Phi\bigg(x;-\sqrt{\frac5{13}}\bigg) . \]
The characteristic function of this distribution is $e^{-t^2/2}\big(1-\eta(\frac{\sqrt5}6 t)\big)$.
\end{ex}

As with Theorem~\ref{theorem:ex_21}, this example represents those universal factor orders $d\in\sif_i$ where the set preceding $f_i$ is a doubleton but the two other relevant sets are singletons. The general theorem for this situation (and later ones) requires two more pieces of notation:

\begin{defn}  \label{defn:nu}
\definitionnu
Note that $\sigma_{i;1,1}$ is the same as $\sigma_i$ from Definition~\ref{defn:mu_sigma_sequences}.
\end{defn}

\begin{theorem}
\label{theorem:dist_21}
\theoremdistB
\end{theorem}

\noindent
While the normalization on the left-hand side uses the same mean $\mu_i\log\log n$ and variance $\sigma_i^2\log\log n$ as Theorem~\ref{theorem:dist_11}, the presence of the doubleton set $\sif_{i-1}$ to the left of $\sif_i$ ``steals'' some of the copies of $\Z_d$ away, resulting in a skew-Erd\H{o}s--Kac law with negative mean instead of the standard Erd\H{o}s--Kac law. Mathematically, the skew-normal distribution arises because $\inv(m;d)$ can be represented in terms of the limiting joint distribution of two additive functions, which in this case is a nonsingular bivariate normal distribution.

A similar phenomenon results from a doubleton set $\sif_{i+1}$ ``stealing'' some copies of $\Z_d$ from the right.

\begin{ex}
The numbers of invariant factors of order $720720$ and of order $2$ each follow skew-Erd\H os--Kac laws:
\begin{align*}
\lim_{n\to\infty} P_n\bigg(\frac{\inv(m;720720) - \frac 1{48}\log\log n}{(\frac{13}{96}\log\log n)^{1/2}} \le x\bigg) &= \Phi\bigg(x;-\sqrt{\frac{45}{163}}\bigg) \\
\lim_{n\to\infty} P_n\bigg(\frac{\inv(m;2) - \frac 12\log\log n}{(\frac 12\log\log n)^{1/2}} \le x\bigg) &= \Phi\bigg(x;-\frac1{\sqrt3}\bigg).
\end{align*}
The characteristic functions of these distributions are, respectively, $e^{-t^2/2}\bigl(1-\eta\bigl( \sqrt{\frac{45}{416}} t \bigr) \bigr)$ and $e^{-t^2/2}\big(1-\eta( \frac t{\sqrt8})\big)$.
\end{ex}

These examples demonstrate the following theorem, which is a refinement of Theorem~\ref{theorem:ex_12} and exhibits a skew-normal distribution analogous to Theorem~\ref{theorem:dist_21}.

\begin{theorem}
\label{theorem:dist_12}
\theoremdistE
\end{theorem}

We now consider the situation where a singleton~$\sif_i$ is flanked by two doubletons~$\sif_{i-1}$ and~$\sif_{i+1}$, each of which ``steal'' some copies of~$\Z_d$. As mentioned before, the corresponding counting function $\inv(m;d)$ is essentially a linear combination of four additive functions, which leads to examining a (singular) tetravariate normal distribution. We call the resulting limiting distribution a ``doubly skew-normal distribution'', although we note that it is different from the distributions described in~\cite{AGS} (their work combines two different skewing mechanisms, whereas our distribution is skewed twice in the same way). The precise distribution can be worked out, although we delay the complicated formulas for the sake of exposition.

\begin{ex}
The number of invariant factors of order $120$ has limiting cumulative distribution function
\[
\lim_{n\to\infty} P_n\bigg(\frac{\inv(m;120) - \frac 1{12}\log\log n}{(\frac 13\log\log n)^{1/2}} \le x\bigg) = \UF\bigg(4x;\frac7{\sqrt6},\frac3{\sqrt2},\sqrt{\frac{10}3}\bigg),
\]
where $\UF(x;\sigma_1,\sigma_2,\sigma_3)$ is given in Definition~\ref{defn:U_function} below. 
The characteristic function of this distribution is
\begin{align*}
e^{-t^2/2}\bigg(1 - \eta\bigg(\frac38 t\bigg) \bigg) \bigg( 1 - \eta\bigg( \sqrt{\frac5{48}} t\bigg) \bigg) .
\end{align*}
\end{ex}

Note that the characteristic function contains two factors involving the function~$\eta$, in contrast to the characteristic functions in Theorems~\ref{theorem:dist_21} and~\ref{theorem:dist_12} that have only one such factor. The general result involves all the constants from Definitions~\ref{defn:mu_sigma_sequences} and~\ref{defn:nu}.

\begin{theorem}
\label{theorem:dist_22}
\theoremdistD
\end{theorem}

Note that the case $i=3$ is explicitly excluded from this theorem even though $\sif_2=\{4,6\}$, $\sif_3=\{12\}$, and $\sif_4=\{24,60\}$ have the appropriate cardinalities. The reason is that the two coincidences that create the doubletons are $\varphi(3)=\varphi(4)=2$ and $\varphi(5)=\varphi(8)=4$, and in this case there is the further coincidence that~$4$ and~$8$ are both powers of the same prime (a situation that never occurs again---see Proposition~\ref{prop:two-totient sequences}(d)). It turns out that these nearby powers of~$2$ result in an additional dependence among the four random variables that breaks a helpful symmetry. Therefore the case $d=12$ must be treated as its own special case.

\begin{theorem}
\label{theorem:dist_22alt}
\theoremdistDalt
\end{theorem}

\noindent
There is not much else that can be said about this distribution, beyond the fact that it exists and its expectation can be computed (as was done implicitly in equation~\eqref{12 example} in a previous example).

We have now described the limiting distributions of all of the most significant universal factor orders (the ones appearing in Figure~\ref{bar chart}), which correspond to singleton sets~$\sif_i$. The remaining universal factor orders, corresponding to doubleton sets~$\sif_i$, are the interstitial elements that occur when transitioning from one more significant universal factor order to another (refer to Figure~\ref{flowchart figure}). Unlike the more significant ones, whose distribution is on a $\log\log n$ scale, these interstitial distributions are on a $\sqrt{\log\log n}$ scale. The distribution is relatively simple but requires a bit more notation.

\begin{defn} \label{defn:right normal}
Using the notation in Definition~\ref{defn:normal}, we define a truncated normal probability density function
\[
\phi_+(x) = \begin{cases}
2\phi(x), &\text{if } x\ge0, \\
0, & \text{if } x<0
\end{cases}
\]
and the corresponding truncated normal cumulative distribution function
\[
\Phi_+(x) = \max\{\Phi(x)-\Phi(-x),0\} = \max\{\erf(\tfrac x{\sqrt2}),0\}.
\]
This truncated normal normal distribution, which is the same as the distribution of $|X|$ where $X$ is a standard normal random variable, has characteristic function $e^{-t^2/2}(1+\eta(\frac t{\sqrt 2}))$.
\end{defn}

\begin{ex}
The numbers of invariant factors of order $360$ and of order $840$ have the following distributions:
\begin{align*}
\lim_{n\to\infty} P_n\bigg(\frac{\inv(m;360)}{(\frac 5{18}\log\log n)^{1/2}} \le x\bigg)
&= \begin{cases} \Phi(x), & \text{if } x \ge 0, \\ 0 & \text{otherwise} \end{cases} \\
\lim_{n\to\infty} P_n\bigg(\frac{\inv(m;840)}{(\frac 5{18}\log\log n)^{1/2}} \le x\bigg)
&= \begin{cases} \Phi(x), & \text{if } x \ge 0, \\ 0 & \text{otherwise}. \end{cases}
\end{align*}
Furthermore, the sum of these counting functions has a truncated Erd\H os--Kac law:
\[ \lim_{n\to\infty} P_n\biggl(\frac{\inv(m;360)+\inv(m;840)}{(\frac 5{18}\log\log n)^{1/2}} \le x\biggr) = \Phi_+(x) . \]
\end{ex}

The corresponding general statement strengthens Theorem~\ref{theorem:ex_2}:

\begin{theorem}
\label{theorem:dist_2}
\theoremdistC
\end{theorem}

\noindent
Since $\Phi(0)=\frac12$, each $\inv(m;d)$ appearing in Theorem~\ref{theorem:dist_2} is positive on a set of integers~$m$ of density~$\frac12$. But note that these two sets are disjoint, since the two integers $d\in\sif_i$ do not divide each other. Therefore almost all multiplicative groups $(\Z/m\Z)^\times$ have one of these two integers~$d$ among their invariant factor orders. (We believe the number of the exceptions up to~$n$ is asymptotic to some constant depending on~$i$ times $n/\sqrt{\log\log n}$.) These counting functions $\inv(m;d)$ are modeled by the absolute difference of two additive functions of similar size, and it is a coin toss as to which one will be larger for particular values of~$m$.

The description of the limiting distribution laws for all the universal factor orders is now complete. By way of contrast, the rare factor orders are much simpler---we can easily strengthen Theorem~\ref{theorem:ex_rare} to the following distributional statement:

\begin{theorem}
\label{theorem:dist_rare}
\theoremdistrare
\end{theorem}

\noindent
For example, almost no multiplicative groups $(\Z/m\Z)^\times$ have an invariant factor with order~$8$.

\medskip
As a side remark, we mention another application of our techniques, namely the determination of the limiting distribution of the multiplicity counting function of any fixed nontrivial abelian subgroup~$G$ of $(\Z/m\Z)^\times$.

\begin{defn}
For any prime power~$q$,
let $\pri(G,q)$ be the number of elementary divisors of~$G$ of order~$q$, so that $G \cong \prod_q \Z_q^{\pri(G,q)}$ is an elementary decomposition.
Define
\[
\msg(m;G) = \max \{ a\in\N \colon G^a \le (\Z/m\Z)^\times \} .
\]
Finally, define $\mu_G = \min \{ 1/\varphi(q)\pri(G;q) \colon \pri(G;q) > 0 \}$.
\end{defn}

The following result, which is proved in the dissertation of second author~\cite[Chapter~7]{simpsonthesis}, identifies the limiting distribution of $\msg(m;G)$,

\begin{theorem}
Let $G$ be a nontrivial abelian finite group, and let $q_1,\ldots,q_N$ be the distinct elementary divisors of~$G$ that satisfy $\varphi(q_i)\pri(G;q_i) = \mu_G^{-1}$.
Then
\[
\lim_{n\to\infty} P_n\bigg(\frac{\msg(m;G)-\mu_G\log\log n}{(\mu_G\log\log n)^{1/2}} \le x\bigg) = \Mn_H(x) ,
\]
where $H(x_1,\ldots,x_N) = F\bigl( \pri(G,q_1)^{1/2}x_1,\ldots,\pri(G,q_N)^{1/2}x_N \bigr)$ 
and $F(x_1,\ldots,x_N)$ is the cumulative distribution function of a multivariate normal random variable~$X$ whose characteristic function is
\[
\chi_X(t_1,\ldots,t_N) = \exp\biggl( -\frac 12 \sum_{i=1}^N\sum_{j=1}^N \frac{(\varphi(q_i)\varphi(q_j))^{1/2}}{\varphi(\lcm[q_i,q_j])} t_it_j \biggr).
\]  
Here~$\Mn_H$ represents the minimum of the~$N$ random variables comprising the components of~$H$ (see Definition~\ref{defn:min_max_distributions} below).
\end{theorem}

\subsection{Overview of the proofs and outline of the paper}  \label{overview intro section}

We highlight some key concepts and quantities that will be helpful to keep in mind throughout this paper. We begin with the class of additive functions that have been mentioned as important to the proofs of the results in Section~\ref{distribution intro section}; these strongly additive functions turn out to be generalizations ot the standard distinct prime factor counting function~$\omega(m)$.

\begin{defn}
\label{defn:omega_functions}
\definitionomegafunctions
\end{defn}

\noindent
It follows easily, from classical techniques and the prime number theorem for arithmetic progressions, that $\omega(m;q,a)$ follows an Erd\H os--Kac law with mean and variance asymptotically $\frac1{\phi(q)}\log\log m$ when $(a,q)=1$.

We now introduce some notation related to the structure of the finite abelian groups $(\Z/m\Z)^\times$. While the main results of this paper are related to the invatiant factor decomposition of such groups, it is helpful to approach the invariant factor decomposition via another canonical representation for finite abelian groups: every such group has a unique elementary divisor decomposition (or primary factor decomposition), that is, a representation as the direct product of cyclic groups of prime power order. Therefore we need notation that indicates how many times a particular prime power appears as an elementary divisor of a multiplicative group. It turns out that a more fundamental quantity groups together elementary divisors that are powers of the same prime:

\begin{defn}
\label{defn:group_counting_functions_prm}
\definitioncountingfunctionsprm
\end{defn}

\noindent
It turns out (see Lemma~\ref{lemma:prime_multiples} below) that $\prm(m;p^\alpha)$ is essentially equal to the quantity $\omega(m;p^\alpha,1)$ from Definition~\ref{defn:omega_functions}. Therefore $\prm(m;p^\alpha)$ also follows an Erd\H os--Kac law with asymptotic mean and variance $\frac1{\phi(p^\alpha)}\log\log m$.

In particular, for most integers~$m$ we would expect $\prm(m;p^\alpha)$ to be larger the smaller $\phi(p^\alpha)$ is and vice versa. This property of~$m$ is extremely important to our proofs throughout, and fortunately the property is indeed common; we are thus motivated to make the following definition.

\begin{defn}
\label{defn:y_typical}
\definitiontypical
\end{defn}

\noindent
In other words, a $y$-typical number is one where the counting functions $\prm(m;q)$ are ``in the correct order'' as~$q$ ranges over all prime powers with $\phi(q)\le y$. The following proposition, proved in
Section~\ref{subsection:y_typical_proof}, justifies why this property is called $y$-typicality:

\begin{prop}
\label{prop:y_typical_is_typical}
For all real numbers $y \ge 2$ and $r \ge 1$ and for all integers $n \in \N$,
\[
P_n(m \text{ is not $y$-typical}) \ll_r \frac{y^{4r}}{(\log\log n)^r}.
\]
In particular, for any function $y = y(n)$ satisfying $y = o\bigl( (\log\log n)^{1/4} \bigr)$, almost all integers~$m$ are $y$-typical, in the sense that
\[
\lim_{n\to\infty} P_n(m \text{ is $y(n)$-typical}) = 1.
\]
\end{prop}

We conclude this introduction by providing a guide to the remainder of the paper
and the proofs it contains. The majority of this paper is concerned with proving the results in Section~\ref{distribution intro section}, since Theorem~\ref{theorem:order_restriction_finite} and the theorems in Section~\ref{expectation intro section} follow easily from them.

The argument starts in Section~\ref{sec:invariant_to_additive} with some elementary number theory to convert the problem of the invariant factors of $(\Z/m\Z)^\times$ into a statement about the values of finitely many additive functions $\omega(m;q,1)$, provided that~$m$ is a $y$-typical integer for an appropriate~$y$.
In Section~\ref{sec:additive to probability} we apply the methods from Kubilius \cite{kubilius} on multivariate distributions
of additive functions to convert statements about additive functions into statements about probabilities.
In Section~\ref{sec:typicality}, we prove probabalistic statements about $y$-typicality, including Proposition~\ref{prop:y_typical_is_typical}, using a mixture of analytic and elementary arguments.

In Section~\ref{sec:distributions}, we combine all the above results to finally prove the limiting distribution theorems in Section~\ref{distribution intro section}, after which we show how they collectively imply Theorem~\ref{theorem:order_restriction_finite}.
The arguments that derive the expectation theorems in Section~\ref{expectation intro section} from the results in Section~\ref{distribution intro section} are given in Section~\ref{sec:implications}.

In Appendix~\ref{appendix:notation} we consolidate all the major notation of this paper (including all the notation that has appeared in this introduction). In Appendix~\ref{appendix:theorems} we gather all of the main theorems of this introduction together for ease of reference. Finally, Appendix~\ref{appendix:distributions} contains two lemmas related to sums of random variables and their characteristic functions and density functions.

\section{From the invariant factors of $(\Z/m\Z)^\times$ to additive functions}
\label{sec:invariant_to_additive}

In this section, we relate the invariant factor decomposition of $(\Z/m\Z)^\times$ to the values of the additive functions $\omega(m;q,1)$ from Definition~\ref{defn:omega_functions}, under the assumption that $m$ is $y$-typical in the sense of Definition~\ref{defn:y_typical}. While readers familiar with invariant factor decompositions will find all the material in this section accessible (modulo the cumbersome notation), our later work will benefit from this careful translation from group theory to number theory.

\subsection{Relation between invariant factors and additive functions}

We begin with lemmas that hold for all integers $m\ge3$, without requiring any $y$-typicality assumption. Recall two pieces of notation from the introduction (Definitions~\ref{defn:omega_functions} and~\ref{defn:group_counting_functions_prm}), namely $\omega(m;S) = \#\{ p\mid m \colon p\in S \}$ for any set~$S$ of primes, and $\prm(m;p^\alpha)$ for the number of elementary divisors of the multiplicative group $(\Z/m\Z)^\times$ that are of the form $\Z_{p^\beta}$ where $\beta\ge\alpha$. We will use a variant of this latter notation as well:

\begin{defn}
\label{defn:group_counting_functions_pri}
\definitioncountingfunctionspri
\end{defn}

\begin{lemma}
\label{lemma:prime_decomposition}
We have $\pri(m;p^\alpha) = \omega \bigl( m;\{q \text{ prime} \colon p^\alpha \xmid (q-1)\} \bigr) + E(m;p^\alpha)$
where
\[ E(m;p^\alpha) = \begin{cases}
2,	& \text{if $p^\alpha = 2$ and $8 \xmid m$}, \\
1,	& \text{if $p^\alpha = 2$ and $(4 \xmid m$ or $16 \mid m)$}, \\
1,	& \text{if $p = 2$ and $\alpha \ge 2$ and $2^{\alpha+2} \xmid m$}, \\
1,	& \text{if $p \ne 2$ and $p^{\alpha+1} \xmid m$}, \\
0,	& \text{otherwise}.
\end{cases} \]
\end{lemma}

\begin{proof}
First suppose that $4\nmid m$. By the Chinese remainder theorem and the fact that multiplicative groups to odd prime power moduli (wihch have primitive roots) are cyclic, the elementary divisor decomposition of the multiplicative group $(\Z/m\Z)^\times$ is
\begin{align*}
(\Z/m\Z)^\times \cong \prod_{p^\alpha \xmid m} (\Z/p^\alpha\Z)^\times &\cong \prod_{\substack{p^\alpha \xmid m \\ p \ne 2}} ( \Z_{p-1} \times \Z_{p^{\alpha-1}} ) \cong \prod_{p \mid m} \prod_{q^\alpha \xmid p-1} \Z_{q^\alpha} \times \prod_{\substack{p^\alpha \xmid m \\ p \ne 2}} \Z_{p^{\alpha-1}} .
\end{align*}
For every prime power $p^\alpha$, the first factor contributes $\omega \bigl( m;\{q \text{ prime} \colon p^\alpha \xmid (q-1)\} \bigr)$ to the quantity $\pri(m;p^\alpha)$. If~$p$ is odd, the second factor contributes~$1$ if $p^{\alpha+1} \xmid m$ and~$0$ otherwise, which is precisely $E(m;p^\alpha)$ in this case; while if~$p=2$ then the second factor contributes~$0$ which also equals $E(m;2^\alpha)$ since $4\nmid m$.

When $4\mid m$, there is one additional factor $(\Z/2^\alpha\Z)^\times \cong \Z_{2^{\alpha-2}} \times \Z_2$ which does not affect $\pri(m;p^\alpha)$ for odd primes~$p$; and one can check that in every case the definition of $E(m;2^\alpha)$ accounts correctly for this additional factor.
\end{proof}

From this formula for $\pri(m;p^\alpha)$ we can deduce a similar formula for $\prm(m;p^\alpha)$, which will facilitate our computation of the invariant factors themselves.

\begin{lemma}
\label{lemma:prime_multiples}
We have $\prm(m;p^\alpha) = \omega(m; p^\alpha, 1 ) + \bar E(m;p^\alpha)$
where
\[ \bar E(m;p^\alpha) = \begin{cases}
2,	& \text{if $p^\alpha = 2$ and $8 \mid m$}, \\
1,	& \text{if $p^\alpha = 2$ and $4 \xmid m$}, \\
1,	& \text{if $p = 2$ and $\alpha \ge 2$ and $2^{\alpha+2} \mid m$}, \\
1,	& \text{if $p \ne 2$ and $p^{\alpha+1} \mid m$}, \\
0,	& \text{otherwise}.
\end{cases} \]
In particular, $\prm(m;p^\alpha) \le \prm(m;2)$ for all prime powers $p^\alpha$.
\end{lemma}

\begin{proof}
By Definition~\ref{defn:group_counting_functions_pri} and Lemma~\ref{lemma:prime_decomposition},
\begin{align*}
\prm(m;p^\alpha) 
= \sum_{\beta=\alpha}^\infty \pri(m;p^\beta) &= \sum_{\beta=\alpha}^\infty \omega\bigl( m; \{ q \text{ prime} \colon p^\beta \xmid (q-1) \} \bigr) + \sum_{\beta=\alpha}^\infty E(m;p^\beta) \\
&= \omega\bigl( m; \{ q \text{ prime} \colon p^\alpha \mid (q-1) \} \bigr) + \sum_{\beta=\alpha}^\infty E(m;p^\beta).
\end{align*}
(Note that all the above sums are convergent since $\pri(m;p^\beta)$ and $E(m;p^\beta)$ have only finitely many nonzero terms for any fixed~$m$). Since $\omega\bigl( m; \{ q \text{ prime} \colon p^\alpha \mid (q-1) \} \bigr) = \omega(m;q,1)$, we must simply check that the last sum equals $\bar E(m;p^\alpha)$ as defined in the statement of the lemma, which is straightforward. The last assertion follows from the individual inequalities $\omega(m; p^\alpha,1) \le \omega(m;2,1)$ and $\bar E(m;p^\alpha) \le \bar E(m;2)$ which are also easy to verify.
\end{proof}

The above statement is equivalent to the result given in~\cite[Definitions~2.7 and~2.8 and Lemma~2.10]{DNSMG}, but stated more conveniently for our purposes. In~\cite{DNSMG} it was shown that (in our notation) $\prod_{k=1}^\infty \Z_{p^k}^{\pri(m;p^k)}$ is the $p$-Sylow subgroup of $(\Z/m\Z)^\times$, and that the partition of $\sum_{k=1}^\infty k \pri(m;p^k)$ that characterizes this finite $p$-group is the partition conjugate to $(\prm(m;p^1),\prm(m;p^2),\ldots)$.

The invariant factor orders of a finite abelian group can be obtained directly from the elementary divisors, in a way that is straightforward in practice but somewhat hard to notate as a general statement. The following example should illuminate the relevant ideas before we craft a general proof.

\begin{ex}
Consider $m=1{,}616{,}615 = 5 \cdot 7 \cdot 11 \cdot 13 \cdot 17 \cdot 19$. Using the Chinese remainder theorem and the fact that odd prime powers possess primitive roots, we see that
\[ \arraycolsep=1pt\def\arraystretch{1.2}
\begin{array}{cccccccccccccc}
(\Z/1616615\Z)^\times & \cong & (\Z/5\Z)^\times & \times & (\Z/7\Z)^\times & \times & (\Z/11\Z)^\times & \times & (\Z/13\Z)^\times & \times & (\Z/17\Z)^\times & \times & (\Z/19\Z)^\times & \\
& \cong & \Z_4 & \times & \Z_6 & \times & \Z_{10} & \times & \Z_{12} & \times & \Z_{16} & \times & \Z_{18} & \\

& \cong & \Z_4 & \times & \Z_2 & \times & \Z_2 & \times & \Z_4 & \times & \Z_{16} & \times & \Z_2 & \\
& & & \times & \Z_3 & & & \times & \Z_3 & & & \times & \Z_9 & \\
& & & & & \times & \Z_5 & & & & & & & \\
\end{array}
\]
where in the last step we again used the Chinese remainder theorem to rewrite each cyclic group on the second line as a product of the elementary divisors in the column below it. These elementary divisors have been grouped prime by prime, and reordering each such row so that the powers of each prime are in increasing order results in
\[ \arraycolsep=3.5pt\def\arraystretch{1.2}
\begin{array}{ccccccccccccc}
(\Z/1616615\Z)^\times & \cong & \Z_2 & \times & \Z_2 & \times & \Z_2 & \times & \Z_4 & \times & \Z_4 & \times & \Z_{16} \\
& & & & & & & \times & \Z_3 & \times & \Z_3 & \times & \Z_9 \\
& & & & & & & & & & & \times & \Z_5 \\
& \cong & \Z_2 & \times & \Z_2 & \times & \Z_2 & \times & \Z_{12} & \times & \Z_{12} & \times & \Z_{720}
\end{array}
\]
after combining the groups in each column (another application of the Chinese remainder theorem). This last expression is the invariant factor decomposition of $(\Z/1616615\Z)^\times$.
\end{ex}

More generally, we can record the outcome of this method in terms of the following lemma:

\begin{prop}
\label{prop:elementary_divisors_to_invariant_factors}
Fix an integer $m\ge3$.
Let $q_1,\ldots,q_M$ be a sequence of distinct prime powers containing all 
the distinct elementary divisors of $(\Z/m\Z)^\times$,
ordered so that $\prm(m;q_i)$ is decreasing in~$i$, and set $d_i = \lcm[q_1,\dots,q_i]$. Then
\begin{equation} \label{G invariant}
(\Z/m\Z)^\times \cong \bigg(\prod_{i=1}^{M-1} \Z_{d_i}^{\prm(m;q_i)-\prm(m;q_{i+1})}\bigg) \times \Z_{d_M}^{\prm(m;q_M)}
\end{equation}
is the invariant factor decomposition of $(\Z/m\Z)^\times$, after removing all factors in the product where the exponent equals~$0$.
\end{prop}

\begin{remark}
In light of the last assertion of Lemma~\ref{lemma:prime_multiples}, we may always take $q_1=2$ if we wish; in particular, doing so proves that the~$d_i$, and hence the orders of the invariant factors of $(\Z/m\Z)^\times$, are always even.
\end{remark}

\begin{proof}
Let $G$ be the group defined on the right-hand side of equation~\eqref{G invariant}. Define~$\cE$ to be the set of all orders of elementary divisors of $(\Z/m\Z)^\times$. Since $(\Z/m\Z)^\times \cong \prod_{p^\alpha\in\cE} \Z_{p^\alpha}^{\pri(m;p^\alpha)}$ by the definition of $\pri(m;p^\alpha)$, it suffices to prove that $G\cong \prod_{p^\alpha\in\cE} \Z_{p^\alpha}^{\pri(m;p^\alpha)}$.

While $q_i$ and $d_i$ are defined only when $1\le i\le M$, we make the convention that the expression $\prm(m;q_i)$ equals~$0$ when $i>M$. Define $\cD = \{ i \in \N \colon \prm(m;q_i) > \prm(m;q_{i+1}) \}$, so that~$M$ is the largest element of~$\cD$. With this notation, we may rewrite
\begin{align*}
G\cong \bigg( \prod_{i=1}^{M-1} \Z_{d_i}^{\prm(m;q_i)-\prm(m;q_{i+1})} \bigg) \times \Z_{d_M}^{\prm(m;q_M)} &\cong \prod_{i\in\N} \Z_{d_i}^{\prm(m;q_i)-\prm(m;q_{i+1})} \\
&\cong \prod_{i\in \cD} \Z_{d_i}^{\prm(m;q_i)-\prm(m;q_{i+1})} \cong \prod_{i\in \cD} \prod_{p^\alpha\xmid d_i} \Z_{p^\alpha}^{\prm(m;q_i)-\prm(m;q_{i+1})}
\end{align*}
by the Chinese remainder theorem.

We claim that if $i\in \cD$, then $p^\alpha\mid d_i$ if and only if $\prm(m;p^\alpha) \ge \prm(m;q_i)$.
\begin{itemize}
\item First suppose that $p^\alpha\mid d_i$. By definition, $p^\alpha\mid q_j$ for some $j\in\{1,\dots,i\}$, and therefore $q_j = p^\beta$ for some $\beta\ge\alpha$. But then $\prm(m;p^\alpha) \ge \prm(m;p^\beta) = \prm(m;q_j) \ge \prm(m;q_i)$ as claimed.
\item Next suppose that $\prm(m;p^\alpha) \ge \prm(m;q_i)$. The assumption $i\in\cD$ implies that $\prm(m;p^\alpha) > \prm(m;q_{i+1}) \ge 0$, and so there exists $\beta\ge\alpha$ such that $p^\beta\in\cE$; choose the smallest such~$\beta$, so that $\prm(m;p^\alpha)=\prm(m;p^\beta)$. Since $p^\beta\in\cE$, we have $p^\beta = q_j$ for some~$j$, and then $\prm(m;q_j) = \prm(m;p^\beta) = \prm(m;p^\alpha) \ge \prm(m;q_i)$ implies that $j\le i$. We conclude that $p^\alpha\mid p^\beta = q_j \mid d_i$ as claimed.
\end{itemize}

Since $p^\alpha\xmid d_i$ if and only if $p^\alpha\mid d_i$ and $p^{\alpha+1}\nmid d_i$, it follows when $i\in \cD$ that $p^\alpha\xmid d_i$ if and only if $\prm(m;p^\alpha) \ge \prm(m;q_i) > \prm(m;p^{\alpha+1})$. In particular, since $\prm(m;p^\alpha) > \prm(m;p^{\alpha+1})$ is equivalent to $\pri(m;p^\alpha)>0$, we see that $p^\alpha\xmid d_i$ implies that ${p^\alpha\in\cE}$.

Therefore
\begin{align}
G\cong \prod_{i\in \cD} \prod_{p^\alpha\xmid d_i} \Z_{p^\alpha}^{\prm(m;q_i)-\prm(m;q_{i+1})} &\cong \prod_{i\in \cD} \prod_{\substack{p^\alpha\in\cE \\ \prm(m;p^\alpha) \ge \prm(m;q_i) > \prm(m;p^{\alpha+1})}} \Z_{p^\alpha}^{\prm(m;q_i)-\prm(m;q_{i+1})} \notag \\
&\cong \prod_{p^\alpha\in\cE} \prod_{\substack{i\in \cD \\ \prm(m;p^\alpha) \ge \prm(m;q_i) > \prm(m;p^{\alpha+1})}} \Z_{p^\alpha}^{\prm(m;q_i)-\prm(m;q_{i+1})} \notag \\
&\cong \prod_{p^\alpha\in\cE} \prod_{\substack{i\in\N \\ \prm(m;p^\alpha) \ge \prm(m;q_i) > \prm(m;p^{\alpha+1})}} \Z_{p^\alpha}^{\prm(m;q_i)-\prm(m;q_{i+1})}
\label{need to telescope}
\end{align}

For a fixed prime power~$p^\alpha\in\cE$, let~$k$ be the smallest integer such that $\prm(m;p^\alpha) \ge \prm(m;q_k)$ and let~$\ell$ be the largest integer such that $\prm(m;q_\ell) > \prm(m;p^{\alpha+1})$. We claim that in fact $\prm(m;p^\alpha) = \prm(m;q_k)$ and $\prm(m;p^{\alpha+1}) = \prm(m;q_{\ell+1})$.
\begin{itemize}
\item We first need to prove that $\prm(m;p^\alpha) \le \prm(m;q_k)$. Since $p^\alpha\in\cE$ we know that $p^\alpha = q_j$ for some~$j$. But if $\prm(m;q_j) > \prm(m;q_k)$, that would force $j<k$ and thus~$j$ would be an integer smaller than~$k$ for which $\prm(m;p^\alpha) \ge \prm(m;q_j)$, contradicting the choice of~$k$. Thus we indeed have $\prm(m;p^\alpha) = \prm(m;q_j) \le \prm(m;q_k)$.
\item We certainly have $\prm(m;q_{\ell+1}) \le \prm(m;p^{\alpha+1})$ by the choice of~$\ell$, so we next need to prove that $\prm(m;q_{\ell+1}) \ge \prm(m;p^{\alpha+1})$. If $\prm(m;p^{\alpha+1})=0$ then $\ell=M$ and so $\prm(m;q_{\ell+1}) = 0$ as well; so we may suppose that $\prm(m;p^{\alpha+1})>0$. Let $\beta$ be the largest integer such that $\prm(m;p^\beta) = \prm(m;p^{\alpha+1})$, which forces $p^\beta\in\cE$, and thus $p^\beta=q_j$ for some~$j$. But if $\prm(m;q_{\ell+1}) < \prm(m;q_j)$, that would force $\ell+1>j$ and hence $\ell\ge j$, and then~$\prm(m;q_\ell) \le \prm(m;q_j) = \prm(m;p^{\alpha+1})$, contradicting the choice of~$\ell$. Thus we indeed have $\prm(m;q_{\ell+1}) \ge \prm(m;q_j) = \prm(m;p^{\alpha+1})$.
\end{itemize}
From this discussion, we see that we have the telescoping sum
\begin{align*}
\sum_{\substack{i\in\N \\ \prm(m;p^\alpha) \ge \prm(m;q_i) > \prm(m;p^{\alpha+1})}} \bigl( \prm(m;q_i)-\prm(m;q_{i+1}) \bigr) &= \prm(m;q_k)-\prm(m;q_{\ell+1}) \\
&= \prm(m;p^\alpha) - \prm(m;p^{\alpha+1}) = \pri(m;p^\alpha),
\end{align*}
and so the last expression in equation~\eqref{need to telescope} can be simplified to
\begin{align*}
G\cong \prod_{p^\alpha\in\cE} \prod_{\substack{i\in\N \\ \prm(m;p^\alpha) \ge \prm(m;q_i) > \prm(m;p^{\alpha+1})}} \Z_{p^\alpha}^{\prm(m;q_i)-\prm(m;q_{i+1})} \cong \prod_{p^\alpha\in\cE} \Z_{p^\alpha}^{\pri(m;p^\alpha)}
\end{align*}
as required.
\end{proof}

Recall from Definition~\ref{defn:group_counting_functions_inv} that $\inv(m;d)$ denotes the multiplicity of $\Z_d$ in the invariant factor decomposition of $(\Z/m\Z)^\times$.

\begin{lemma}
\label{lemma:invariant_factors}
Fix an integer $m\ge3$.
Let $q_1,\ldots,q_N$ be a sequence of distinct prime powers,
ordered so that $\prm(m;q_i)$ is decreasing in~$i$, and set $d_i = \lcm[q_1,\dots,q_i]$. Suppose that $\prm(m;q_N) > \prm(m;q)$ for every prime power $q\notin\{q_1,\dots,q_N\}$. Then $\inv(m;d_N) \le \prm(m;q_N)$, and
$\inv(m;d_i) = \prm(m;q_i)-\prm(m;q_{i+1})$ for each $1 \le i \le N-1$.

Furthermore, the sequence $d_1,\ldots,d_N$ contains every invariant factor of $(\Z/m\Z)^\times$ less than or equal to $d_N$.
If $d$ is an invariant factor of $(\Z/m\Z)^\times$ exceeding $d_N$, then $d_N \mid d$.
\end{lemma}


\begin{proof}
As before, define~$\cE$ to be the set of all orders of elementary divisors of $(\Z/m\Z)^\times$, and set $\cE' = \cE \setminus \{q_1,\ldots,q_N\}$ and $M=N+\#\cE'$. Write $\cE' = \{q_{N+1},\dots,q_N\}$ so that $\prm(m;q_i)$ is decreasing in~$i$ for all $1\le i\le M$ (this is possible since $\prm(m;q_N) > \prm(m;q)$ for all $q\in\cE'$ by hypothesis), and extend the definition $d_i = \lcm[q_1,\dots,q_i]$ to all $i\le M$. We may now apply Proposition~\ref{prop:elementary_divisors_to_invariant_factors}, obtaining
\[
(\Z/m\Z)^\times \cong \bigg(\prod_{i=1}^M \Z_{d_i}^{\prm(m;q_i)-\prm(m;q_{i+1})}\bigg) \times \Z_{d_N}^{\prm(m;q_N)}
\]
as the invariant factor decomposition of $(\Z/m\Z)^\times$ (after removing all factors in the product where the exponent equals~$0$). It follows immediately that $\inv(m;d_i) = \prm(m;q_i)-\prm(m;q_i+1)$ for $1 \le i \le N-1$; it also follows that $\inv(m;d_N) = \prm(m;q_N)-\prm(m;q_{N+1}) \le \prm(m;q_N)$ if $N<M$, while $\inv(m;d_N) = \prm(m;q_N)$ if $N=M$. The last assertions follow from the fact that $d_1\mid\cdots\mid d_M$ and every invariant factor of $(\Z/m\Z)^\times$ is one of the~$d_j$.
\end{proof}

A consequence of this lemma is a clean formula for the total number of invariant factors of $(\Z/m\Z)^\times$:

\begin{lemma}
\label{lemma:total_invariant_factors}
For any integer $m\ge3$,
\[ \inv(m) = \begin{cases}
\omega(m) + 1, & \text{if } 8 \mid m, \\
\omega(m) - 1, & \text{if } 2 \xmid m, \\
\omega(m), & \text{otherwise} .
\end{cases} \]
\end{lemma}

\noindent In particular, this formula justifies the remark made after Theorem~\ref{theorem:ex_rare} that the average order of $\inv(m)$ is $\log\log m+O(1)$.

\begin{proof}
Set $q_1=2$ and let $q_2,\ldots,q_M$ be the orders of the elementary divisors of $(\Z/m\Z)^\times$ exceeding~$2$, ordered so that $\prm(m;q_i)$ is decreasing in~$i$ (which is consistent with $q_1=2$ thanks to the last assertion of Lemma~\ref{lemma:prime_multiples}). If we set $d_i = \lcm[q_1,\dots,q_i]$, then by Proposition~\ref{prop:elementary_divisors_to_invariant_factors},
\[
(\Z/m\Z)^\times \cong \bigg(\prod_{i=1}^{M-1} \Z_{d_i}^{\prm(m;q_i)-\prm(m;q_{i+1})}\bigg) \times \Z_{d_M}^{\prm(m;q_M)}
\]
is the invariant factor decomposition of $(\Z/m\Z)^\times$. In particular, the total number of invariant factors of $(\Z/m\Z)^\times$ is the telescoping sum
\[
\inv(m) = \sum_{i=1}^M \inv(m;d_i) = \sum_{i-1}^{M-1} \bigl( \prm(m;q_i)-\prm(m;q_{i+1}) \bigr) + \prm(m;q_M) = \prm(m;q_1) = \prm(m;2).
\]
By Lemma~\ref{lemma:prime_multiples},
\[
\prm(m;2) = \omega(m; 2, 1) + \bar E(m;2) = \begin{cases}
\omega(m) -1, &\text{if $2\mid m$}, \\
\omega(m), & \text{if $2\nmid m$}
\end{cases} \Biggr\} +
 \begin{cases}
2,	& \text{if $8 \mid m$}, \\
1,	& \text{if $4 \xmid m$}, \\
0,	& \text{otherwise},
\end{cases} \]
which is equivalent to the desired formula.
\end{proof}

\subsection{Properties of the invariant factor order sets}

Having now expressed (via Lemmas~\ref{lemma:prime_multiples} and~\ref{lemma:invariant_factors}) the counting functions for invariant factors in terms of additive functions, we now study the structure of the invariant factor order sets. As these sets are defined in terms of the total $\varphi$-sequence, this section opens with some lemmas concerning this sequence, before culminating in Proposition~\ref{prop:two-totient sequences} which classifies the short $\varphi$-sequences relevant to the various theorems in Sections~\ref{expectation intro section} and~\ref{distribution intro section}.

The first lemma appeared independently as~\cite[Lemma 4.4]{hannesson_martin}.

\begin{lemma}
\label{lemma:equal_totient}
If $q_1 < q_2$ are prime powers satisfying $\varphi(q_1) = \varphi(q_2)$, then~$q_1$ is a prime.
Consequently, at most two distinct prime powers can have the same totient values, and any two such prime powers are relatively prime.
\end{lemma}

\begin{proof}
We can immediately rule out $q_2$ being prime, for if it were then $\varphi(q_1) < q_1 \le q_2 - 1 = \varphi(q_2)$, contrary to assumption. Write $q_1 = p_1^{\alpha_1}$ and $q_2 = p_2^{\alpha_2}$ with $\alpha_2\ge2$, so that
\[
p_1^{\alpha_1-1}(p_1-1) = \varphi(q_1) = \varphi(q_2) = p_2^{\alpha_2-1}(p_2-1).
\]
We can also rule out $p_1=p_2$ since that would imply $\alpha_1=\alpha_2$ and hence $q_1=q_2$, contrary to assumption. In particular, $(p_1,p_2)=1$, and thus $p_2 \mid p_2^{\alpha_2-1} \mid p_1^{\alpha_1-1}(p_1-1)$ implies $p_2 \mid (p_1-1)$ and hence $p_2<p_1$. If we had $\alpha_1\ge2$, then the symmetric argument would show that $p_1 \mid (p_2-1)$ and thus $p_1<p_2$ also, a contradiction; therefore $\alpha_1=1$ and $q_1=p_1$ is prime, establishing the first assertion of the lemma. The second assertion follows from the observation that for any three prime powers $r_1, r_2 < r_3$
where $\varphi(r_1) = \varphi(r_2) = \varphi(r_3)$, the first assertion implies that~$r_1$ and~$r_2$ are primes with the same totient, and hence $r_1=r_2$.
\end{proof}

\begin{remark}
We have seen that pairs of prime powers with the same value under~$\varphi$ correspond precisely to solutions of the equations $p'-1=p^{\alpha-1}(p-1)$ in primes $p,p'$ for $\alpha\ge2$. Equivalently, they correspond precisely to simultaneous prime values of the polynomials $n,n^\alpha-n^{\alpha-1}+1$. When $\alpha\equiv2\mod6$ it is easy to check that $n^2-n+1$ divides $n^\alpha-n^{\alpha-1}+1$, and therefore $n^\alpha-n^{\alpha-1}+1$ is never prime for $\alpha=8,14,20,\dots$.

In all other cases, however, Ljunggren~\cite[Theorem 3]{ljunggren} proved that $n^\alpha-n^{\alpha-1}+1$ is irreducible, and we can check that $n,n^\alpha-n^{\alpha-1}+1$ is an ``admissible'' pair of polynomials (the residue class $n\equiv1\mod q$ is always feasible); therefore Schinzel's Hypothesis~H implies that there should be infinitely many solutions for every other $\alpha\ge2$. Contrary to early data, such solutions are rather rare: the primes~$p'$ appearing in such solutions have relative density~$0$, for the simple reason that at most $O(\sqrt x)$ integers up to~$x$ are values of any of the polynomials $n^\alpha-n^{\alpha-1}+1$. In other words, almost all of the $\sif_i$ are singleton sets.
\end{remark}

\begin{lemma}
\label{lemma:equal_totient2}
No two doubleton sets are adjacent in the sequence $(\sif_i)$ of universal factor order sets.
\end{lemma}

\begin{proof}
This is an immediate consequence of Definition~\ref{defn:standard_invariant_sequence} and Lemma~\ref{lemma:equal_totient}.
\end{proof}

The next lemma will be used to show that powers of the same prime are rarely close to each other in the total $\varphi$-sequence.

\begin{lemma} \label{lemma:Bertrand}
For any prime power $p^\alpha \ge 5$ and any integer $\beta>\alpha$, there exists a prime~$\ell$ satisfying $\varphi(p^\alpha) < \ell-1 < \varphi(p^\beta)$.
\end{lemma}

\begin{proof}
We use Bertrand's postulate in the form originally stated by Bertrand and proved by Chebyshev~\cite{Chebyshev}: for any integer $n\ge4$, there exists a prime in the interval $(n,2n-2)$. Taking $n=\varphi(p^\alpha)+1$ (which is at least~$5$ since the only integers~$m$ with $\varphi(m)<4$ are $m=1,2,3,4,6$), we see that there is a prime~$\ell$ satisfying $\varphi(p^\alpha)+1 < \ell < 2\varphi(p^\alpha)$. In particular, $\varphi(p^\alpha) < \ell-1 < p^{\beta-\alpha} \varphi(p^\alpha) = \varphi(p^\beta)$ as desired.
\end{proof}

We have seen that Theorems~\ref{theorem:ex_11}--\ref{theorem:ex_22} that the expectations of the counting functions $\inv(m;d)$ depend upon the size of three consecutive universal factor order sets $\sif_{i-1}, \sif_i, \sif_{i+1}$; the same is true of their normalized limiting distributions, as we saw in Theorems~\ref{theorem:dist_11} and~\ref{theorem:dist_21}--\ref{theorem:dist_22}. The next proposition will establish the fact that the cases in these quatrains of theorems exhaust all possibilities.

Recall the notion of a (finite) $\varphi$-sequence from Definition~\ref{defn:phi_sequence}.

\begin{defn}
\definitiontwototient
\end{defn}

\begin{prop} \label{prop:two-totient sequences}
Every two-totient $\varphi$-sequence is of one of four types:
\begin{enumerate}
\item $\{ \tq_i,\tq_{i+1} \}$ with $\varphi(\tq_i) < \varphi(\tq_{i+1})$.
  \item[] In this case, $\#\sif_{i-1} = \#\sif_i = \#\sif_{i+1} = 1$, and $\tq_i,\tq_{i+1}$ are relatively prime.
\item $\{ \tq_{i-1},\tq_i,\tq_{i+1} \}$ with $\varphi(\tq_{i-1}) = \varphi(\tq_i) < \varphi(\tq_{i+1})$. 
  \item[] In this case, $\#\sif_{i-1} = 2$ and $\#\sif_i = \#\sif_{i+1} = 1$, and $\tq_{i-1},\tq_i,\tq_{i+1}$ are pairwise relatively prime.
\item $\{ \tq_i,\tq_{i+1},\tq_{i+2} \}$ with $\varphi(\tq_i) < \varphi(\tq_{i+1}) = \varphi(\tq_{i+2})$. 
  \item[] In this case, $\#\sif_{i-1} = \#\sif_i = 1$ and $\#\sif_{i+1} = 2$, and $\tq_i,\tq_{i+1},\tq_{i+2}$ are pairwise relatively prime, except when $i=1$ in which case $\{ \tq_i,\tq_{i+1},\tq_{i+2} \} = \{ 2, 3, 4\}$.
\item $\{ \tq_{i-1},\tq_i,\tq_{i+1},\tq_{i+2} \}$ with $\varphi(\tq_{i-1}) = \varphi(\tq_i) < \varphi(\tq_{i+1}) = \varphi(\tq_{i+2})$. 
  \item[] In this case, $\#\sif_{i-1} = 2$ and $\#\sif_i = 1$ and $\#\sif_{i+1} = 2$, and $\tq_{i-1},\tq_i,\tq_{i+1},\tq_{i+2}$ are pairwise relatively prime, except when $i=3$ in which case $\{ \tq_{i-1},\tq_i,\tq_{i+1},\tq_{i+2} \} = \{ 3,4,5,8 \}$.
\end{enumerate}
\end{prop}

\begin{proof}
By Lemma~\ref{lemma:equal_totient}, each of the two totient values has either one or two preimages, which immediately yields the four cases described. The cardinalities of the sets~$\sif_j$ follow directly from Definition~\ref{defn:standard_invariant_sequence} and the definition of a two-totient $\varphi$-sequence: for example, in the first case we automatically have $\varphi(\tq_{i-1}) < \varphi(\tq_i)$ and $\varphi(\tq_{i+1}) < \varphi(\tq_{i+2})$.

The coprimality assertions follow from Lemma~\ref{lemma:Bertrand}: for example, in the first case $\varphi(\tq_i)$ and $\varphi(\tq_{i+1})$ are consecutive prime-power totients, which means there cannot be a prime~$\ell$ with $\varphi(\tq_i) < \varphi(\ell) < \varphi(\tq_{i+1})$, which by Lemma~\ref{lemma:Bertrand} means that $\tq_i$ and $\tq_{i+1}$ cannot be powers of the same prime. These arguments hold whenever $\varphi(\tq_i) \ge 4$, and the remaining cases are easily checked by hand.
\end{proof}

\begin{remark}
The proof of Proposition~\ref{prop:two-totient sequences} gives a bit of extra information, namely that $\tq_{i-1}$ is prime in cases~(b) and~(d) and $\tq_{i+1}$ is prime in cases~(c) and~(d).
\end{remark}

\subsection{Invariant factors of $y$-typical numbers}

Before summarizing the important properties of $y$-typical numbers for later use, we consider the following definitions of functions involving the prime powers with totient values up to some parameter.

\begin{defn} \label{defn:iV}
\definitionVW
As a reality check on the notation, we can verify from Definition~\ref{defn:standard_invariant_sequence} that $\sif_{\iW(x)} = \{ \iV(x) \}$ for all $x \ge 1$.
\end{defn}

The two functions just defined are closely related to the classical prime-counting functions
\begin{align}
\psi(x) &= \sum_{n\le x} \Lambda(n) = \log \bigl( \lcm\{ p^\alpha \colon p^\alpha \le x \} \bigr) \sim x \notag \\
\pi^*(x) &= \#\{ p^\alpha \le x \} = \sum_{k=1}^\infty \pi(x^{1/k}) \sim \frac x{\log x}; \label{psi and pi* def}
\end{align}
we detour slightly to establish a brief lemma solidifying this intuition.

\begin{lemma}
For all $x\ge3$, we have $0 \le \lV(x)-\psi(x) \ll \log x\log\log x$ and $0 \le \iW(x)-\pi^*(x) \ll \log x$.
\end{lemma}

\begin{proof}
We first observe that for any integer $\alpha\ge1$, there is at most one prime~$p$ such that $\varphi(p^\alpha) \le x < p^\alpha$. Indeed, if $p^\alpha(1-\frac1p) = \varphi(p^\alpha) \le x < p^\alpha$ then $p-1 = p(1-\frac1p) \le p(1-\frac1p)^{1/\alpha} \le x^{1/\alpha} < p$, and therefore $p = \lfloor x^{1/\alpha} \rfloor +1$ is forced.

Now, writing
\begin{align*}
\psi(x) &= \sum_p \log p \sum_{\substack{\alpha\in\N \\ p^\alpha \le x}} 1 = \sum_{\alpha\in\N} \sum_{\substack{p \\ p^\alpha \le x}} \log p \\
\lV(x) &= \sum_p \log p \sum_{\substack{\alpha\in\N \\ \varphi(p^\alpha) \le x}} 1 = \sum_{\alpha\in\N} \sum_{\substack{p \\ \varphi(p^\alpha) \le x}} \log p,
\end{align*}
we immediately see that $\lV(x) \ge \psi(x)$. More precisely, each sum can be restricted to $\alpha \le \frac{\log x}{\log 2}+1$ (for otherwise $2^{\alpha-1} > x$), and thus
\begin{align*}
\lV(x) - \psi(x) &= \sum_{\alpha\le(\log x)/\log 2+1} \sum_{\substack{p \\ \varphi(p^\alpha) \le x < p^\alpha}} \log p \\
&\le \sum_{\alpha\le(\log x)/\log 2+1} \log\bigl( \lfloor x^{1/\alpha} \rfloor +1 \bigr) \ll \sum_{\alpha\le(\log x)/\log 2+1} \alpha\log x \ll \log x \log \log x.
\end{align*}
Similarly,
\begin{align*}
\iW(x) - \pi^*(x) &= \sum_{\alpha\le(\log x)/\log 2+1} \sum_{\substack{p \\ \varphi(p^\alpha) \le x < p^\alpha}} 1 \le \sum_{\alpha\le(\log x)/\log 2+1} 1 \ll \log x.
\qedhere
\end{align*}
\end{proof}

The following proposition gives a condensed summary of the important properties of $y$-typical numbers, using the functions~$\iV(x)$ and~$\iW(x)$ just defined as well as the total $\varphi$-sequence $(\tq_i)$ and the universal factor order sets $(\sif_i)$ from Definitions~\ref{defn:phi_sequence} and~\ref{defn:standard_invariant_sequence}. Part~\eqref{yt6} of the proposition relates the invariant factor counting function $\inv(m;d)$ from Definition~\ref{defn:group_counting_functions_inv} to the additive function $\omega(m;S-T)$ from Definition~\ref{defn:omega_functions}.

\begin{prop}
\label{prop:y_typical}
Let $y \ge 1$, and set $N=\iW(y)$, so that $\tq_1,\dots,\tq_N$ are the prime powers whose totients are less than or equal to $\varphi(\tq_N)$. Fix a $y$-typical number~$m$, and define $q_1,\dots,q_N$ to be $\tq_1,\dots,\tq_N$ ordered so that for each $1\le i\le N-1$,
\[
\prm(m;q_i) \ge \prm(m;q_{i+1}), \text{ and } \varphi(q_i) \le \varphi(q_{i+1}) \text{ if } \prm(m;q_i) = \prm(m;q_{i+1}).
\]
For each $1\le i\le N$, define $d_i = \lcm [ q_1, \dots, q_i]$. Then:
\begin{enumerate}
\item	$1 = \varphi(q_1) \le \cdots \le \varphi(q_N)$.  \label{yt1}
\item	Let $1 \le i \le N-1$. If $\varphi(q_i) < \varphi(q_{i+1})$ then $d_i = \iV(\varphi(\tq_i))$ and $\sif_i = \{d_i\}$.  \label{yt4}
\item	For $1 \le i \le N-1$, we have $d_i \in \sif_i$.  \label{yt5}
\item	$d_N = \iV(y)$ and $\sif_N = \{d_N\}$.  \label{yt2}
\item	If $d$ is an invariant factor of $(\Z/m\Z)^\times$ not exceeding $\iV(y)$, then $d=d_i$ for some $1\le i\le N$.  \label{ytnew}
\item	If $d$ is an invariant factor of $(\Z/m\Z)^\times$ exceeding $\iV(y)$, then $\iV(y) \mid d$.  \label{yt3}
\item	Define $Q_i = \{ p \colon p\equiv1\mod{q_i} \}$. Then for $1 \le i \le N-1$, 
\begin{equation}  \label{J as an additive function}
  \inv(m;d_i) = \prm(m;q_i) - \prm(m;q_{i+1}) = \omega(m;Q_i-Q_{i+1}) + \bar E(m;q_i) - \bar E(m;q_{i+1}) ,
\end{equation}
where $\bar E$ is the function defined in Lemma~\ref{lemma:prime_multiples}.
Furthermore, $\inv(m;d_N) \le \omega(m;Q_N) + \bar E(m;q_N)$.  \label{yt6}
\end{enumerate}
\end{prop}

\noindent
The reader should take care to distinguish the temporary ordering $q_1,\dots,q_N$, which depends on~$m$, from the universal ordering $\tq_1,\dots,\tq_N$ from Definition~\ref{defn:phi_sequence}. Indeed, part of what is being proved in this proposition is that the two orderings are essentially the same (which is what $y$-typicality is designed to achieve), except that when $\varphi(\tq_i) = \varphi(\tq_{i+1})$ we only have the equality of unordered sets $\{q_i,q_{i+1}\} = \{\tq_i,\tq_{i+1}\}$. Recall Definition~\ref{defn:y_typical} for $y$-typicality.

\begin{proof}
Throughout this proof, $1\le i\le N-1$.

\begin{enumerate}
\item 
If $\prm(m;q_i) = \prm(m;q_{i+1})$, we are already given that $\phi(q_i) \le \varphi(q_{i+1})$. On the other hand, if $\prm(m;q_i) > \prm(m;q_{i+1})$, then $\phi(q_i) > \varphi(q_{i+1})$ would violate the assumption that~$m$ is $y$-typical; therefore $\phi(q_i) \le \varphi(q_{i+1})$ in this case as well. This argument establishes the inequalities in part~\eqref{yt1}, which in turn implies that $q_1=2$ (since no prime power has as small a totient value) and thus that $\varphi(q_1)=1$.

\item 
By part~\eqref{yt1} and the assumption $\varphi(q_i) < \varphi(q_{i+1})$, there are exactly~$i$ prime powers with totient values less than or equal to $\varphi(q_i)$, and therefore $\varphi(q_i) = \varphi(\tq_i) < \varphi(\tq_{i+1})$. (Note that we can't conclude that $q_i=\tq_i$, since we might have $\varphi(\tq_{i-1}) = \varphi(\tq_i)$.) Then both $\sif_i = \{\lcm[\tq_1,\dots,\tq_i]\} = \{\lcm[q_1,\dots,q_i]\} = \{d_i\}$ and $\sif_i = \sif_{\iW(\varphi(\tq_i))} = \{ \iV(\varphi(\tq_i)) \}$ follow from Definition~\ref{defn:standard_invariant_sequence}.

\item
The case $\varphi(q_i) < \varphi(q_{i+1})$ is already done by part~\eqref{yt4}. On the other hand, if $\varphi(q_i) = \varphi(q_{i+1})$, then by part~\eqref{yt1} and Lemma~\ref{lemma:equal_totient},
\[
\varphi(q_1) \le \cdots \le \varphi(q_{i-1}) < \varphi(q_i) = \varphi(q_{i+1}) < \varphi(q_{i+2}).
\]
It follows that $\{q_1,\dots,q_{i-1}\} = \{\tq_1,\dots,\tq_{i-1}\}$ and $\{q_i,q_{i+1}\} = \{\tq_i,\tq_{i+1}\}$, and so $d_i = \lcm[q_1,\dots,q_i] = \lcm[\tq_1,\dots,\tq_{i-1}, q_i] \in \sif_i$ by Definition~\ref{defn:standard_invariant_sequence} regardless of whether $q_i=\tq_i$ or $q_i=\tq_{i+1}$.

\item
Since $\varphi(q_N) < \varphi(\tq_{N+1})$ by construction, the proof of part~\eqref{yt2} is the same as the proof of part~\eqref{yt4} (on noting that $\iV(\varphi(\tq_N)) = \iV(y)$).

\item
Extend the sequence $q_1,\dots,q_N$ to a longer sequence $q_1,\dots,q_M$ that contains all the distinct elementary divisors of $(\Z/m\Z)^\times$, ordered so that $\inv(m;q_i)$ is decreasing in~$i$ (this extension preserves the first~$N$ elements precisely because of part~(b) of the definition of $y$-typicality), and extend the sequence~$d_i$ accordingly. Then all invariant factors of $(\Z/m\Z)^\times$ are in the sequence $d_1,\dots,d_M$ by Proposition~\ref{prop:elementary_divisors_to_invariant_factors}. In particular, since each $d_i \mid d_{i+1}$ and $d_N = \iV(y)$ by part~\eqref{yt2}, all invariant factors up to $d_N$ must be in $\{d_1,\dots,d_N\}$ as claimed.

\item 
Since $d_N = \iV(y)$ by part~\eqref{yt2}, part~\eqref{yt3} follows immediately from Lemma~\ref{lemma:invariant_factors}, since the assumption that~$m$ is $y$-typical implies that $\prm(m;q_N) > \prm(m;q)$ for every prime power $q\notin\{q_1,\dots,q_N\}$.

\item 
The equality $\inv(m;d_i) = \prm(m;q_i) - \prm(m;q_{i+1})$ follows immediately from Lemma~\ref{lemma:invariant_factors}, and thus by Lemma~\ref{lemma:prime_multiples},
\begin{align*}
\inv(m;d_i) &= \bigl( \omega(m;Q_i) + \bar E(m;q_i) \bigr) - \bigl( \omega(m;Q_{i+1}) + \bar E(m;q_{i+1}) \bigr) \\
&= \omega(m;Q_i)-\omega(m;Q_{i+1}) + \bar E(m;q_i) - \bar E(m;q_{i+1})
\end{align*}
by Definition~\ref{defn:omega_functions}. The assertions $\inv(m;d_N) \le \prm(m;q_N) \le \omega(m;Q_N) + \bar E(m;q_N)$ similarly follow from Lemmas~\ref{lemma:invariant_factors} and~\ref{lemma:prime_multiples}. \qedhere

\end{enumerate}
\end{proof}

This concludes the examination of the algebraic properties of $y$-typical numbers. The study of the frequency of $y$-typical
numbers, and the justification for why they are indeed typical, can be found in Section~\ref{sec:typicality}.

\section{From additive functions to probability distributions} \label{sec:additive to probability}

In this section, we form a bridge between additive functions and probability distributions by using the techniques of Kubilius to explore the properties of the strongly additive functions $\omega(m;S-T)$, including $\omega(m;S)$ and $\omega(m;q,a)$, from Definition~\ref{defn:omega_functions}. Given that Section~\ref{sec:invariant_to_additive} connected invariant factors to these additive functions, the cumulative effect by the end of this section will be the ability to obtain multivariate limiting distributions for vectors whose components are the additive functions arising in the study of invariant factors. Proofs of our primary theorems, however, will still require further work in Section~\ref{sec:distributions} to translate the multivariate distributions discussed in this section into univariate distributions for the invariant factors themselves.

\subsection{Concepts from Kubilius}

Our arguments will make significant use of material from Kubilius~\cite{kubilius} concerning joint distributions of multiple additive functions. For our purposes, it will be important to understand the dependence of error terms on the additive functions involved; accordingly, we include some slightly modified versions of Kubilius's results in this section.

The following definitions encode information about the asymptotic means and variances of additive functions:

\begin{defn}
\label{defn:kubilius_functions}
\definitionABCD
\end{defn}

%

\begin{defn} \label{rho def}
For any prime power $p^\alpha$, define
\[ \rho(p^\alpha) = \frac 1{p^\alpha}\bigg(1-\frac 1p\bigg) . \]
Since the number of integers up to $n$ that are divisible by an integer $d$ is $\frac nd+O(1)$, we have
\begin{align*}
\#\{ m \le n \colon p^\alpha \xmid m \} &= \#\{ m \le n \colon p^\alpha \mid m \} - \#\{ m \le n \colon p^{\alpha+1} \mid m \} \\
&= \frac n{p^\alpha}+O(1) - \frac n{p^{\alpha+1}}+O(1) \\
&= n\rho(p^\alpha) + O(1).
\end{align*}
\end{defn}

We first give an adaptation of~\cite[Lemmas~4.2 and~4.3]{kubilius}.

\begin{lemma}
\label{lemma:kubilius_special_bounds}
Uniformly for all additive functions~$f$ and all integers $n\ge2$,
\[
\Ex_n(f(m)) - \sum_{p^\alpha \le n} f(p^\alpha)\rho(p^\alpha) \ll \frac{C_f(n)}{\log n} \qquad\text{and}\qquad
\Ex_n(f(m)) - A_f(n) \ll C_f(n).
\]
\end{lemma}

\begin{proof}
In this proof, all implicit constants are absolute.
We calculate that
\begin{align*}
\Ex_n(f(m))
&= \frac 1n \sum_{m=1}^n f(m) = \frac1n \sum_{m=1}^n \sum_{p^\alpha \xmid m} f(p^\alpha) \\
&= \frac 1n \sum_{p^\alpha \le n} f(p^\alpha)\#\{m \le n \colon p^\alpha \xmid m\} \\
&= \frac 1n \sum_{p^\alpha \le n} f(p^\alpha) \bigl(n\rho(p^\alpha) + O(1) \big) \\
&= \sum_{p^\alpha \le n} f(p^\alpha)\rho(p^\alpha) + O\biggl( \frac{C_f(n)}n \sum_{p^\alpha \le n} 1 \biggr) = \sum_{p^\alpha \le n} f(p^\alpha)\rho(p^\alpha) + O\biggl( \frac{C_f(n)}n \pi^*(n) \biggr),
\end{align*}
using the prime-power counting function~$\pi^*$ defined in equation~\eqref{psi and pi* def}; this establishes the first inequality since $\pi^*(n) \ll n/\log n$. Continuing,
\begin{align*}
\Ex_n(f(m))
&= \sum_{p^\alpha \le n} f(p^\alpha)\rho(p^\alpha) + O\biggl( \frac{C_f(n)}{\log n} \biggr) \\
&= \sum_{p \le n} \biggl( \frac{f(p^\alpha)}p - \frac{f(p^\alpha)}{p^2} \biggr) + \sum_{\substack{p^\alpha \le n \\ \alpha \ge 2}} f(p^\alpha)\rho(p^\alpha) + O\biggl( \frac{C_f(n)}{\log n} \biggr) \\
&= A_f(n) + O\biggl( \sum_{p \le n} \frac{C_f(n)}{p^2} + \sum_{\substack{p^\alpha \le n \\ \alpha \ge 2}} \frac{C_f(n)}{p^\alpha} + \frac{C_f(n)}{\log n} \biggr)
\end{align*}
since $\rho(p^\alpha) \le \frac1{p^\alpha}$. The first sum in the error term is $\ll C_f(n)$, while
\[
\sum_{\substack{p^\alpha \le n \\ \alpha \ge 2}} \frac{C_f(n)}{p^\alpha}
\le C_f(n) \sum_{p \le n} \sum_{\alpha = 2}^\infty \frac 1{p^\alpha}
= C_f(n) \sum_{p \le n} \frac 1{p^2-p} \ll C_f(n)
\]
as well. We conclude that
\[
\Ex_n(f(m)) = A_f(n) + O\biggl( C_f(n) + C_f(n) + \frac{C_f(n)}{\log n} \biggr) = A_f(n) + O(C_f(n))
\]
as claimed.
\end{proof}

We next provide an adaptation of~\cite[Lemma~4.4 (see also the beginning of Chapter~3)]{kubilius}.

\begin{lemma} \label{lemma:ksfv1}
Uniformly for all additive functions $f_1$ and $f_2$,
\[
\Ex_n(f_1(m)f_2(m)) = \sum_{p^\alpha q^\beta \le n} f_1(p^\alpha)f_2(q^\beta) \rho(p^\alpha)\rho(q^\beta) + \sum_{p^\alpha\le n} \frac{f_1(p^\alpha)f_2(p^\alpha)}{p^\alpha} + O\bigl( C_{f_1}(n)C_{f_2}(n) \bigr).
\]
\end{lemma}

\begin{proof}
We begin by writing
\begin{align*}
\Ex_n(f_1(m)f_2(m)) = \frac1n \sum_{m\le n} f_1(m)f_2(m) &= \sum_{m\le n} \sum_{p^\alpha\xmid m} f_1(p^\alpha) \sum_{q^\beta\xmid m} f_2(q^\beta) \\
&= \frac1n \sum_{\substack{p^\alpha \le n \\ q^\beta \le n}} f_1(p^\alpha)f_2(q^\beta)\#\{ m \le n \colon p^\alpha \xmid m,\, q^\beta \xmid m \} \\
&= \frac1n \sum_{\substack{p^\alpha q^\beta \le n \\ p\ne q}} f_1(p^\alpha)f_2(q^\beta)\#\{ m \le n \colon p^\alpha \xmid m,\, q^\beta \xmid m \} \\
&\qquad{}+ \sum_{p^\alpha\le n} f_1(p^\alpha)f_2(p^\alpha) \#\{ m \le n \colon p^\alpha \xmid m \}.
\end{align*}
Since
\begin{align*}
\#\{ m \le n \colon p^\alpha \xmid m, q^\beta \xmid m \}
&= \#\{ m \le n \colon p^\alpha q^\beta \mid m \}
- \#\{ m \le n \colon p^{\alpha+1} q^\beta \mid m \} \\
&\qquad{}- \#\{ m \le n \colon p^\alpha q^{\beta+1} \mid m \}
+ \#\{ m \le n \colon p^{\alpha+1} q^{\beta+1} \mid m \} \\
&= \frac n{p^\alpha q^\beta} - \frac n{p^{\alpha+1} q^\beta} - \frac n{p^\alpha q^{\beta+1}} + \frac n{p^{\alpha+1} q^{\beta+1}} + O(1)
\end{align*} 
when $p\ne q$, we see that
\begin{align}
\Ex_n(f_1(m)f_2(m)) &= \sum_{\substack{p^\alpha q^\beta \le n \\ p\ne q}} f_1(p^\alpha)f_2(q^\beta) \biggl( \rho(p^\alpha)\rho(q^\beta) + O\biggl( \frac1n \biggr)\biggr) \notag \\
&\qquad{}+ \sum_{p^\alpha\le n} f_1(p^\alpha)f_2(p^\alpha) \biggl( \rho(p^\alpha) + O\biggl( \frac1n \biggr)\biggr) \notag \\
&= \sum_{p^\alpha q^\beta \le n} f_1(p^\alpha)f_2(q^\beta) \biggl( \rho(p^\alpha)\rho(q^\beta) + O\biggl( \frac1n \biggr)\biggr)- \sum_{p^\alpha p^\beta \le n} f_1(p^\alpha)f_2(p^\beta) \rho(p^\alpha)\rho(p^\beta) \notag \\
&\qquad{} + \sum_{p^\alpha\le n} f_1(p^\alpha)f_2(p^\alpha) \biggl( \frac1{p^\alpha} + O\biggl( \frac1{p^{\alpha+1}} \biggr)\biggr).  \label{eqn:ksfv1.1}
\end{align}
We quickly estimate the error terms in this expression, since
\[
\sum_{p^\alpha q^\beta \le n} f_1(p^\alpha)f_2(q^\beta) O\biggl( \frac1n \biggr) \ll \frac{C_{f_1}(n)C_{f_2}(n)}n \sum_{p^\alpha q^\beta \le n} 1 \le \frac{C_{f_1}(n)C_{f_2}(n)}n \sum_{k\le n} 1 = C_{f_1}(n)C_{f_2}(n)
\]
(we could save a factor of $(\log\log n)/\log n$, but other error terms make that unneeded) and
\begin{align*}
\sum_{p^\alpha\le n} f_1(p^\alpha)f_2(p^\alpha) O\biggl( \frac1{p^{\alpha+1}} \biggr) &\le C_{f_1}(n)C_{f_2}(n) \sum_{p \le n} \sum_{\alpha = 1}^\infty \frac 1{p^{\alpha+1}} \\
&= C_{f_1}(n)C_{f_2}(n) \sum_{p \le n} \frac 1{p^2-p} \ll C_{f_1}(n)C_{f_2}(n).
\end{align*}
Furthermore,
\begin{align*}
\sum_{p^\alpha q^\beta \le n} f_1(p^\alpha)f_2(q^\beta)\rho(p^\alpha)\rho(q^\beta)
& \ll C_{f_1}(n)C_{f_2}(n) \sum_{\substack{p^{\alpha+\beta} \le n \\ \alpha,\beta \ge 1}} \frac1{p^{\alpha+\beta}} \\
& \le C_{f_1}(n)C_{f_2}(n) \sum_{p \le n} \sum_{\alpha = 1}^\infty \sum_{\beta = 1}^\infty \frac 1{p^{\alpha+\beta}} \\
& \le C_{f_1}(n)C_{f_2}(n) \sum_{p \le n} \frac 1{(p-1)^2} \ll C_{f_1}(n)C_{f_2}(n).
\end{align*}
Inserting these three estimates into equation~\eqref{eqn:ksfv1.1} establishes the lemma.
\end{proof}

\begin{lemma} \label{lemma:ksfv2}
Uniformly for all additive functions $f_1$ and $f_2$,
\[
\Ex_n(f_1(m))\Ex_n(f_2(m)) = \sum_{p^\alpha \le n} \sum_{q^\beta \le n} f_1(p^\alpha)f_2(q^\beta)\rho(p^\alpha)\rho(q^\beta) + O\biggl( C_{f_1}(n)C_{f_2}(n) \frac{\log\log n}{\log n} \biggr).
\]
\end{lemma}

\begin{proof}
By Lemma~\ref{lemma:kubilius_special_bounds},
\begin{align*}
\Ex_n (f_1(m))\Ex_n (f_2(m))
&= \bigg(\sum_{p^\alpha \le n} f_1(p^\alpha)\rho(p^\alpha) + O\bigg(\frac{C_{f_1}(n)}{\log n}\bigg)\bigg)
    \bigg(\sum_{q^\beta \le n} f_2(q^\beta)\rho(q^\beta) + O\bigg(\frac{C_{f_2}(n)}{\log n}\bigg)\bigg) \\
&= \sum_{p^\alpha \le n} \sum_{q^\beta \le n} f_1(p^\alpha)f_2(q^\beta)\rho(p^\alpha)\rho(q^\beta) + O\biggl( \frac{C_{f_2}(n)}{\log n} \sum_{p^\alpha\le n} |f_1(p^\alpha)|\rho(p^\alpha) \\
&\qquad{} + \frac{C_{f_1}(n)}{\log n} \sum_{q^\beta\le n} |f_2(p^\alpha)|\rho(q^\beta) + \frac{C_{f_1}(n)C_{f_2}(n)}{\log^2 n} \biggr) \\
&= \sum_{p^\alpha \le n} \sum_{q^\beta \le n} f_1(p^\alpha)f_2(q^\beta)\rho(p^\alpha)\rho(q^\beta) + O\biggl( \frac{C_{f_1}(n)C_{f_2}(n)}{\log n} \sum_{p^\alpha\le n} \rho(p^\alpha) \biggr).
\end{align*}
The lemma now follows since
\[
\sum_{p^\alpha\le n} \rho(p^\alpha) \le  \sum_{p\le n} \sum_{\alpha=1}^\infty \rho(p^\alpha) = \sum_{p\le n} \frac1p \ll \log\log n.
\qedhere
\]
\end{proof}

\begin{lemma}
\label{lemma:kubilius_special_finite_variance}
Uniformly for all additive functions $f_1$ and $f_2$,
\[
\Ex_n(f_1(m)f_2(m)) - \Ex_n(f_1(m))\Ex_n(f_2(m)) = \sum_{p^\alpha \le n} \frac{f_1(p^\alpha)f_2(p^\alpha)}{p^\alpha} + O\big(C_{f_1}(n)C_{f_2}(n)\big).
\]
\end{lemma}

\begin{proof}
Subtracting the formula in Lemma~\ref{lemma:ksfv2} from the formula in Lemma~\ref{lemma:ksfv1} yields
\begin{align*}
\Ex_n(f_1(m)f_2(m)) &{} - \Ex_n(f_1(m))\Ex_n(f_2(m)) \\
&= \sum_{p^\alpha\le n} \frac{f_1(p^\alpha)f_2(p^\alpha)}{p^\alpha} - \sum_{\substack{p^\alpha\le n \\ q^\beta\le n \\ p^\alpha q^\beta>n}} f_1(p^\alpha)f_2(q^\beta)\rho(p^\alpha)\rho(q^\beta) + O\big(C_{f_1}(n)C_{f_2}(n)\big) \\
&= \sum_{p^\alpha\le n} \frac{f_1(p^\alpha)f_2(p^\alpha)}{p^\alpha} + O\biggl( C_{f_1}(n)C_{f_2}(n) \sum_{\substack{p^\alpha\le n \\ q^\beta\le n \\ p^\alpha q^\beta>n}} \frac1{p^\alpha} \frac1{q^\beta} + C_{f_1}(n)C_{f_2}(n) \biggr).
\end{align*}
But by~\cite[equation~(3.7)]{kubilius}, the sum in the error term is bounded, which completes the proof of the proposition.
\end{proof}

The main distributional results of Kubilius hold for a specific class of additive functions, which we now define.

\begin{defn}
\label{defn:kubilius_H}
The class~$\cH$ is the set of real-valued additive functions $f(n)$
for which $\lim_{n\to\infty} D_f(n) = \infty$ and for which
there exists an unbounded increasing function $r(n) = n^{o(1)}$ such that
$\lim_{n\to\infty} {D_f(r)}/{D_f(n)} = 1$.
For strongly additive functions, it is equivalent to stipulate that $\lim_{n\to\infty} B_f(n) = \infty$ and $\lim_{n\to\infty} B_f(r)/B_f(n) = 1$.
\end{defn}

The following proposition gives the joint limiting distribution of a finite number of additive functions under the
same scaling.
Define the vector
\begin{equation} \label{Lnp def}
\mathbf L_n(p) = \biggl(\frac{f_1(p)}{B_{f_1}(n)}, \ldots, \frac{f_s(p)}{B_{f_s}(n)}\biggr).
\end{equation}

\begin{prop}[Theorem 8.1 of \cite{kubilius}]
\label{prop:kubilius_multivariable}
If $f_1(m),\ldots,f_s(m)\in\cH$ are strongly additive functions, then the following are equivalent:
\begin{enumerate}
\item	
$\displaystyle F(x_1,\ldots,x_s) = \lim_{n\to\infty} P_n \bigg( \frac{f_1(m)-A_{f_1}(n)}{B_{f_1}(n)} < x_1,\, \ldots,\, \frac{f_s(m)-A_{f_s}(n)}{B_{f_s}(n)} < x_s \bigg)$
is the cumulative distribution function of a random vector~$\mathbf Y$;
\item There exists a nonnegative completely additive set
function~$Q(\mathcal E)$ defined for all Borel sets of $\R^s$ not containing the origin $\mathbf 0$,
a constant vector~$\mathbf\Gamma$, and
a nonnegative quadratic form~$\sigma(\mathbf T)$ in~$s$ variables
that satisfy:
\begin{enumerate}
\item
$\displaystyle \lim_{n\to\infty} \sum_{\substack{p \le n \\ \mathbf L_n(p) \in \mathcal I}} \frac 1p = Q(\mathcal I)$
for every interval of continuity~$\mathcal I$ of~$Q(\mathcal E)$ with $\mathbf 0 \not\in \bar{\mathcal I}$;
\item	$\displaystyle \lim_{n\to\infty} \sum_{p \le n} \frac{|\mathbf L_n(p)|^2}{p(1+|\mathbf L_n(p)|^2)} \mathbf L_n(p) = -\mathbf\Gamma$;
\item	$\displaystyle \lim_{\ep\to 0^+} \lim_{n\to\infty} \sum_{\substack{p \le n \\ |\mathbf L_n(p)| < \ep}} \frac{(\mathbf T \cdot \mathbf L_n(p))^2}p = \sigma(\mathbf T). $
\end{enumerate}
\end{enumerate}
When these conditions hold, the characteristic function $\chi_{\mathbf Y}(\mathbf T)$ of~$Y$ satisfies
\begin{equation} \label{char function before 0s}
\log \chi_{\mathbf Y}(\mathbf T) = i\mathbf\Gamma\cdot \mathbf T - \frac{\sigma(\mathbf T)}2+ \int_{|\mathbf X| > 0} \bigg(e^{i \mathbf T\cdot \mathbf X}-1-\frac{i\mathbf T\cdot \mathbf X}{1+|\mathbf X|^2}\bigg) \, dQ.
\end{equation}
\end{prop}

\begin{remark}
In its full generality, the statement of Proposition~\ref{prop:kubilius_multivariable} calls for some clarifications.
An {\em exceptional point} for a vector-valued random variable is any point at which any of the random vector's component's cumulative distribution function is discontinuous.
Since a cumulative distribution function can have exceptional points, in part~(a) the limit defining $F(x_1,\dots,x_s)$ is only required to exist for nonexceptional points of~$F$.

Also, in part~(b)(i), an {\em interval of continuity} $\mathcal I$ for a function~$Q$ on the Borel sets in $\R^s$ is an $s$-dimensional rectangle with the property that $\lim_{n\to\infty} Q(\mathcal A_n) = Q(\mathcal I)$ for any ascending chain of Borel sets $\mathcal A_1 \subseteq \mathcal A_2 \subseteq \cdots$ where $\bigcup_{n=1}^\infty \mathcal A_i = \mathcal I$,
and similarly that $\lim_{n\to\infty} Q(\mathcal B_n) = Q(\mathcal I)$ for any descending chain of Borel sets $\mathcal B_1 \supseteq \mathcal B_2 \supseteq \cdots$ where $\bigcap_{n=1}^\infty \mathcal B_i = \mathcal I$.

All this being said, our application of Proposition~\ref{prop:kubilius_multivariable} (see
Proposition~\ref{lemma:omega_diff_multivariable} below) will yield a multivariate normal distribution; in particular, there will be no exceptional points, and every interval will be an interval of continuity, in that situation.
\end{remark}

We have now concluded the adaptation of results from~\cite{kubilius}; in the next section we apply these results to the additive functions relevant to our work.

\subsection{Application to the relevant additive functions} \label{omega S minus T section}

We now use the results of the previous section to establish multivariate Erd\H os--Kac theorems for vector-valued functions whose components are of the form $\omega(m;S-T)$.
We first set out some convenient notation for our analysis of the additive functions from Definition~\ref{defn:omega_functions} and their associated quantities from Definition~\ref{defn:kubilius_functions}.

\begin{defn}  \label{defn:indicator and sdif}
Given a set~$S$ of primes, define its indicator function
\[
1_S(p) = \begin{cases}
1, &\text{if } p\in S, \\
0, & \text{if } p\notin S.
\end{cases}
\]
We write $A_\omega(x;S)$ for $A_{\omega(m;S)}(x)$ and similarly for the other quantities from Definition~\ref{defn:kubilius_functions}, so that in particular
\[
A_\omega(x;S) = B_\omega(x;S)^2 = \sum_{p \le x} \frac{1_S(p)}p
\]
When $S=\{p\equiv a\mod q\}$, we similarly write $A_\omega(x;q,a)$ for $A_{\omega(m;q,a)}(x)$ and so on.

Also, given two sets~$S$ and~$T$, define their symmetric difference $S\sdif T = (S\setminus T) \cup (T\setminus S)$. We also write $A_\omega(x;S-T)$ for $A_{\omega(m;S-T)}(x)$ and so on, so that
\begin{equation} \label{A and B for S-T}
A_\omega(x;S-T) = A_\omega(x;S)-A_\omega(x;T) \quad\text{and}\quad
B_\omega(x;S-T)^2 = \sum_{p \le x} \frac{1_{S \sdif T}(p)}p .
\end{equation}
\end{defn}

\begin{lemma}
\label{lemma:omega_finite_mean_variance}
Uniformly for all sets~$S$ and~$T$ of prime numbers,
\begin{align*}
\Ex_n \bigl( \omega(m;S-T) \bigr) &= A_\omega(n;S-T) + O(1).
\end{align*}
\end{lemma}

\begin{proof}
Note that $0 \le C_\omega(n;S-T) \le 1$ for any sets~$S$ and~$T$.
Thus by Lemma~\ref{lemma:kubilius_special_bounds},
\[ \Ex_n \bigl( \omega(m;S-T) \bigr) = A_\omega(n;S-T) + O\bigl( C_\omega(n;S-T)^2 \bigr) = A_\omega(n;S-T) + O(1), \]
where the implicit constant is absolute.
\end{proof}

\begin{lemma}
\label{lemma:sum_arithmetic_progression_pnt}
Uniformly for relatively prime integers~$a$ and $q \ge 2$ and real numbers $x \ge 3$,
\[
A_\omega(x;q,a) = B_\omega(x;q,a)^2 = \frac{\log\log x}{\varphi(q)} + O\biggl( \frac1{p(q,a)} + \frac{\log q}{\varphi(q)} \biggr),
\]
where $p(q,a)$ is the smallest prime congruent to $a\mod q$. In particular,
\[
A_\omega(x;q,1) = B_\omega(x;q,1)^2 = \frac{\log\log x}{\varphi(q)} + O\biggl( \frac{\log q}{\varphi(q)} \biggr).
\]
\end{lemma}

\begin{proof}
By definition,
\[
A_\omega(x;q,a) = B_\omega(x;q,a)^2 = \sum_{\substack{p \le x \\ p \equiv a \mod q}} \frac 1p = \frac{\log\log x}{\varphi(q)} + O\biggl( \frac1{p(q,a)} + \frac{\log q}{\varphi(q)} \biggr),
\]
where the last asymptotic formula appears in~\cite[Remark~1 after Theorem~1]{p77}. The second assertion follows since $p(q,1) \ge q+1 > \varphi(q)$.
\end{proof}

\begin{defn}  \label{defn:Dirichlet density}
\definitionDd
Note that for $0<\mu\le 1$, we have $\den(S) = \mu$ if and only if $A_\omega(x;S) = B_\omega(x;S)^2 \sim \mu\log\log x$. In particular, Dirichlet showed that $\den\bigl( \{p\equiv a\mod q\} \bigr) = \frac1{\varphi(q)}$ when~$a$ and $q\ge2$ are relatively prime.
\end{defn}

The following lemma, used in the proof of Proposition~\ref{lemma:omega_diff_multivariable} below, will help us practice the notation just introduced.

\begin{lemma} \label{four sets lemma}
Let $S_1$, $S_2$, $T_1$, and $T_2$ be sets of primes such that each $S_i\cap T_j$ has a Dirichlet density. Then
\begin{multline*}
\lim_{n\to\infty} \frac1{\log\log n} \sum_{p \le n} \frac{(1_{S_1}(p)-1_{T_1}(p))(1_{S_2}(p)-1_{T_2}(p))}p  \\
= \den(S_1\cap S_2) + \den(T_1\cap T_2) - \den(S_1\cap T_2) - \den(T_1\cap S_2).
\end{multline*}
\end{lemma}

\begin{proof}
We have simply
\begin{align*}
\lim_{n\to\infty} {} & \frac1{\log\log n} \sum_{p\le n} \frac{(1_{S_1}(p)-1_{T_1}(p))(1_{S_2}(p)-1_{T_2}(p))}p \\
&= \lim_{n\to\infty} \frac1{\log\log n} \sum_{p \le n} \frac{1_{S_1}(p)1_{S_2}(p)+1_{T_1}(p)1_{T_2}(p)-1_{S_1}(p)1_{T_2}(p)-1_{T_1}(p)1_{S_2}(p)}p \\
&= \lim_{n\to\infty} \frac1{\log\log n} \biggl( \sum_{p \le n} \frac{1_{S_1 \cap S_2}(p)}p + \sum_{p \le n} \frac{1_{T_1\cap T_2}(p)}p + \sum_{p \le n} \frac{1_{S_1\cap T_2}(p)}p + \sum_{p \le n} \frac{1_{T_1\cap S_2}(p)}p \biggr) \\
&= \den(S_1\cap S_2) + \den(T_1\cap T_2) - \den(S_1\cap T_2) - \den(T_1\cap S_2)
\end{align*}
by the definition of Dirichlet density.
\end{proof}

As the results of Kubilius
apply primarily to additive functions in the class~$\cH$ from Definition~\ref{defn:kubilius_H}, we need to understand when $\omega(m;S-T)$ lies in~$\cH$.

\begin{lemma}
\label{lemma:omega_diff_H}
Let $S$ and $T$ be sets of primes.
\begin{enumerate}
\item If $S \sdif T$ has positive Dirichlet density, then $\omega(m;S-T)\in\cH$.
\item If $S$ has positive Dirichlet density, then $\omega(m;S)\in\cH$.
\item For any relatively prime integers $a$ and $q\ge2$, we have $\omega(m;q,a)\in\cH$.
\end{enumerate}
\end{lemma}

\begin{proof}
Certainly $\omega(m;S-T)$ is a strongly additive function by definition. Moreover,
\[
B_\omega(x;S-T) = \bigg( \sum_{p \le x} \frac{1_{S \sdif T}(p)}p \bigg)^{1/2} \sim \bigl( \den(S \sdif T) \log\log x \bigr)^{1/2} \to \infty.
\]
Set $r = r(n) = n^{1/\log\log n}$, so that $r = n^{o(1)}$ but still $r\to\infty$. Then
\[ \bigg(\frac{B_\omega(r;S-T)}{B_\omega(n;S)}\bigg)^2 \sim \frac{\den(S \sdif T) \log\log r}{\den(S \sdif T) \log\log n}
= \frac{\log\log n - \log\log\log n}{\log \log n} \sim 1, \]
which completes the verification that $\omega(m;S-T)\in\cH$ as claimed in part~(a). Part~(b) follows from part~(a) by taking $T=\emptyset$, and part~(c) follows from part~(b) by taking $S=\{p\equiv a\mod q\}$.
\end{proof}

With these preliminary definitions and observations out of the way, we move on to statements about limiting distributions.
While the primary tool used in our proofs will concern multivariate limiting distributions, the univariate
distributions will also play a role. The limiting distribution for $\omega(m;S-T)$ can be stated as follows:

\begin{lemma}
\label{lemma:omega_diff_univariable}
Let $S$ and~$T$ be sets of primes such that $\den(S \sdif T) > 0$. Then
\begin{equation} \label{we get Phi}
\lim_{n\to\infty} P_n\bigg( \frac{\omega(m;S-T)-A_\omega(n;S-T)}{B_\omega(n;S-T)} < x \bigg) = \Phi(x),
\end{equation}
the cumulative distribution function of a standard normal random variable.
\end{lemma}

\begin{proof}
By equation~\eqref{A and B for S-T},
\[
B_\omega(n;S-T)^2 = \sum_{p \le n} \frac{1_{S \sdif T}(p)}p \sim \den(S \sdif T) \log\log n
\]
by Definition~\ref{defn:Dirichlet density} of Dirichlet density. In particular, for any fixed $\ep>0$, we see that $|\omega(p;S-T)| \le 1 < \ep B_\omega(n;S-T)$ once~$n$ is sufficiently large in terms of~$\ep$, and consequently
\[
\lim_{n\to\infty} \frac1{B_\omega(n;S-T)^2} \sum_{\substack{p\le n \\ |\omega(p,S-T)| > \ep B_\omega(n;S-T)}} \frac{\omega(p,S-T)^2}p = \lim_{n\to\infty} \frac1{B_\omega(n;S-T)^2} \cdot 0 = 0.
\]
Since $\omega(m;S-T)$ is strongly additive, the evaluation of the limiting cumulative distribution function~\eqref{we get Phi} now follows from~\cite[Theorem~4.2]{kubilius} (an analogue of Lindeberg's condition from probability).
\end{proof}

The multivariate version of the above lemma is essential to later proofs:

\begin{prop}
\label{lemma:omega_diff_multivariable}
Let $Q_1,\ldots,Q_s$ and $R_1,\ldots,R_s$ be sets of primes such that each $Q_i \sdif R_i$ has a nonzero Dirichlet density,
and each $Q_i \cap Q_j$, $Q_i \cap R_j$, and $R_i \cap R_j$ (including when $i=j$) has a Dirichlet density.
Then
\begin{multline} \label{big cumulative distribution}
F(x_1,\ldots,x_s) = \lim_{n\to\infty} P_n \bigg( \frac{\omega(m;Q_1-R_1)-(\den(Q_1)-\den(R_1))\log\log n}{(\den(Q_1 \sdif R_1)\log\log n)^{1/2}} < x_1, \ldots \\
\ldots, \frac{\omega(m;Q_s-R_s)-(\den(Q_s)-\den(R_s))\log\log n}{(\den(Q_s \sdif R_s)\log\log n)^{1/2}} < x_s  \bigg)
\end{multline}
is the cumulative distribution function (with domain $\R^s$) of a random vector $\mathbf X$ with characteristic function
$\chi_{\mathbf X}(t_1,\ldots,t_s) = \exp(-\frac12\sigma(t_1,\ldots,t_s))$,
where
\begin{equation} \label{quad form fraction}
\sigma(t_1,\ldots,t_s) = \sum_{i=1}^s \sum_{j=1}^s \frac{\den(Q_i \cap Q_j)+\den(R_i \cap R_j)-\den(Q_i \cap R_j)-\den(R_i \cap Q_j)} {\den(Q_i \sdif R_i)^{1/2} \den(Q_j \sdif R_j)^{1/2}} t_it_j.
\end{equation}
\end{prop}

\begin{remark}
The quadratic form $\sigma(t_1,\ldots,t_s)$ can be written in the form $\mathbf XM\mathbf X^\intercal$, where $M$ is the symmetric matrix whose $(i,j)$th entry is the fraction in equation~\eqref{quad form fraction}.
Note that when $i=j$,
\begin{align*}
\den(Q_i \cap Q_j) & + \den(R_i \cap R_j) - \den(Q_i \cap R_j) - \den(R_i \cap Q_k) \\
&= \den(Q_i)+\den(R_i)-2\den(Q_i \cap R_i) \\
&= \bigl( \den(Q_i \cap R_i) + \den(Q_i \setminus R_i) \bigr) + \bigl( \den(R_i \cap Q_i) + \den(R_i \setminus Q_i) \bigr) - 2\den(Q_i \cap R_i) \\
&= \den(Q_i \setminus R_i) + \den(R_i \setminus Q_i) = \den(Q_i \sdif R_i),
\end{align*}
and thus the diagonal entries of~$M$ always equal~$1$.

Furthermore, note that the formula for $\chi_{\mathbf X}$ implies that the limiting distribution is a multivariate normal distribution.
The matrix~$M$ is the covariance matrix of the distribution, and each entry in the matrix is the covariance of the random variables that are the vector components of~$\mathbf X$.
While Kubilius~\cite[pages~136--137]{kubilius} gives a short discussion of this possibility as a special case of Proposition~\ref{prop:kubilius_multivariable}, we provide a full proof here for the sake of clarity.
\end{remark}

\begin{remark}
It is helpful to consider the special case $R_1 = \cdots = R_s = \emptyset$, so that we are examining the distribution of $\omega(m;Q_1),\ldots,\omega(m;Q_s)$ since $\omega(m;S) = \omega(m;S-\emptyset)$.
In this situation, the function in equation~\eqref{quad form fraction} simplifies to
\[ \sigma(t_1,\ldots,t_s) = \sum_{i=1}^s \sum_{j=1}^s \frac{\den(Q_i \cap Q_j)}{\den(Q_i)^{1/2}\den(Q_j)^{1/2}} t_it_j. \]
\end{remark}

\begin{proof}[Proof of Proposition~\ref{lemma:omega_diff_multivariable}]
We invoke Proposition~\ref{prop:kubilius_multivariable} with the strongly additive functions $f_i(m) = \omega(m;Q_i-R_i)$; note that these functions are indeed in the class~$\cH$ by Lemma~\ref{lemma:omega_diff_H}. In this setting, equation~\eqref{Lnp def} becomes
\begin{equation} \label{Lnp in this case}
\mathbf L_n(p) = \bigg(
\frac{1_{Q_1}(p)-1_{R_1}(p)}{B_\omega(n;Q_1-R_1)}, \ldots, 
\frac{1_{Q_s}(p)-1_{R_s}(p)}{B_\omega(n;Q_s-R_s)} \bigg)
\end{equation}
using the notation from Definition~\ref{defn:indicator and sdif}.

We need to verify the conditions in Proposition~\ref{prop:kubilius_multivariable}(b), which requires defining three objects; we use the quadratic form $\sigma(\mathbf T)$ from equation~\eqref{quad form fraction}, we define $\mathbf\Gamma = \mathbf 0$, and we define $Q(\mathcal E) = 0$ for all Borel sets $\mathcal E \subseteq \R^s$ whose closure does not contain the origin.

By equation~\eqref{A and B for S-T}, for each $1 \le i \le s$ we have
\begin{equation} \label{B- limit}
\lim_{n\to\infty} \frac{B_\omega(n;Q_i-R_i)^2}{\log\log n} = \den(Q_i \sdif R_i) > 0.
\end{equation}
If we define $K = \frac 12 \min_{1 \le i \le s} \den(Q_i \sdif R_i)$, then each ${B_\omega(n;Q_i-R_i)^2}/\log\log n \ge K$ once~$n$ is large enough, and therefore eventually
\begin{equation} \label{Lnp2 bound}
|\mathbf L_n(p)|^2 = \sum_{i=1}^s \frac{1_{Q_i \sdif R_i}(p)}{B_\omega(n;Q_i-R_i)^2} \le \frac s{K\log\log n} \ll \frac1{\log\log n}.
\end{equation}
It follows that for any interval $\mathcal I \subseteq \R^s$ whose closure does not contain the origin, no vector of the form $\mathbf L_n(p)$ can lie in $\mathcal I$ once~$n$ is large enough. Since $Q(\mathcal I)=0$, we see that Proposition~\ref{prop:kubilius_multivariable}(b)(i) is satisfied.

Furthermore, by the triangle inequality,
\begin{align*}
\biggl| \sum_{p \le n} \frac{|\mathbf L_n(p)|^2}{p(1+|\mathbf L_n(p)|^2)} \mathbf L_n(p) \biggr| &\le \sum_{p \le n} \frac{|\mathbf L_n(p)|^2}{p(1+|\mathbf L_n(p)|^2)} |\mathbf L_n(p)| \\
&\le \sum_{p \le n} \frac{|\mathbf L_n(p)|^3}p \ll \sum_{p\le n} \frac{1/(\log\log n)^{3/2}}p \ll \frac1{(\log\log n)^{1/2}}
\end{align*}
by equation~\eqref{Lnp2 bound}. It follows that
\[
\lim_{n\to\infty} \sum_{p \le n} \frac{|\mathbf L_n(p)|^2}{p(1+|\mathbf L_n(p)|^2)} \mathbf L_n(p) = \mathbf 0 = -\mathbf \Gamma,
\]
which verifies Proposition~\ref{prop:kubilius_multivariable}(b)(ii).

Finally, for any fixed $\ep > 0$, equation~\eqref{Lnp2 bound} implies that $|\mathbf L_n(p)| < \ep$ for all primes~$p$ once~$n$ is large enough.
Thus for any $\mathbf T = (t_1,\dots,t_s) \in \R^s$,
\[
\lim_{n\to\infty} \sum_{\substack{p \le n \\ |\mathbf L_n(p)| < \ep}} \frac{(\mathbf T \cdot \mathbf L_n(p))^2}p = \lim_{n\to\infty} \sum_{p \le n} \frac{(\mathbf T \cdot \mathbf L_n(p))^2}p,
\]
and therefore
\begin{align*}
\lim_{\ep\to 0^+} \lim_{n\to\infty} \sum_{\substack{p \le n \\ |\mathbf L_n(p)| < \ep}} \frac{(\mathbf T \cdot \mathbf L_n(p))^2}p &= \lim_{n\to\infty} \sum_{p \le n} \frac{(\mathbf T \cdot \mathbf L_n(p))^2}p \\
&= \lim_{n\to\infty} \sum_{p \le n} \frac 1p \bigg( \sum_{i=1}^s \frac{1_{Q_i}(p)-1_{R_i}(p)}{B_\omega(n;Q_i-R_i)} t_i \bigg)^2 \\
&= \lim_{n\to\infty} \sum_{p \le n} \frac 1p \sum_{i=1}^s \sum_{j=1}^s
	\frac{(1_{Q_i}(p)-1_{R_i}(p))(1_{Q_j}(p)-1_{R_j}(p))}{B_\omega(n;Q_i-R_i)B_\omega(n;Q_j-R_j)} t_it_j \\
&= \lim_{n\to\infty} \sum_{i=1}^s \sum_{j=1}^s \frac{(\log\log n)^{1/2}}{B_\omega(n;Q_i-R_i)} \frac{(\log\log n)^{1/2}}{B_\omega(n;Q_j-R_j)} \\
	& \qquad{} \times \frac1{\log\log n} \sum_{p \le n} \bigg(\frac{(1_{Q_i}(p)-1_{R_i}(p))(1_{Q_j}(p)-1_{R_j}(p))}p\bigg) t_it_j \\
&= \sum_{i=1}^s \sum_{j=1}^s \den(Q_i \sdif R_i)^{-1/2} \den(Q_j \sdif R_j)^{-1/2} \\
	& \qquad{} \times \bigl( \den(S_1\cap S_2) + \den((T_1\cap T_2) - \den(S_1\cap T_2) - \den(T_1\cap S_2) \bigr) t_it_j 
\end{align*}
by equation~\eqref{B- limit} and Lemma~\ref{four sets lemma}.
Note that the right-hand side is the limit of nonnegative quantities and is therefore itself nonnegative.
Since this expression is equal to the quadratic form $\sigma(t_1,\dots,t_s)$ from equation~\eqref{quad form fraction}, we conclude that Proposition~\ref{prop:kubilius_multivariable}(b)(iii) holds as well.

We may now deduce that Proposition~\ref{prop:kubilius_multivariable}(a) holds, that is, the function $F(x_1,\ldots,x_s)$ given in equation~\eqref{big cumulative distribution} is the cumulative distribution function of a random vector~$\mathbf X$ with characteristic function given in equation~\eqref{char function before 0s}. However, since $\mathbf\Gamma = 0$ and $Q(\mathcal E)=0$ identically, equation~\eqref{char function before 0s} simplifies to $\log \chi_{\mathbf X}(t_1,\ldots,t_s) = -\tfrac 12 \sigma(t_1,\ldots,t_s)$, which completes the proof of the proposition. (As each component of the random vector $X$ has a normal distribution by Lemma~\ref{lemma:omega_diff_univariable}, there are no exceptional points, and so equation~\eqref{big cumulative distribution} holds for all $x_1,\dots,x_s\in\R$.)
\end{proof}

\section{Probability and $y$-typicality} \label{sec:typicality}

In this section, we prove Proposition~\ref{prop:y_typical_is_typical}, a quantitative bound that implies that $y$-typical numbers are indeed typical (that is, comprise a set of density~$1$) for any particular~$y$. Some technical work is necessary before we begin the proof itself.

\subsection{Background and technical lemmas}

The main goal of this section is to establish Lemmas~\ref{lemma:chebyshev_omega} and~\ref{lemma:few_large_omega new}, two technical results that are critical to the proof of Proposition~\ref{prop:y_typical_is_typical}. We first deduce some estimates for central moments of additive functions from the work of Granville and Soundararajan~\cite{gs07}, which we then apply to various expectations involving the additive functions $\omega(m;q,1)$.
\definitionEk
When~$a$ and~$b$ are positive quantities, we note the general estimate
\begin{equation} \label{Ex cheat}
\Ex(a+b)^k \ll_k \Ex a^k + \Ex b^k
\end{equation}
which follows from $(a+b)^k \le (2\max\{a,b\})^k = 2^k \max\{a^k,b^k\} \ll_k a^k+b^k$.

\begin{lemma}
\label{lemma:ek_sieve_specialized}
Let~$f$ be a real-valued strongly additive function with $|f(p)| \le M$ for all primes~$p$,
and define
\[
A_f(z) = \sum_{p\le z} \frac{f(p)}p \quad\text{and}\quad B^*_f(z)^2 = \sum_{p\le z} \frac{f(p)^2}p\bigg( 1-\frac1p \bigg) .
\]
Let $S_k = 2^{k/2} \Gamma (\frac{k+1}2) / \sqrt\pi$.
Then uniformly for all even natural numbers $k \le (B^*_f(z)/M)^{2/3}$,
\[
\sum_{m\le x} \biggl( \sum_{\substack{p \mid m \\ p\le z}} f(p)-A_f(z) \biggr)^k = S_k x B^*_f(z)^k \biggl(1+O\bigg(\frac{k^3M^2}{B^*_f(z)^2}\bigg)\biggr)+O\biggl( M^k \biggl( \sum_{p\le z} \frac1p \biggr)^k \frac{\pi(z)^k}{k!} \biggr) ;
\]
and uniformly for all odd natural numbers $k \le (B^*_g(z)/M)^{2/3}$,
\[
\sum_{m\le x} \biggl( \sum_{\substack{p \mid m \\ p\le z}} f(p)-A_f(z)  \biggr)^k \ll S_k x B^*_f(z)^{k-1} k^{3/2}M + M^k \biggl( \sum_{p\le z} \frac1p \biggr)^k \frac{\pi(z)^k}{k!} .
\]
\end{lemma}

\begin{proof}
This lemma is a restatement, using notation consistent with the rest of this paper, of~\cite[Proposition~4]{gs07} in the special case where $\mathcal A = [1,x]\cap\Z$ and $\mathcal P$ is the set of primes not exceeding~$z$.
The standard sieve notation in that proposition simplifies greatly, as in this situation $h(d)=1$ identically and $|r_d| < 1$ for all~$d$.
\end{proof}

\begin{lemma}
\label{lemma:ek_sieve_specialized cor}
Let~$f$ be a real-valued strongly additive function with $|f(p)| \le M$ for all primes~$p$, and let $k\in\N$ and $n\in\N$. If $z\le n^{1/k}$ is such that $B_f(z)^2 \ge 2k^3M^2$, then
\[
\sum_{m\le n} \biggl| \sum_{\substack{p \mid m \\ p\le z}} f(p)-A_f(z) \biggr|^k \ll_{M,k} n B^*_f(z)^k,
\]
and therefore
\[
\Ex_n \biggl| \sum_{\substack{p \mid m \\ p\le z}} f(p)-A_f(z) \biggr|^k \ll_{M,k} B_f(z)^k.
\]
\end{lemma}

\begin{proof}
We concentrate on the first claim, since the second claim follows directly from it and the inequality $B^*_f(z) \le B_f(z)$.
Note that the hypotheses imply that $k \le (B^*_f(z)/M)^{2/3}$ since $B^*_f(z)^2 \ge \frac12B_f(z)^2$. When~$k$ is even, the lemma follows directly from Lemma~\ref{lemma:ek_sieve_specialized} since $k^3M^2/B^*_f(z)^2 \le 1$ and
\[
\biggl( \pi(z) \sum_{p\le z}\frac1p \biggr)^k \ll_k \biggl( \frac{z\log\log z}{\log z} \biggr)^k \ll_k n \ll_{M,k} n B^*_f(z)^k.
\]
Then when~$k$ is odd, the lemma follows from the even case via the Cauchy--Schwarz inequality
\[
\biggl( \sum_{m\le n} \biggl| \sum_{\substack{p \mid m \\ p\le z}} f(p)-A_f(z) \biggr|^k \biggr)^2 \le \sum_{m\le n} \biggl| \sum_{\substack{p \mid m \\ p\le z}} f(p)-A_f(z) \biggr|^{k-1} \sum_{m\le n} \biggl| \sum_{\substack{p \mid m \\ p\le z}} f(p)-A_f(z) \biggr|^{k+1}.
\qedhere
\]
\end{proof}

\begin{lemma}
\label{lemma:ek_sieve_big_O new}
Let~$f$ be a real-valued strongly additive function with $|f(p)| \le M$ for all primes~$p$, and let $k \in \N$ and $\ep > 0$.
For all $n\ge3$, if $B_f(n)^2 \ge \frac12$ then
\[
\Ex_n|f(m)-A_f(n)|^k \ll _{M,k} B_f(n)^k .
\]
\end{lemma}

\begin{proof}
Define $f(m;z) = \sum_{p\mid m,\, p\le z} f(p)$.
By the estimate~\eqref{Ex cheat}, we may write
\begin{multline} \label{three way split}
\Ex_n|f(m)-A_f(n)|^k \ll_k \Ex_n|f(m)-f(m;n^{1/k})|^k \\
+ \Ex_n|f(m;n^{1/k})-A_f(n^{1/k})|^k + \Ex_n|A_f(n)-A_f(n^{1/k})|^k.
\end{multline}
The first summand is bounded in terms of~$M$ and~$k$ because of the pointwise bound
\[
f(m)-f(m;n^{1/k}) = \sum_{\substack{p\mid m \\ p>n^{1/k}}} f(p) \le M(k-1),
\]
while the third summand is also bounded since
\begin{align}
A_f(n)-A_f(n^{1/k}) = \sum_{n^{1/k}<p\le n} \frac{f(p)}p &\ll M \sum_{n^{1/k}<p\le n} \frac1p \notag \\
&\ll_M \log\log n-\log\log n^{1/k} + O\biggl( \frac1{\log n^{1/k}} \biggr) \ll_M k.
\label{A bd B bd}
\end{align}
Consequently, equation~\eqref{three way split} becomes
\[
\Ex_n|f(m)-A_f(n)|^k \ll_{M,k} \Ex_n|f(m;n^{1/k})-A_f(n^{1/k})|^k + 1.
\]
If $B_f(n^{1/k})^2 \ge 2k^3M^2$, then Lemma~\ref{lemma:ek_sieve_specialized cor} yields
\[
\Ex_n|f(m)-A_f(n)|^k \ll_{M,k} B_f(n)^2 + 1 \ll_{M,k} B_f(n)^2
\]
as desired; in particular, this establishes the lemma when $B_f(n)^2$ is sufficiently large in terms of~$M$ and~$k$, since $B_f(n^{1/k})^2 = B_f(n)^2 + O(M^2k)$ by a calculation analogous to equation~\eqref{A bd B bd}.
But since the implicit constant in the lemma is allowed to depend on~$M$ and~$k$, we can extend the range of validity down to $B_f(n)^2 \ge \frac12$ as claimed.
\end{proof}

\begin{remark}
In all applications of Lemma~\ref{lemma:ek_sieve_big_O new} in this work, we will have $|f(p)| \le 1$ for all primes~$p$; consequently, the dependence on~$M$ in Lemma~\ref{lemma:ek_sieve_big_O new} will be omitted entirely from this point onwards.
\end{remark}

We now have the tools required to establish the first lemma required for the proof of Proposition~\ref{prop:y_typical_is_typical}, which provides an upper bound on the likelihood of $\omega(m;q,1)$ being far from its expected value.

\begin{lemma}
\label{lemma:chebyshev_omega}
Let~$q$ be a prime power and let~$\delta$ be a positive real number. There exists an absolute constant $C>0$ such that if $\log\log n \ge C\varphi(q)$, then for every integer $r\ge1$,
\[
P_n\bigl( |\omega(m;q,1)-\varphi(q)^{-1}\log\log n| > \delta\log\log n \bigr) \ll_{r} \frac 1{\delta^{2r}(\varphi(q)\log\log n)^r}.
\]
\end{lemma}

\begin{proof}
First note that by Lemma~\ref{lemma:sum_arithmetic_progression_pnt},
\[
B_\omega(n;q,1)^2 = \frac{\log\log n + O(\log q)}{\varphi(q)} \asymp \frac{\log\log n}{\varphi(q)}
\]
once $(\log\log n)/\log q$ is sufficiently large, which can be guaranteed by the assumption $\log\log n \ge C\varphi(q)$.
By Markov's inequality,
\begin{align}
	P_n\bigl( & |\omega(m;q,1)-\varphi(q)^{-1} \log\log n| > \delta\log\log n \bigr) \notag \\
	& = P_n\bigl( (\omega(m;q,1)-\varphi(q)^{-1}\log\log n)^{2r} > (\delta\log\log n)^{2r} \bigr) \notag \\
	& \le \frac{\Ex_n(\omega(m;q,1)-\varphi(q)^{-1}\log\log n)^{2r}}{(\delta\log\log n)^{2r}}.
\label{cheb omega markov}
\end{align}
Since $\varphi(q)^{-1}\log\log n = A_\omega(n;q,1) + O(1)$ by Lemma~\ref{lemma:sum_arithmetic_progression_pnt}, we see that
\begin{align*}
{\Ex_n(\omega(m;q,1)-\varphi(q)^{-1}\log\log n)^{2r}} &= {\Ex_n \bigl( \omega(m;q,1)-A_\omega(n;q,1)+O(1) \bigr)^{2r}} \\
& \ll_r \Ex_n(\omega(m;q,1)-A_\omega(n;q,1))^{2r} + 1 \\
& \ll_r B_\omega(n;q,1)^{2r} + 1 \ll_r \frac{(\log\log n)^r}{\varphi(q)^r}
\end{align*}
by equation~\eqref{Ex cheat} and Lemma~\ref{lemma:ek_sieve_big_O new} and the assumption $\log\log n \ge C\varphi(q)$. Combining this esimate with the bound~\eqref{cheb omega markov} completes the proof of the lemma.
\end{proof}

We will need further results of a similar type, bounding probabilities that the $\omega(m;q,1)$ are unexpectedly large. Part~(c) of the next lemma will be useful in particular when the modulus~$q$ is substantial.

\begin{lemma}
\label{lemma:few_large_omega new pre}
Let $0<\lambda\le1$ and~$a$ be real numbers, let~$q$ be a prime power, and let $r\ge1$ and $n\ge3$ be integers.
\begin{enumerate}
\item If $\log q \ll \log \log n$ and $B_\omega(n;q,1)^2 \ge \frac12$ and $\lambda>1/\varphi(q)$, then
\[
P_n\bigl( \omega(m;q,1) + a > \lambda\log\log n \bigr) \ll_{a,r} \frac{(\lambda-1/\varphi(q))^{-2r}}{(q\log\log n)^r}.
\]
\item If $B_\omega(n;q,1)^2 < \frac12$, then $P_n\bigl( \omega(m;q,1) + a > \lambda\log\log n \bigr) \ll_{a} (\log n)^{-\lambda\log2}$.
\item If $\lambda\log\log n > 1+a$, then
\[
P_n\bigl( \omega(m;q,1) + a > \lambda\log\log n \bigr) \ll \frac{(\log\log n)^2 + (\log q)^2}{q^2}.
\]
\end{enumerate}
\end{lemma}

\begin{proof}
We first establish part~(a). Since $\lambda>1/\varphi(q)$, we note that by Lemma~\ref{lemma:sum_arithmetic_progression_pnt},
\begin{align}
\lambda_n\log\log n - A_\omega(n;q,1) - a
&= \lambda_n\log\log n - \biggl( \frac{\log\log n}{\varphi(q)} + O_a(1) \biggr) \notag \\
&\gg_a \Big(\lambda_n-\frac 1{\varphi(q)}\Big)\log\log n, \label{case a}
\end{align}
which is positive when~$n$ is sufficiently large in terms of~$a$.
Hence by Markov's inequality,
\begin{align} \label{case a 2}
P_n\bigl( \omega(m;q,1) & {}+a > \lambda_n\log\log n \bigr) \\
&= P_n\bigl( \omega(m;q,1) - A_\omega(n;q,1) > \lambda_n\log\log n - A_\omega(n;q,1) - a) \bigr) \notag \\
&= P_n\bigl( (\omega(m;q,1) - A_\omega(n;q,1))^{2r} > (\lambda_n\log\log n - A_\omega(n;q,1) - a))^{2r} \bigr) \notag \\
& \le \frac{\Ex_n\bigl(\omega(m;q,1)-A_\omega(n;q,1)\bigr)^{2r}}{(\lambda_n\log\log n - A_\omega(n;q,1) - a)^{2r}} \ll_r \frac{B_\omega(n;q,1)^{2r}}{(\lambda_n\log\log n - A_\omega(n;q,1) - a)^{2r}} \notag
\end{align}
by Lemma~\ref{lemma:ek_sieve_big_O new} (valid since $B_\omega(n;q,1)^2 \ge \frac12$). Moreover, since $\log q \ll \log \log n$, Lemma~\ref{lemma:sum_arithmetic_progression_pnt} implies that
\begin{align*}
B_\omega(n;q,1)^2 \ll \frac{\log\log n}{\varphi(q)} + \frac{\log q}{\varphi(q)} \ll \frac{\log\log n}q,
\end{align*}
since $\varphi(q) \ge \frac q2$ when~$q$ is a prime power. Using this estimate and the lower bound~\eqref{case a}, the upper bound~\eqref{case a} becomes
\[
P_n\bigl( \omega(m;q,1)+a > \lambda_n\log\log n \bigr) \ll_{a,r} \frac{(\log\log n)^r/q^r}{(\lambda-1/\varphi(q))^{2r}(\log\log n)^{2r}},
\]
which is the same as the bound stated in part~(a). This proof required that~$n$ be sufficiently large in terms of~$a$, but that assumption can be omitted since the implicit constant may depend on~$a$.

For part~(b), we first prove by induction on~$\ell$ that $\#\bigl\{ m\le n\colon \omega(m;q,1) = \ell \bigr\} \le (\frac12)^{\ell-1} n$ for all integers $n\ge3$ and $\ell\ge1$. The base case $\ell=1$ is trivial. When $\ell\ge1$,
\begin{align*}
\#\bigl\{ m\le n\colon \omega(m;q,1) = \ell+1 \bigr\}
& \le \sum_{\substack{p \le n \\ p\equiv1\mod q}} \#\bigl\{ m\le \lfloor \tfrac np \rfloor \colon \omega(m;q,1) = \ell \bigr\} \\
& \le \sum_{\substack{p \le n \\ p\equiv1\mod q}} \frac1{2^{\ell-1}} \biggl\lfloor \frac np \biggr\rfloor \\
& \le \frac n{2^{\ell-1}} \sum_{\substack{p \le n \\ p\equiv1\mod q}} \frac1p = \frac n{2^{\ell-1}} B_\omega(n;q,1)^2 < \frac n{2^\ell}
\end{align*}
since $B_\omega(n;q,1)^2 < \frac12$.
Consequently,
\begin{align*}
P_n\bigl( \omega(m;q,1) + a > \lambda_n\log\log n \bigr)
&= \frac1n \sum_{\ell > \lambda_n\log\log n - a} \#\bigl\{ m\le n\colon \omega(m;q,1) = \ell \bigr\} \\
& \le \frac1n \sum_{\ell > \lambda_n\log\log n - a} \frac n{2^\ell} \ll 2^{-(\lambda_n\log\log n - a)} \ll_a \frac 1{(\log n)^{\lambda_n\log 2}}
\end{align*}
as claimed.

Finally, when $\lambda_n\log\log n > 1+a$,
\[
P_n\bigl( \omega(m;q,1)+a > \lambda_n\log\log n \bigr) \le P_n\bigl( \omega(m;q,1) > 1 \bigr) = P_n\bigl( \omega(m;q,1) \ge 2 \bigr) = \frac 1n \sum_{\substack{m\le n \\ \omega(m;q,1)\ge2}} 1.
\]
Then by Lemma~\ref{lemma:sum_arithmetic_progression_pnt},
\begin{align*}
\frac 1n \sum_{\substack{m\le n \\ \omega(m;q,1)\ge2}} 1 
& \le \frac 1n \sum_{\substack{p_1 < p_2 \le n \\ p_1,p_2 \equiv 1\mod q}} \sum_{\substack{m\le n \\ p_1p_2\mid m}} 1 \\
& \le \frac 1n \sum_{\substack{p_1, p_2 \le n \\ p_1,p_2 \equiv 1\mod q}} \frac n{p_1p_2} = \biggl( \sum_{\substack{p \le n \\ p \equiv1\mod q}} \frac 1p \biggr)^2 \ll \biggl( \frac{\log\log n}{\varphi(q)} + \frac{\log q}{\varphi(q)} \biggr)^2,
\end{align*}
which implies the estimate in part~(c) since again $\varphi(q) \ge \frac q2$.
\end{proof}

\begin{remark}
The estimates in part~(b) are certainly not optimal, but the result suffices for our needs; the proof is effectively a simplified variant of a technique employed by Hardy and Ramanujan.
The proof for part~(c) is a variant of a technique employed by Chang and the first author in~\cite{SIFMG}.
\end{remark}

We can now establish the second lemma required for the proof of Proposition~\ref{prop:y_typical_is_typical}.

\begin{lemma}
\label{lemma:few_large_omega new}
Let $0<\lambda\le1$ and~$a$ and $r\ge1$ be real numbers, and let $n\ge3$ be an integer.
Define $L(\lambda)$ to be the minimum positive value of $\lambda-1/\varphi(q)$ as~$q$ ranges over prime powers.
If $\lambda\log\log n > 1+a$ and $L(\lambda) > (\log\log n)^{-1/2}$, then
\begin{multline} \label{eqn:few_large_omega new}
P_n\bigl( \text{there exists a prime power $q$ with $1/\varphi(q) < \lambda$ and } \omega(m;q,1) + a > \lambda\log\log n \bigr) \\
\ll_{a,r} \frac{L(\lambda)^{-2r}}{(\log\log n)^r}.
\end{multline}
\end{lemma}

\begin{proof}
For now we assume that $r\ge2$ is an integer.
By the union bound and Lemma~\ref{lemma:few_large_omega new pre}(iii), the probability on the left-hand side of equation~\eqref{eqn:few_large_omega new} is at most
\begin{align} \label{eqn:few_large_omega together}
\sum_{\substack{q \\ 1/\varphi(q) < \lambda}} P_n\bigl( \omega(m;q,1) + a > \lambda\log\log n \bigr) .
\end{align}
Letting $\nu$ be a parameter, the contribution to this sum of the prime powers exceeding~$\nu$~is
\begin{align*}
\sum_{\substack{q > \nu \\ 1/\varphi(q) < \lambda}} P_n\bigl( \omega(m;q,1) + a > \lambda\log\log n \bigr) &\ll \sum_{q>\nu} \biggl( \frac{(\log\log n)^2}{q^2} + \frac{(\log q)^2}{\varphi(q)^2} \biggr) \\
&\ll \frac{(\log\log n)^2}\nu + \frac{(\log \nu)^2}\nu.
\end{align*}
The contribution to the upper bound~\eqref{eqn:few_large_omega together} of the remaining prime powers can be estimated~by
\begin{align*}
\sum_{\substack{q \le \nu \\ 1/\varphi(q) < \lambda \\ B_\omega(n;q,1)^2 \ge 1/2}} & P_n\bigl( \omega(m;q,1) + a > \lambda\log\log n \bigr) + \sum_{\substack{q \le \nu \\ 1/\varphi(q) < \lambda \\ B_\omega(n;q,1)^2 < 1/2}} P_n\bigl( \omega(m;q,1) + a > \lambda\log\log n \bigr) \\
&\ll_{a,r} \sum_{\substack{q \le \nu \\ 1/\varphi(q) < \lambda}} \frac{(\lambda-1/\varphi(q))^{-2r}}{(q\log\log n)^r} + \sum_{q \le \nu} \frac1{(\log n)^{\lambda\log2}} \\
&\le L(\lambda)^{-2r} \sum_q \frac1{(q\log\log n)^r} + \sum_{q \le \nu} \frac1{(\log n)^{\lambda\log2}} \ll \frac{L(\lambda)^{-2r}}{(\log\log n)^r} + \frac\nu{(\log n)^{\lambda\log2}}
\end{align*}
by Lemma~\ref{lemma:few_large_omega new pre}(i)--(ii). If we set $\nu=(\log\log n)^{r+2}$, then these contributions total at most
\begin{multline*}
\frac{L(\lambda)^{-2r}}{(\log\log n)^r} + \frac\nu{(\log n)^{\lambda\log2}} + \frac{(\log\log n)^2}\nu + \frac{(\log \nu)^2}\nu \\
\ll \frac{L(\lambda)^{-2r}}{(\log\log n)^r} + \frac{(\log\log n)^{r+2}}{(\log n)^{\lambda\log2}} + \frac1{(\log\log n)^r}.
\end{multline*}
Since $L(\lambda) < \lambda \le 1$, the first fraction on the right-hand side dominates the last fraction; moreover, since $\lambda > L(\lambda) > (\log\log n)^{-1/2}$, the first fraction also dominates the second fraction (with an implicit constant depending on~$r$, which is acceptable).

We have therefore established the upper bound~\eqref{eqn:few_large_omega new} for all integers $r\ge2$; but the fact that the upper bound is a decreasing function of~$r$ then implies that the upper bound holds for all real numbers~$r\ge1$.
\end{proof}

With these technical results completed, the proof of Proposition~\ref{prop:y_typical_is_typical} can begin.

\subsection{Proof that $y$-typical numbers are typical} \label{subsection:y_typical_proof}

Before starting, we observe that the condition that $\prm(m;q_1) \le \prm(m;q_2)$
is, by Lemma~\ref{lemma:prime_multiples}, equivalent to
$\omega(m;q_1,1) + \bar E(m;q_1) \le \omega(m;q_2,1) + \bar E(m;q_2)$;
by the definition of $\bar E(m;q)$ in that lemma, it follows that
\begin{equation} \label{from prm to omega}
\prm(m;q_1) \le \prm(m;q_2) \implies \omega(m;q_1,1) \le \omega(m;q_2,1) + 2.
\end{equation}
We also introduce the new notation $\varphi_-(\tq_i)$ for the largest totient value less than $\varphi(\tq_i)$ and
$\varphi_+(\tq_i)$ for the smallest totient value greater than $\varphi(\tq_i)$. (So $\varphi_-(\tq_i) = \varphi(\tq_{i-1})$ if $\varphi(\tq_{i-1}) \ne \varphi(\tq_i)$ while $\varphi_-(\tq_i) = \varphi(\tq_{i-2})$ if $\varphi(\tq_{i-1}) = \varphi(\tq_i)$, and analogously for $\varphi_+(\tq_i)$.) Define
$\delta_1=\frac16$ and for $i\ge2$,
\begin{align} \label{delta gap def}
\delta_i &= \frac 13 \min\biggl\{ \frac 1{\varphi_-(\tq_i)}-\frac 1{\varphi(\tq_i)}, \frac 1{\varphi(\tq_i)} - \frac 1{\varphi_+(\tq_i)} \biggr \}.
\end{align}

\begin{lemma} \label{two bad cases}
Let $y\ge1$ be a real number, and let $n\in\N$ satisfy $\log\log n \ge 6y^2$. Suppose that~$m$ is not $y$-typical. Then one of the following conditions must hold:
\begin{enumerate}
\item There exists $1\le i\le \iW(y)$ such that $\displaystyle \biggl| \omega(m;\tq_i,1) - \frac1{\varphi(\tq_i)}\log\log n \biggr| > \delta_i\log\log n$;
\item There exists a prime power~$q$ with $\varphi(q)>y$ such that
\[
\omega(m;q,1)+2 > \biggl( \frac1{3\varphi(\tq_N)} + \frac2{3\varphi(\tq_{N+1})} \biggr) \log\log n.
\]
\end{enumerate}
\end{lemma}

\begin{proof}
We prove the contrapositive, supposing that neither~(a) nor~(b) holds and proving that~$m$ is $y$-typical. Set $N=\iW(y)$. First note that if~(a) does not hold, then for $2\le i\le N$,
\begin{align} \label{2313}
\biggl( \frac2{3\varphi(\tq_i)} + \frac1{3\varphi_+(\tq_i)} \biggr) \log\log n \le \omega(m;\tq_i,1) \le \biggl( \frac2{3\varphi(\tq_i)} + \frac1{3\varphi_-(\tq_i)} \biggr) \log\log n
\end{align}
If $\varphi(\tq_i) < \varphi(\tq_j)$ then $\varphi(\tq_i) \le \varphi_-(\tq_j)$ and $\varphi_+(\tq_i) \le \varphi(\tq_j)$ and thus
\begin{align*}
\omega(m;\tq_i,1) - \omega(m;\tq_j,1) &\ge \biggl( \biggl( \frac2{3\varphi(\tq_i)} + \frac1{3\varphi(\tq_j)} \biggr) - \biggl( \frac2{3\varphi(\tq_j)} + \frac1{3\varphi(\tq_i)} \biggr) \biggr) \log\log n \\
&= \frac13\biggl( \frac1{\varphi(\tq_i)} + \frac1{\varphi(\tq_j)} \biggr) \log\log n \\
&\ge \frac13\frac1{\varphi(\tq_i) \varphi(\tq_j)} \log\log n > \frac1{3y^2}\log\log n \ge 2,
\end{align*}
which implies that $\prm(m;\tq_i) > \prm(m;\tq_j)$ by equation~\eqref{from prm to omega}.
The special case $i=1$ is similar: if~(a) does not hold then $\omega(m;q_2,1) \ge 1\log\log n - \frac16\log\log n$ and so
\begin{align*}
\omega(m;\tq_1,1) - \omega(m;\tq_j,1) &\ge \biggl( \frac56 - \biggl( \frac2{3\varphi(\tq_j)} + \frac1{3\varphi_-(\tq_j)} \biggr) \biggr) \log\log n \\
&\ge \biggl( \frac56 - \biggl( \frac2{3\cdot2} + \frac1{3\cdot1} \biggr) \biggr) \log\log n = \frac16\log\log n \ge 2,
\end{align*}
which again implies that $\prm(m;\tq_1) > \prm(m;\tq_j)$. Since $\tq_1,\dots,\tq_N$ are all the prime powers with totients not exceeding~$y$, we have verified the first half of Definition~\ref{defn:y_typical} of $y$-typicality.

The first inequality in equation~\eqref{2313} implies that when~(a) does not hold,
\begin{align*}
\omega(m;\tq_N,1) &\ge \biggl( \frac2{3\varphi(\tq_N)} + \frac1{3\varphi_+(\tq_N)} \biggr) \log\log n \\
&= \biggl( \frac2{3\varphi(\tq_N)} + \frac1{3\varphi(\tq_{N+1})} \biggr) \log\log n > \biggl( \frac1{3\varphi(\tq_N)} + \frac2{3\varphi(\tq_{N+1})} \biggr) \log\log n.
\end{align*}
In particular, when~(b) does not hold either, we see that $\omega(m;\tq_N,1) > \omega(m;q,1) + 2$ for every prime power~$q$ whose totient value exceeds~$y$, and thus $\prm(m;\tq_N) > \prm(m;q)$ by equation~\eqref{from prm to omega}. The exact same argument shows that if $\varphi(\tq_{N-1}) = \varphi(\tq_N)$ then $\prm(m;\tq_{N-1}) > \prm(m;q)$ for all such prime powers.

On the other hand, the first half of the definition of $y$-typicality, which we have already established, implies that $\prm(m;\tq_i) \ge \prm(m;\tq_N)$ for all $1\le i\le N-2$, and also that $\prm(m;\tq_{N-1}) \ge \prm(m;\tq_N)$ if $\varphi(\tq_{N-1}) < \varphi(\tq_N)$. In summary, we have shown that $\prm(m;\tq_i) \ge \prm(m;q)$ for all $1\le i\le N$ and all~$q$ whose totient values exceed~$y$, which verifies the second half of Definition~\ref{defn:y_typical} of $y$-typicality as well. 
\end{proof}

We are now ready to prove Proposition~\ref{prop:y_typical_is_typical}, which asserts that for all real numbers $y \ge 2$ and $r \ge 1$ and for all integers $n \in \N$,
\[
P_n(m \text{ is not $y$-typical}) \ll_r \frac{y^{4r}}{(\log\log n)^r}.
\]

\begin{proof}[Proof of Proposition~\ref{prop:y_typical_is_typical}]
We may assume that $\log\log n$ is larger than $y^4$ times some constant depending on~$r$, for otherwise the statement is trivial. Set $N=\iW(y)$, define $\delta_1=\frac16$, and for $2\le i\le N$ let~$\delta_i$ be defined as in equation~\eqref{delta gap def}. Further define
\[
\lambda = \frac1{3\varphi(\tq_N)} + \frac2{3\varphi(\tq_{N+1})}
\]
and
\begin{equation} \label{Q_j def}
\TQ_j = \bigl\{ p \colon p \equiv 1 \mod{\tq_j} \bigr\} ,
\quad\text{so that } \den(\TQ_j) = 1/\varphi(\tq_j).
\end{equation}
Note that $1/\varphi(\tq_N) > \lambda > 1/\varphi(\tq_{N+1})$ and that $\varphi(\tq_N) \le y < \varphi(\tq_{N+1})$, and so $\varphi(q)>y$ is equivalent to $1/\varphi(q) < \lambda$ for prime powers~$q$. Therefore by Lemma~\ref{two bad cases},
\begin{multline} \label{applying bad cases}
P_n(m \text{ is not $y$-typical}) \le \sum_{i=1}^N P_n(|\omega(m;\TQ_i)-\varphi(\tq_i)^{-1}\log\log n| > \delta_i\log\log n) \\
+ P_n\bigl( \text{there exists a prime power $q$ with $1/\varphi(q) < \lambda$ and } \omega(m;q,1) + 2 > \lambda\log\log n \bigr),
\end{multline}
and it suffices to show that both terms on the right-hand side are $\ll_r {y^{4r}}{(\log\log n)^r}$.

We may apply Lemma~\ref{lemma:chebyshev_omega} to each term in the sum, since $\tq_i \le y \ll \log\log n$. By the inequality $\delta_i \gg 1/\varphi(\tq_i)^2$,
\begin{align*}
\sum_{i=1}^N & P_n(|\omega(m;\TQ_i)-\varphi(\tq_i)^{-1}\log\log n| > \delta_i\log\log n) \\
& \ll_r \sum_{i=1}^N \frac 1{\delta_i^{2r}(\varphi(\tq_i)\log\log n)^r} \\
& \ll_r \frac 1{(\log\log n)^r} \sum_{i=1}^N \varphi(\tq_i)^{3r} \le \frac 1{(\log\log n)^r} \cdot 2\sum_{m\le y} m^{3r} \ll_r \frac{y^{3r+1}}{(\log\log n)^r} ,
\end{align*}
which is sufficient since $r\ge1$.
On the other hand, we may apply Lemma~\ref{lemma:few_large_omega new} to the last term in equation~\eqref{applying bad cases}, since $\lambda\log\log n \gg y^{-1}\log\log n > 3$ and
\[
L(\lambda) = \lambda - \frac1{\varphi(\tq_{N+1})} = \frac13\biggl( \frac1{\varphi(\tq_N)} - \frac1{\varphi(\tq_{N+1})} \biggr) \ge \frac1{3\varphi(\tq_N)\varphi(\tq_{N+1})} \gg y^{-2},
\]
which implies that $L(\lambda) > (\log\log n)^{-1/2}$
since we are assuming that $\log\log n$ is at least a sufficiently large multiple of~$y^4$.
Lemma~\ref{lemma:few_large_omega new} thus yields the bound
\begin{multline*}
P_n\bigl( \text{there exists a prime power $q$ with $1/\varphi(q) < \lambda$ and } \omega(m;q,1) + 2 > \lambda\log\log n \bigr) \\
\ll_r \frac{L(\lambda)^{-2r}}{(\log\log n)^r} \ll_r \frac{y^{4r}}{(\log\log n)^r},
\end{multline*}
which completes the proof.
\end{proof}

\section{Probability distributions of the invariant factors} \label{sec:distributions}

In this section, the probability distributions of the invariant factors are characterized.
Section~\ref{classifying distributions section} contains the lemmas that convert certain relevant
distributions described in terms of their characteristic
function into cumulative distribution functions, and discusses some of their properties.
Section~\ref{phi-sequences and probability section} builds on the first to provide technical lemmas needed to prove the theorems in Section~\ref{distribution intro section}.
In Section~\ref{proofs section}, the theorems in Section~\ref{distribution intro section} are proved.
[ALSO Theorem~\ref{theorem:order_restriction_finite}]

\subsection{Lemmas for classifying the distributions} \label{classifying distributions section}

The key goal of this section is to find the distribution functions of maximum and minimum values of
correlated multivariate normal distributions. The first step is to provide notation for discussing
these maximums and minimums in terms of probability distributions:

\begin{defn}
\label{defn:min_max_distributions}
\definitionMnMx
\end{defn}

As these minimums and maximums are expressable as finite sums,
they behave nicely in terms of limits. In particular:

\begin{prop}
\label{prop:min_max_distributions}
Let $(X_1,\ldots,X_n)$ be a random vector with multivariate distribution function~$F$. Then
\begin{align*}
\Mx_F(x) &= P(\max\{X_1,\ldots,X_n\} \le x) \\
\Mn_F(x) &= P(\min\{X_1,\ldots,X_n\} \le x) 
\end{align*}
are the distribution functions of the random variables $\max\{X_1,\ldots,X_n\}$ and $\min\{X_1,\ldots,X_n\}$, respectively.
Furthermore, if there exists a sequence of multivariate probability distributions such that $\lim_{n\to\infty} F_n = F$, then
\begin{align*} 
\lim_{n\to\infty} \Mx_{F_n}(x) &= \Mx_F(x) , \\
\lim_{n\to\infty} \Mn_{F_n}(x) &= \Mn_F(x) .
\end{align*}
\end{prop}

\begin{proof}
Observe that as $F$ is the multivariate distribution function of $X_1,\ldots,X_n$,
\begin{align*}
P(\max\{X_1,\ldots,X_n\} \le x)
&= P(X_1 \le x, \ldots, X_n \le x) \\
&= F(x,\ldots,x) = \Mx_F(x).
\end{align*}
Similarly, as $F$ is the multivariate distribution function of $X_1,\ldots,X_n$ and
by application of the inclusion-exclusion principle, it follows 
\begin{align*}
P(\min\{X_1,\ldots,X_n\} \le x)
&= P((X_1 \le x) \cup \ldots \cup (X_n \le x)) \\
&= \sum_{k=1}^N \sum_{\substack{S \subset \{1,\ldots,N\} \\ \# S = k}} (-1)^{k-1} P(X_s \le x \colon s \in S) \\
&= \sum_{k=1}^N (-1)^{k-1} \sum_{\substack{S \subset \{1,\ldots,N\} \\ \# S = k}} F(\vec x_S)
= \Mn_F(x) .
\end{align*}

For the final statement, observe that
\[ \lim_{n\to\infty} \Mx_{F_n}(x) = \lim_{n\to\infty} F_n(x,\ldots,x) = F(x,\ldots,x) = \Mx_F(x) . \]
Likewise,
\begin{align*}
\lim_{n\to\infty} \Mn_{F_n}(x)
&= \lim_{n\to\infty} \sum_{k=1}^N (-1)^{k-1} \sum_{\substack{S \subset \{1,\ldots,N\} \\ \#S = k}} F_n(\vec x_S) \\
&= \sum_{k=1}^N (-1)^{k-1} \sum_{\substack{S \subset \{1,\ldots,N\} \\ \#S = k}} F(\vec x_S)
= \Mn_F(x) .
\qedhere
\end{align*}
\end{proof}

We now characterize the bivariate distributions that will arise
in the proofs of the theorems in Section~\ref{distribution intro section}.
The following proofs require descriptions of the characteristic functions and distribution functions of sums of independent normal variables and truncated normal variables. Most of these facts are routine within probability theory; two lemmas of this type are stated in Appendix~\ref{appendix:distributions} for clarity.


In the following proofs, we use the fact that sums of random variables whose joint distribution is a multivariate normal distribution
which correspond to the eigenvectors of the covariance matrix
are independent.
For example, if $X,Y$ have a joint bivariate normal distribution, and $[ 1, 1 ]$ and $[ 1,-1 ]$ are eigenvectors
of the covariance matrix, then the sums $X+Y$ and $X-Y$ are independent random variables.
This assertion can be inferred from~\cite[pp.~27--28]{multivariate_normal}, which relates
the standard normal multivariate distribution to arbirary multivariate normal distributions. In doing so,
it demonstrates that these eigenvector-matching sums must be independent because they are the independent
random variables of the standard normal multivariate distribution after a change of basis.

\begin{lemma}
\label{lemma:bivariate_min_max}
Suppose $F(x,y)$ is a normal bivariate distribution function with invertible covariance matrix
\[ M = \begin{bmatrix} 1 & a \\ a & 1 \end{bmatrix} . \]
Then
\begin{align*}
\Mx_F(x) &= \Phi\Big(x;\frac{1-a}{\sqrt{1-a^2}}\Big) = \Phi(x) - 2\OT\biggl(x,\sqrt{\frac{1-a}{1+a}}\biggr) \\
\Mn_F(x) &= \Phi\Big(x;-\frac{1-a}{\sqrt{1-a^2}}\Big) = \Phi(x) + 2\OT\biggl(x,\sqrt{\frac{1-a}{1+a}}\biggr)
\end{align*}
with corresponding characteristic functions
\begin{align*}
\chi_{\Mx_F}(t) = e^{-t^2/2}\bigg(1+\eta\Big(\frac{\sqrt{1-a}}2 t\Big)\bigg) \\
\chi_{\Mn_F}(t) = e^{-t^2/2}\bigg(1-\eta\Big(\frac{\sqrt{1-a}}2 t\Big)\bigg).
\end{align*}

\end{lemma}

\begin{proof} 
Let $X,Y$ be any random variables with a joint distribution of $F$, so that $F(x,y) = P(X \le x, Y \le y)$.
Since the matrix $M$ is invertible, there exists a probability density function
\[
f(\vec x) = \frac{\exp -\frac 12 \vec x^\intercal M^{-1} \vec x}{2\pi \sqrt{|M|}}
\]
for this bivariate normal distribution \cite[Definition 3.2.1, p.~26]{multivariate_normal}.

Observe that $M$ has eigenvectors $\vec e_1 = \frac 1{\sqrt 2} [1,1]$ and $\vec e_2 = \frac 1{\sqrt 2}[1,-1]$ with respective eigenvalues
$\lambda_1 = (1+a)$, $\lambda_2 = (1-a)$.
Thus $X+Y$ and $X-Y$ are independent random variables as discussed earlier.
Set $F_k(x) = P((X,Y)\cdot \vec e_k \le x)$ and $A_k(x) = \{ \vec y \in \R^2 \colon \vec y \cdot \vec e_k \le x \}$.
Then
\begin{align*}
F_1(x)
&= \iint_{A_1(x)} f(u\vec e_2 + v\vec e_1) \, du \, dv \\
&= \int_{-\infty}^x \int_{-\infty}^\infty f(u\vec e_2 + v\vec e_1) \, du \, dv \\
&= \int_{-\infty}^x \int_{-\infty}^\infty \frac{\exp(-\frac 12(\frac{u^2}{\lambda_2}+\frac{v^2}{\lambda_1}))}{2\pi\sqrt{1-a^2}} \, du \, dv \\
&= \int_{-\infty}^x \frac{e^{-\frac 12 \frac{v^2}{\lambda_1}}}{\sqrt{2\pi\lambda_1}}
\int_{-\infty}^\infty \frac{e^{-\frac 12 \frac{u^2}{\lambda_2}}}{\sqrt{2\pi\lambda_2}} \, du \, dv = \int_{-\infty}^x \frac{e^{-\frac{v^2}{2\lambda_1}}}{\sqrt{2\pi\lambda_1}} \, dv = \Phi\Big(\frac x{\sqrt{\lambda_1}}\Big)
\end{align*}
and similarly for $F_2(x)$.

Now, observe that $P(X+Y \le x) = F_1(x/\sqrt{2}) = \Phi(x/\sqrt{2(1+a)})$ and
$P(X-Y \le x) = F_2(x/\sqrt 2) = \Phi(x/\sqrt{2(1-a)})$, and that
\[
P(|X-Y| \le x) = P(-x \le X-Y \le x) = \Phi_+\Big(\frac x{\sqrt{2(1-a)}}\Big).
\]
As $|X-Y|$ and $X+Y$ are independent random variables, we can find the probability density function
of $2\max\{X,Y\} = |X-Y|+X+Y$ by application of Lemma~\ref{lemma:normal+trunc=skew}, obtaining
\[
P(2\max\{X,Y\} \le x) = \Phi\Big(\frac x2; \frac{1-a}{\sqrt{1-a^2}}\Big).
\]
Thus by a simple change of variables and Proposition~\ref{prop:min_max_distributions} we obtain the desired result for the maximum.

For the minimum, $2\min\{X,Y\} = X+Y-|X-Y|$, which allows the application of Lemma~\ref{lemma:normal+trunc=skew}
and gives the result that
\[ P(2\min\{X,Y\} \le x) = \Phi\Big(\frac x2; -\frac{1-a}{\sqrt{1-a^2}}\Big), \]
which completes the proof after a final change of variables and Proposition~\ref{prop:min_max_distributions}.
\end{proof}

It is equally possible to describe a singular four-variable distribution that also arises in the proofs of
the theorems in Section~\ref{distribution intro section} and Theorem~\ref{theorem:dist_22}.

\begin{defn}
\label{defn:S_function}
\definitionS
\end{defn} 

\begin{defn}
\label{defn:U_function}
\definitionU
\end{defn}

\begin{lemma}
\label{lemma:4_variable_min_max}
Suppose $F(x,y,z,w)$ is a multivariate normal distribution centred at the origin with covariance matrix
\[ M = \left[ \begin{array}{cccc}
1 & a & b & c \\
a & 1 & c & b \\
b & c & 1 & a \\
c & b & a & 1
\end{array} \right] \]
with the condition that $0 < c < b < a$ and $1-a-b+c = 0$. 
Then
\begin{multline*}
\Mx_F(x) = \Phi(x) - 2\OT\Big(x;\sqrt{\frac{1+a-b-c}{2+2b}}\Big) \\
- 2\OT\Big(x;\sqrt{\frac{1-a+b-c}{2+2a}}\Big)
+ \SF\Big(4x;2\sqrt{1+a+b+c},2\sqrt{1+a-b-c},2\sqrt{1-a+b-c}\Big) , 
\end{multline*}
with corresponding characteristic function
\[ \chi_{\Mx F}(t)
= e^{-t^2/2} \bigg(1+\eta\Big(\frac{\sqrt{1+a-b-c}}{2\sqrt 2}t\Big)\bigg) \bigg(1+\eta\Big(\frac{\sqrt{1-a+b-c}}{2\sqrt 2}t\Big)
\bigg) , \]
and
\begin{multline*}
\Mn_F(x) = \Phi(x) - 2\OT\Big(x;-\sqrt{\frac{1+a-b-c}{2+2b}}\Big) \\
- 2\OT\Big(x;-\sqrt{\frac{1-a+b-c}{2+2a}}\Big)
+ \SF\Big(4x;2\sqrt{1+a+b+c},2\sqrt{1+a-b-c},2\sqrt{1-a+b-c}\Big).
\end{multline*}
with corresponding characteristic function
\[ \chi_{\Mn F}(t) = e^{-t^2/2} \bigg(1-\eta\Big(\frac{\sqrt{1+a-b-c}}{2\sqrt 2}t\Big)\bigg) \bigg(1-\eta\Big(\frac{\sqrt{1-a+b-c}}{2\sqrt 2}t\Big)
\bigg) . \]

\end{lemma}

\begin{proof}
Let $X,Y,Z,W$ be random variables with
$F(x,y,z,w) = P(X \le x, Y \le y, Z \le z, W \le w)$.
Observe that $M$ has four unit eigenvectors
\[ \vec{e_1} = \left[ \begin{array}{r} 1/2 \\ 1/2 \\ 1/2 \\ 1/2 \end{array} \right], \quad
   \vec{e_2} = \left[ \begin{array}{r} 1/2 \\ 1/2 \\ -1/2 \\ -1/2 \end{array} \right], \quad
   \vec{e_3} = \left[ \begin{array}{r} 1/2 \\ -1/2 \\ 1/2 \\ -1/2 \end{array} \right], \quad
   \vec{e_4} = \left[ \begin{array}{r} 1/2 \\ -1/2 \\ -1/2 \\ 1/2 \end{array} \right], \]	
with corresponding eigenvalues
\[ \lambda_1 = 1+a+b+c, \quad \lambda_2 = 1+a-b-c, \quad \lambda_3 = 1-a+b-c, \quad \lambda_4 = 1-a-b+c \]
where $\lambda_1 > \lambda_2 > \lambda_3 > \lambda_4 = 0$.
Let $R = \mathbf{span}(\vec{e_1},\vec{e_2},\vec{e_3})$. 
Let $M_R$ denote the matrix $M$ restricted to this subspace with basis $\vec{e_1},\vec{e_2},\vec{e_3}$,
with determinant $|M_R| = \lambda_1\lambda_2\lambda_3$. Note that $M_R$ is invertible and diagonal by
definition.

Now, observe that $X-Y-Z+W = 0$ implies that $X-Y = Z-W$ and therefore $|X-Y|-|Z-W| = 0$.
Furthermore, since $X-Y$ and $Z-W$ have the same sign, it follows that $X-Y+Z-W$
has the same sign, and thus $|X-Y|+|Z-W| = |X-Y+Z-W|$. Consequently,
\begin{align*}
&\!\! \max \{X,Y,Z,W\} \\
&= \max\{\max\{X,Y\},\max\{Z,W\}\} \\
&= \max\{X+Y+|X-Y|,Z+W+|Z-W|\}/2 \\
&= \tfrac14\bigl(X+Y+Z+W+|X-Y|+|Z-W|+\big|X+Y-Z-W+|X-Y|-|Z-W|\big|\bigr) \\
&= \tfrac14\bigl(X+Y+Z+W+|X-Y+Z-W|+|X+Y-Z-W|\bigr) . 
\end{align*}
Likewise,
\begin{align*}
&\!\!\min \{X,Y,Z,W\} \\
&= \min\{\min\{X,Y\},\min\{Z,W\}\} \\
&= \min\{X+Y-|X-Y|,Z+W-|Z-W|\}/2 \\
&= \tfrac14\bigl(X+Y+Z+W-|X-Y|-|Z-W|-\big|X+Y-Z-W-|X-Y|+|Z-W|\big|\bigr) \\
&= \tfrac14\bigl(X+Y+Z+W-|X-Y+Z-W|-|X+Y-Z-W|\bigr) .
\end{align*}

As $\vec{e_1},\vec{e_2},\vec{e_3}$ are eigenvectors of $M$ and $X,Y,Z,W$ is a multivariate
normal distribution,
it follows that $X+Y+Z+W$, $|X+Y-Z-W|$, and $|X-Y+Z-W|$ are independent random variables, whose distributions we now determine.
%
For $k=1,2,3$, set $F_k(x) = P((X,Y,Z,W) \cdot \vec e_k \le x)$ and $A_k(x) = \{ \vec v \in R \colon \vec v \cdot \vec {e_k} \le x \}$.
Then for any permutation $(i,j,k)$ of $(1,2,3)$,
\begin{align*}
F_k(x) 
&= \iiint_{A_k(x)} f(\vec u) \, dV \\
&= \iiint_{A_k(x)} \frac{\exp(-\frac 12 \vec u^\intercal M_R^{-1} \vec u)}{\sqrt{(2\pi)^3|M_R|}} \, dV \\
&= \iiint_{A_k(x)}
\frac{\exp(-\frac 12 (\frac{u^2}{\lambda_i}+\frac{v^2}{\lambda_j}+\frac{t^2}{\lambda_k}))}{\sqrt{(2\pi)^3|M_R|}} \, du \, dv \, dt \\
&= \int_{-\infty}^x \int_{-\infty}^\infty \int_{-\infty}^\infty
\frac{\exp(-\frac 12 (\frac{u^2}{\lambda_i}+\frac{v^2}{\lambda_j}+\frac{t^2}{\lambda_k}))}{\sqrt{(2\pi)^3\lambda_1\lambda_2\lambda_3}}
\, du \, dv \, dt \\
&= \int_{-\infty}^x \frac{e^{-\frac{t^2}{2\lambda_k}}}{\sqrt{2\pi\lambda_k}}
\int_{-\infty}^\infty \frac{e^{-\frac{v^2}{2\lambda_j}}}{\sqrt{2\pi\lambda_j}}
\int_{-\infty}^\infty \frac{e^{-\frac{u^2}{2\lambda_i}}}{\sqrt{2\pi\lambda_i}}
\, du \, dv \, dt = \int_{-\infty}^x \frac{e^{-\frac{t^2}{2\lambda_k}}}{\sqrt{2\pi\lambda_k}} \, dt
  = \Phi\bigg(\frac x{\sqrt{\lambda_k}} \bigg). 
\end{align*}
Consequently,
\begin{align*}
P((X,Y,Z,W) \cdot \vec e_k \le \frac x2) &= F_k(\frac x2) = \Phi\Big(\frac x{2\sqrt{\lambda_k}}\Big) \\
P(|(X,Y,Z,W) \cdot \vec{e_k}| \le \frac x2) &= \Phi_+\Big(\frac x{2\sqrt{\lambda_k}}\Big).
\end{align*}
Thus by Lemma~\ref{lemma:triple_distribution},
\[
4\max\{X,Y,Z,W\} = X+Y+Z+W+|X+Y-Z-W|+|X-Y+Z-W|.
\]
Since $(X+Y+Z+W), |X+Y-Z-W|, |X-Y+Z-W|$ are independent random variables with the distributions given above,
\begin{multline*}
P(4\max\{X,Y,Z,W\} \le x) = \Phi\Big(\frac x4\Big)
- 2\OT\Big(\frac x4; \sqrt{\frac{\lambda_2}{\lambda_1+\lambda_3}}\Big) \\
- 2\OT\Big(\frac x4; \sqrt{\frac{\lambda_3}{\lambda_1+\lambda_2}}\Big)
+ \SF(x;2\sqrt{\lambda_1},2\sqrt{\lambda_2},2\sqrt{\lambda_3}).
\end{multline*}
The desired result is obtained by a simple change of variables and Proposition~\ref{prop:min_max_distributions}.

Furthermore, by the same method, since
\[ 4\min\{X,Y,Z,W\} = X+Y+Z+W-|X+Y-Z-W|-|X-Y+Z-W|, \]
it follows that
\begin{multline*}
P(4\min\{X,Y,Z,W\} \le x) = \Phi\Big(\frac x4\Big)
- 2\OT\Big(\frac x4; -\sqrt{\frac{\lambda_2}{\lambda_1+\lambda_3}}\Big) \\
- 2\OT\Big(\frac x4; -\sqrt{\frac{\lambda_3}{\lambda_1+\lambda_2}}\Big)
+ \SF(x;2\sqrt{\lambda_1},2\sqrt{\lambda_2},2\sqrt{\lambda_3}).
\qedhere
\end{multline*}
\end{proof}

The standard normal distribution and the two distributions characterized in the lemmas in this section
fully characterize all probability distribtions found in the theorems in Section~\ref{distribution intro section} except for one.
This last distribution is irregular and will be treated as a special case, though it is similar
to the final distribution discussed above.

\subsection{$\varphi$-sequences and probability distributions} \label{phi-sequences and probability section}

By Proposition~\ref{prop:y_typical}, the distribution of the counting functions $\inv(m;d)$ of invariant factors is dictated by functions of the form $\omega(m;Q_1-Q_2)$ from Definition~\ref{defn:omega_functions}, where~$Q_1$ and~$Q_2$ are sets of primes congruent to~$1$ modulo two prime powers coming from a two-totient sequence (as defined in Proposition~\ref{prop:two-totient sequences}). In this section we write down the distributions corresponding to all functions of this type arising from two-totient sequences. In our application, two-totient sequences with~$2$ elements give rise to $1$-dimensional normal distributions, which we can determine using Lemma~\ref{lemma:omega_diff_univariable}; on the other hand, two-totient sequences with~$3$ or~$4$ elements give rise to multivariate normal distributions of dimension~$2$ and~$4$, respectively, which we can determine using Proposition~\ref{lemma:omega_diff_multivariable}.
The lemmas in this section describe all the possibilities we will need.

The first lemma will be used to prove Theorems~\ref{theorem:dist_11} and~\ref{theorem:dist_2}:


\begin{lemma}
\label{lemma:type_phi(1{:}1)} 
Suppose $\varphi(p_1^\alpha) \le \varphi(p_2^\beta)$ with $p_1 \ne p_2$. Let
\[
Q_1 = \{ p \colon p \equiv 1 \mod{p_1^\alpha} \}, \quad Q_2 = \{ p \colon p \equiv 1 \mod{p_2^\beta} \},
\]
and set $\Delta_1 = \den(Q_1)$ and $\Delta_2 = \den(Q_2)$ and $\Delta_3 = \Delta_1\Delta_2$.
Then the limiting distribution $F( x)$~of
\[ F_n(x) = P_n\bigg\{ \frac{\omega(m;Q_1-Q_2)-(\Delta_1-\Delta_2)\log\log n}{(\Delta_1+\Delta_2-2\Delta_3)^{1/2}(\log\log n)^{1/2}} \le x \bigg\} \]
is the standard normal distribution $\Phi(x)$ with mean~$0$ and variance~$1$.
\end{lemma}

\begin{proof}
Since $\omega(m;Q_1-Q_2)\in\cH$ by Lemma~\ref{lemma:omega_diff_H}, this is an immediate consequence of Lemma~\ref{lemma:omega_diff_univariable}, where $A_\omega(n;S-T)$ and $B_\omega(n;S-T)$ are evaluated asymptotically in a way similar to Lemma~\ref{lemma:sum_arithmetic_progression_pnt} and its proof.
\end{proof}


The second lemma will be used to prove Theorem~\ref{theorem:dist_21}:

\begin{lemma}
\label{lemma:type_phi(2{:}1)}
Let $\{p_1^\alpha, p_2, p_3^\beta\}$ be a two-totient sequence with $\varphi(p_1^\alpha) = \varphi(p_2) < \varphi(p_3^\beta)$. Define
\begin{align*}
Q_1 &= \{ p \colon p \equiv 1 \mod{p_1} \}, \\
Q_2 &= \{ p \colon p \equiv 1 \mod{p_2^\alpha} \}, \\
Q_3 &= \{ p \colon p \equiv 1 \mod{p_3^\beta} \},
\end{align*}
with $\Delta_1 = \den(Q_1)$, $\Delta_2 = \den(Q_3)$, and $\Delta_3 = \Delta_1\Delta_2$.
Then the limiting distribution $F(\vec x)$ of
\begin{multline*}
F_n(\vec x) = P_n \bigg\{
\frac{\omega(m;Q_1-Q_3)-(\Delta_1-\Delta_2)\log\log n}{(\Delta_1+\Delta_2-2\Delta_3)^{1/2}(\log\log n)^{1/2}} \le x_1, \\ 
\frac{\omega(m;Q_2-Q_3)-(\Delta_1-\Delta_2)\log\log n}{(\Delta_1+\Delta_2-2\Delta_3)^{1/2}(\log\log n)^{1/2}} \le x_2 \bigg\}
\end{multline*}
is the bivariate normal distribution centred at $0$ and covariance matrix
\[
M = \begin{bmatrix} 1 & a \\ a & 1 \end{bmatrix}
\quad\text{where}\quad
a = \frac{\Delta_1^2+\Delta_2-2\Delta_3}{\Delta_1+\Delta_2-2\Delta_3},
\]
which has determinant $1-a^2$ and eigenvalue--eigenvector pairs
\[ \big(1+a, [1,1]\big),\, \big(1-a, [1,-1]\big). \]
The probability density function of $F(\vec x)$ is
\[
f(\vec x) = \frac{\exp(-\frac12 \vec x M^{-1} \vec x^\intercal)}{2\pi \sqrt{1-a^2}}.
\]
\end{lemma}

\begin{proof}
Since $\omega(m;Q_1-Q_3)\in\cH$ and $\omega(m;Q_2-Q_3)\in\cH$ by Lemma~\ref{lemma:omega_diff_H}, this is an immediate consequence of Proposition~\ref{lemma:omega_diff_multivariable}. We must confirm that $\den(Q_1\cap Q_3) = \den(Q_2\cap Q_3) = \Delta_1\Delta_2$, but this follows from the fact that $\{p_1,p_2,p_3\}$ are distinct in such two-totient sequences.
\end{proof}

The third lemma will be used to prove Theorem~\ref{theorem:dist_12}:

\begin{lemma}
\label{lemma:type_phi(1{:}2)}
Let $\{p_1^\alpha, p_2, p_3^\beta\}$ be a two-totient sequence with $\varphi(p_1^\alpha) < \varphi(p_2) = \varphi(p_3^\beta)$. Define
\begin{align*}
Q_1 &= \{ p \colon p \equiv 1 \mod{p_1^\alpha} \}, \\
Q_2 &= \{ p \colon p \equiv 1 \mod{p_2} \}, \\
Q_3 &= \{ p \colon p \equiv 1 \mod{p_3^\beta} \},
\end{align*}
and set $\Delta_1 = \den(Q_1)$, $\Delta_2 = \den(Q_2)$, and $\Delta_3 = \Delta_1\Delta_2$.
Then the limiting distribution $F(\vec x)$ of
\begin{multline*}
F_n(\vec x) = P_n \bigg\{
\frac{\omega(m;Q_1-Q_2)-(\Delta_1-\Delta_2)\log\log n}{(\Delta_1+\Delta_2-2\Delta_3)^{1/2}(\log\log n)^{1/2}} \le x_1, \\ 
\frac{\omega(m;Q_1-Q_3)-(\Delta_1-\Delta_2)\log\log n}{(\Delta_1+\Delta_2-2\Delta_3)^{1/2}(\log\log n)^{1/2}} \le x_2 \bigg\}
\end{multline*}
is the nonsingular bivariate normal distribution with mean $\vec 0$ and covariance matrix
\[
M = \left[ \begin{array}{cc} 1 & a \\ a & 1 \end{array} \right]
\quad\text{where}\quad
a = \frac{\Delta_1+\Delta_2^2-2\Delta_3}{\Delta_1+\Delta_2-2\Delta_3},
\]
which has determinant~$1-a^2$ and eigenvalue--eigenvector pairs
\[
\big(1+a, [1,1]\big),\, \big(1-a, [1,-1]\big).
\]
The probability density function of $F(\vec x)$ is
\[
f(\vec x) = \frac{\exp(-\frac12 \vec x M^{-1} \vec x^\intercal)}{2\pi \sqrt{1-a^2}} .
\]
\end{lemma}

\begin{proof}
Since $\omega(m;Q_1-Q_2)\in\cH$ and $\omega(m;Q_1-Q_3)\in\cH$ by Lemma~\ref{lemma:omega_diff_H}, this is an immediate consequence of Proposition~\ref{lemma:omega_diff_multivariable}. We must confirm that $\den(Q_1\cap Q_2) = \den(Q_1\cap Q_3) = \Delta_1\Delta_2$, but this follows from the fact that $\{p_1,p_2,p_3\}$ are distinct in such two-totient sequences; the only exception is $(p_1^\alpha, p_2, p_3^\beta) = (2,3,4)$, but in this case $Q_1\cap Q_2 = Q_2$ and $Q_1\cap Q_3=Q_3$, and so we still have $\den(Q_1\cap Q_2) = \den(Q_1\cap Q_3) = \Delta_1\Delta_2$ since $\Delta_1=1$.
\end{proof}

The fourth lemma will be used to prove Theorem~\ref{theorem:dist_22}:

\begin{lemma}
\label{lemma:type_phi(2{:}2)}
Let $\{p_1, p_2^\alpha, p_3, p_4^\beta\}$ be a two-totient sequence with $\varphi(p_1) = \varphi(p_2^\alpha) < \varphi(p_3) = \varphi(p_4^\beta)$, and assume that $(p_1, p_2^\alpha, p_3, p_4^\beta) \ne (3,4,5,8)$. Define
\begin{align*}
Q_1 &= \{ p \colon p \equiv 1 \mod{p_1} \}, \\
Q_2 &= \{ p \colon p \equiv 1 \mod{p_2^\alpha} \}, \\
Q_3 &= \{ p \colon p \equiv 1 \mod{p_3} \}, \\
Q_4 &= \{ p \colon p \equiv 1 \mod{p_4^\beta} \},
\end{align*}
and set $\Delta_1 = \den(Q_1)$, $\Delta_2 = \den(Q_3)$, and $\Delta_3 = \Delta_1\Delta_2$.
Then the limiting distribution $F(\vec x)$ of
\begin{multline} \label{4d general Fnx}
F_n(\vec x) = P_n \bigg\{ 
\frac{\omega(m;Q_1-Q_3) - (\Delta_1-\Delta_2)\log\log n}{(\Delta_1+\Delta_2-2\Delta_3)^{1/2}(\log\log n)^{1/2}} \le x_1 , \\ 
\frac{\omega(m;Q_1-Q_4) - (\Delta_1-\Delta_2)\log\log n}{(\Delta_1+\Delta_2-2\Delta_3)^{1/2}(\log\log n)^{1/2}} \le x_2 , 
\frac{\omega(m;Q_2-Q_3) - (\Delta_1-\Delta_2)\log\log n}{(\Delta_1+\Delta_2-2\Delta_3)^{1/2}(\log\log n)^{1/2}} \le x_3 , \\
\frac{\omega(m;Q_2-Q_4) - (\Delta_1-\Delta_2)\log\log n}{(\Delta_1+\Delta_2-2\Delta_3)^{1/2}(\log\log n)^{1/2}} \le x_4  \bigg\}
\end{multline}
is the singular multivariate normal distribution with mean $\vec 0$ and covariance matrix
\[
M = \left[ \begin{array}{cccc}
1 & a & b & c \\
a & 1 & c & b \\
b & c & 1 & a \\
c & b & a & 1
\end{array} \right]
\]
where
\begin{align*}
a = \frac{\Delta_1+\Delta_2^2-2\Delta_3}{\Delta_1+\Delta_2-2\Delta_3}, \quad
b = \frac{\Delta_1^2+\Delta_2-2\Delta_3}{\Delta_1+\Delta_2-2\Delta_3}, \quad
c = \frac{\Delta_1^2+\Delta_2^2-2\Delta_3}{\Delta_1+\Delta_2-2\Delta_3} ,
\end{align*}
which has determinant~$0$ and eigenvalue--eigenvector pairs (in ascending order of eigenvalue)
\begin{multline*}
\big(0, [1,-1,-1,1]\big),\, \big(1-a+b-c, [1,-1,1,-1]\big), \\
\big(1+a-b-c, [1,1,-1,-1]\big),\, \big(1+a+b+c, [1,1,1,1]\big) .
\end{multline*}
The probability density function of $F(\vec x)$ is supported on the hyperplane orthogonal to $[1,-1,-1,1]$ and is given there by
\[
f(\vec x) = \frac{\exp(-\frac12 \vec x N \vec x^\intercal)}{\sqrt{8\pi^3(1+a+b+c)(1+a-b-c)(1-a+b-c)}},
\]
where the matrix~$N$ is given by
\[
N = C
\left[ \begin{array}{ccc}
	\frac 1{{1+a+b+c}} & 0 & 0 \\
	0 & \frac 1{{1+a-b-c}} & 0 \\
	0 & 0 & \frac 1{{1-a+b-c}}
\end{array} \right]
C^\intercal
\quad\text{with}\quad
C = \frac12 \left[ \begin{array}{ccc}
	1 &  1 &  1 \\
	1 &  1 & -1 \\
	1 & -1 &  1 \\
	1 & -1 & -1
\end{array} \right] .
\]
\end{lemma}

\begin{proof}
First, observe that by Lemma~\ref{lemma:omega_diff_H}, that $\omega(m;Q_1-Q_3)$, $\omega(m;Q_1-Q_4)$, $\omega(m;Q_2-Q_3)$, and $\omega(m;Q_2-Q_4)$
are all in~$\cH$ by Lemma~\ref{lemma:omega_diff_H}. Thus we can apply Proposition~\ref{lemma:omega_diff_multivariable} to determine the distribution~$F$, with the entries of~$M$ coming easily from the Dirichlet densities of the sets~$Q_i$ and their symmetric differences. It is easy to verify the eigenvector--eigenvalue pairs above once written down. The fact that $\Delta_2<\Delta<1$, together with the evaluation $c = (\Delta_1-\Delta_2)^2/(\Delta_1+\Delta_2-2\Delta_3)$, implies that $a>b>c>0$, from which the ordering of the eigenvalues follows.
\end{proof}

\begin{remark}
The fact that the first eigenvalue--eigenvector pair is $\big(0, [1,-1,-1,1]\big)$ is obvious in hindsight, since the sum of the left-hand sides of the first and fourth inequalities in equation~\eqref{4d general Fnx} is exactly equal to the sum of the left-hand sides of the second and third inequalities.
\end{remark}

Finally, the next lemma will be used in the proof of Theorem~\ref{theorem:dist_22alt}. This lemma concerns an irregular
distribution lacking the elegant symmetry that makes the preceding distributions easy to simplify into
explicit univariate distributions for the corresponding invariant factor; but the multivariate joint distribution is still
fairly nice, being simply a normal distribution.

\begin{lemma}
\label{lemma:type_phi(3,4,5,8)}
Let
\begin{align*}
Q_1 &= \{ p \colon p \equiv 1 \mod{3} \}, \\
Q_2 &= \{ p \colon p \equiv 1 \mod{4} \}, \\
Q_3 &= \{ p \colon p \equiv 1 \mod{5} \}, \\
Q_4 &= \{ p \colon p \equiv 1 \mod{8} \},
\end{align*}
Then the limiting distribution of 
\begin{multline}
F_n(\vec x) = P_n \bigg\{
\frac{\omega(m;Q_1-Q_3) - \frac14\log\log n}{(\frac12\log\log n)^{1/2}} \le x_1 ,
\frac{\omega(m;Q_1-Q_4) - \frac14\log\log n}{(\frac12\log\log n)^{1/2}} \le x_2 , \\
\frac{\omega(m;Q_2-Q_3) - \frac14\log\log n}{(\frac12\log\log n)^{1/2}} \le x_3 ,
\frac{\omega(m;Q_2-Q_4) - \frac14\log\log n}{(\frac12\log\log n)^{1/2}} \le x_4
 \bigg\} \end{multline}
is the singular multivariate normal distribution with mean $\vec 0$ and covariance matrix
\[ M' = \left[ \begin{array}{cccc}
1 & 5/8 & 1/2 & 1/8 \\ 5/8 & 1 & -1/8 & 1/4 \\ 1/2 & -1/8 & 1 & 3/8 \\ 1/8 & 1/4 & 3/8 & 1/2 
\end{array} \right] \]
wihch has determinant~$0$ and eigenvalue--eigenvector pairs
\begin{multline*}
(0, [1,-1,-1,1]),\, (\tfrac12, [1,-\tfrac12,0,-\tfrac32]), \\
\bigl( -\tfrac{\sqrt 2}4+\tfrac32, [1, 3\sqrt 2 + 5,-4\sqrt 2 -5, -\sqrt 2 - 1] \bigr), \,
\bigl( \tfrac{\sqrt 2}4+\tfrac32, [1, -3\sqrt 2 + 5,4\sqrt 2 -5, \sqrt 2 - 1] \bigr) .
\end{multline*}
\end{lemma}

\begin{proof}
As before, we can apply Proposition~\ref{lemma:omega_diff_multivariable} to determine the distribution~$F$, with the entries of~$M$ coming easily from the Dirichlet densities of the sets~$Q_i$ and their symmetric differences; the computations are tedious but straightforward. It is easy to verify the eigenvector--eigenvalue pairs above once written down.
\end{proof}

This concludes all the prior work needed for the proof of the invariant factor distribution theorems, which is the subject of the next section.

\subsection{Proofs of the theorems in Section~\ref{distribution intro section}} \label{proofs section}

In this section, we prove Theorems~\ref{theorem:dist_11}, \ref{theorem:dist_21}, \ref{theorem:dist_12}, \ref{theorem:dist_22}, \ref{theorem:dist_22alt}, \ref{theorem:dist_2}, and~\ref{theorem:dist_rare}. We begin by getting the rare invariant factor orders out of the way:

\begin{proof}[Proof of Theorem~\ref{theorem:dist_rare}]
Suppose $d$ is a rare invariant factor order. Let $y$ be a real number such that $d \le \iV(y)$.
By Proposition~\ref{prop:y_typical}\eqref{yt5}--\eqref{ytnew}, if~$m$ is a $y$-typical number then every invariant factor of $(\Z/m\Z)^\times$ not exceeding $\iV(y)$ must be an element of one of the $\sif_i$; consequently, the rare invariant factor order~$d$ cannot be such an invariant factor unless~$m$ is not $y$-typical.
In particular,
\[ P_n(\inv(m;d) > 0) \le P_n(m \text{ is not $y$-typical}) \ll_y \frac1{\log\log n} \]
by Proposition~\ref{prop:y_typical_is_typical}, and therefore $\lim_{n\to\infty} P_n(\inv(m;d) > 0) = 0$.
\end{proof}

The fact that non-$y$-typical numbers are rare will also matter greatly in the remaining proofs; the following lemma codifies this idea in a useful way.

\begin{lemma} \label{1/loglogn lemma}
Fix $y\ge2$. If $S_1$ and $S_2$ are two statements about~$m$ that are equivalent when~$m$ is $y$-typical, then $P_n(S_1) = P_n(S_2) + O_y(1/\log\log n)$.
\end{lemma}

\begin{proof}
Proposition~\ref{prop:y_typical_is_typical} implies that for $i\in\{1,2\}$,
\begin{align*}
P_n(S_i) &= P_n(S_i,\, m \text{ is $y$-typical}) + P_n(S_i,\, m \text{ is not $y$-typical}) \\
&= P_n(S_i,\, m \text{ is $y$-typical}) + O\bigl( P_n(m \text{ is not $y$-typical}) \bigr) \\
&= P_n(S_i,\, m \text{ is $y$-typical}) + O_y(1/\log\log n)
\end{align*}
Since $P_n(S_1,\, m \text{ is $y$-typical}) = P_n(S_2,\, m \text{ is $y$-typical})$ by assumption, we conclude that
\begin{align*}
P_n(S_1) &= P_n(S_1,\, m \text{ is $y$-typical}) + O_y(1/\log\log n) \\
&= P_n(S_2,\, m \text{ is $y$-typical}) + O_y(1/\log\log n) = P_n(S_1) + O_y(1/\log\log n).
\qedhere
\end{align*}
\end{proof}

We now set some notation and make some strategic comments that apply to the proofs of Theorems~\ref{theorem:dist_11}, \ref{theorem:dist_21}, \ref{theorem:dist_12}, \ref{theorem:dist_22}, and~\ref{theorem:dist_22alt}, after which we give each of those five proofs in turn. (The proof of Theorem~\ref{theorem:dist_2} will appear at the end of this section.)

In all five theorems, we are assuming that $\sif_i = \{d\}$ is a singleton set, which implies that $\varphi(\tq_i) < \varphi(\tq_{i+1})$. Recalling the quantities~$\mu_i$ and~$\sigma_i$ from Definition~\ref{defn:mu_sigma_sequences},
define the random variable
\begin{equation} \label{In def}
I_n = I_n(d) = \frac{\inv(m;d) - \mu_i\log\log n}{\sigma_i(\log\log n)^{1/2}},
\end{equation}
where~$m$ is chosen uniformly from $\{1,\ldots,n\}$;
to prove these five theorems, it suffices to prove that
$P_n(I_n \le x)$ has the given distributions.

Set $y = \varphi(\tq_{i+1})$ and $N=\iW(y)$. For any $y$-typical integer~$m$, let $q_1(m),\ldots,q_N(m)$ be the ordered sequence of prime powers that exists when
$m$ is $y$-typical (as in Proposition~\ref{prop:y_typical}), and define
\[
Q_j(m) = \bigl\{ p \colon p \equiv 1 \mod{q_j(m)} \bigr\} .
\]
Since $m$ is $y$-typical, we have $\inv(m;d) = \omega(m;Q_i(m)-Q_{i+1}(m)) + O(1)$ by Proposition~\ref{prop:y_typical}\eqref{yt6} and Lemma~\ref{lemma:prime_multiples}. Thus by Lemma~\ref{1/loglogn lemma},
\begin{align}
P_n(I_n \le x) &= P_n\biggl( \frac{\omega(m;Q_i(m)-Q_{i+1}(m)) + O(1) - \mu_i\log\log n}{\sigma_i(\log\log n)^{1/2}} \le x \biggr) + O_y\biggl( \frac1{\log\log n} \biggr) \notag \\
&= P_n\biggl( \frac{\omega(m;Q_i(m)-Q_{i+1}(m)) - \mu_i\log\log n}{\sigma_i(\log\log n)^{1/2}} \le x \biggr) + O_y\biggl( \frac1{\log\log n} \biggr) ,
\label{Pn QQ}
\end{align}
where the second equality is valid as long as the limiting distribution function is continuous in~$x$ (since $O(1)/\sigma_i(\log\log n)^{1/2} \to 0$), which will always be the case in the proofs to follow.

When $m$ is $y$-typical, we have $q_i(m) = \tq_i$ except possibly when $\varphi(\tq_i) = \varphi(\tq_{i-1})$, in which case we might have $q_i(m) = \tq_{i-1}$; similarly we have $q_{i+1}(m) = \tq_{i+1}$ except possibly when $\varphi(\tq_{i+1}) = \varphi(\tq_{i+2})$, in which case we might have $q_{i+1}(m) = \tq_{i+2}$.
Accordingly,
for $a\in\{0,1\}$ and $b\in\{1,2\}$ we define the random variable
\begin{equation} \label{Ynab def}
Y_n^{a,b} = \frac{\omega(m;\TQ_{i-a}-\TQ_{i+b}) - \mu_i\log\log n}{\sigma_i(\log\log n)^{1/2}}.
\end{equation}
where~$m$ is chosen uniformly from $\{1,\ldots,n\}$, with $\TQ_j$ as defined in equation~\eqref{Q_j def}.
Then the probability on the right-hand side of equation~\eqref{Pn QQ} can be expressed (again using Lemma~\ref{1/loglogn lemma}) in terms of the cumulative distribution functions of the~$Y_n^{a,b}$ that are relevant under the varying assumptions of the five theorems.
%

We dispose of the simplest case first:

\begin{proof}[Proof of Theorem~\ref{theorem:dist_11}]
The assumption $\#\sif_{i-1} = \#\sif_{i+1} = 1$ forces $q_i(m) = \tq_i$ and $q_{i+1}(m) = \tq_{i+1}$, so that equation~\eqref{Pn QQ} becomes
\[
P_n(I_n \le x) = P_n(Y_n^{0,1} \le x) + O_y\biggl( \frac1{\log\log n} \biggr).
\]
Thus by Lemma~\ref{lemma:type_phi(1{:}1)},
\begin{align*}
\lim_{n\to\infty} P_n(I_n \le x) = \lim_{n\to\infty} P_n(Y_n^{0,1} \le x) = \Phi(x)
\end{align*}
on noting that the quantities $\Delta_1-\Delta_2$ and $\Delta_1+\Delta_2-2\Delta_3$ in that lemma are precisely the quantities~$\mu_i$ and~$\sigma_i^2$ as per Definition~\ref{defn:mu_sigma_sequences}.
\end{proof}

The presence of a doubleton~$\sif_{i-1}$ to the left complicates the distribution in a way we can still analyze.

\begin{proof}[Proof of Theorem~\ref{theorem:dist_21}]
Under the assumptions $\#\sif_{i-1} = 2$ and $\#\sif_{i+1} = 1$, we know that
$q_i(m) \in \{\tq_{i-1},\tq_i\}$ and $q_{i+1}(m) = \tq_{i+1}$.
Moreover, according to the ordering on the $q_j(m)$ described in
Proposition~\ref{prop:y_typical}, the determining factor in whether
$q_i(m) = \tq_{i-1}$ or $q_i(m) = \tq_i$ is which $\prm(m,\tq_j)$ is smaller. In particular,
\[ \prm(m,q_i(m)) = \min\{\prm(m,\tq_{i-1}),\prm(m,\tq_i)\} . \]
Equation~\eqref{Pn QQ} then becomes
\[
P_n(I_n \le x) = P_n\bigl( \min\{ Y_n^{1,1}, Y_n^{0,1} \} \le x \bigr) + O_y\biggl( \frac1{\log\log n} \biggr).
\]
Consider the two-dimensional distribution with cumulative distribution function $F_n(x_1,x_2) = P_n(Y_n^{1,1} \le x_1,\, Y_n^{0,1} \le x_2)$. By Lemma~\ref{lemma:type_phi(2{:}1)}, 
$F(x_1,x_2) = \lim_{n\to\infty} F_n(x_1,x_2)$
exists and is a nonsingular bivarate normal distribution centred at the origin with covariance matrix
\[
M = \begin{bmatrix} 1 & a \\ a & 1 \end{bmatrix},
\quad\text{where}\quad
a = \frac{\den(\TQ_i)^2+\den(\TQ_{i+1})-2\den(\TQ_i)\den(\TQ_{i+1})}{\den(\TQ_i)+\den(\TQ_{i+1})-2\den(\TQ_i)\den(\TQ_{i+1})}  = \biggl( \frac{\sigma_{i;2,1}}{\sigma_i} \biggr)^2
\]
as in Definition~\ref{defn:nu}.
Thus by Proposition~\ref{prop:min_max_distributions} and Lemma~\ref{lemma:bivariate_min_max},
\begin{align*}
\lim_{n\to\infty} P_n(I_n \le x) &= \lim_{n\to\infty} P_n(\min\{Y_n^{1,1},Y_n^{0,1}\} \le x) \\
&= \lim_{n\to\infty} \Mn_{F_n}(x) = \Mn_F(x) = \Phi\biggl(x, -\sqrt{\frac{1-a}{1+a}}\biggr) .
\end{align*}
This equality establishes the theorem, once we confirm that $a=(\sigma_{i;2,1}/\sigma_i)^2$ does imply that $\sqrt{(1-a)/(1+a)}$ equals $\nu_i/\sqrt{\sigma_i^2+\sigma_{i;2,1}^2}$ (for which the identity $\nu_i^2 = \sigma_i^2 - \sigma_{i;2,1}^2$ is helpful), the formula for the characteristic function following from Definition~\ref{defn:skew-normal}.
\end{proof}

The next proof is extremely similar to the previous proof, other than the fact that the doubleton~$\sif_{i+1}$ lies to the right rather than to the left.

\begin{proof}[Proof of Theorem~\ref{theorem:dist_12}]
Under the assumptions $\#\sif_{i-1} = 1$ and $\#\sif_{i+1} = 2$, we know that
$q_i(m) = \tq_i$ and $q_{i+1}(m) \in \{\tq_{i+1}, \tq_{i+2}\}$; these facts hold when $i=1$ as well.
As in the previous proof, Proposition~\ref{prop:y_typical} implies that
\[
\prm(m,q_{i+1}(m)) = \max\{\prm(m,\tq_{i+1}),\prm(m,\tq_{i+2})\} .
\]
Equation~\eqref{Pn QQ} then becomes
\[
P_n(I_n \le x) = P_n\bigl( \min\{ Y_n^{0,1}, Y_n^{0,2} \} \le x \bigr) + O_y\biggl( \frac1{\log\log n} \biggr).
\]
Consider the two-dimensional distribution with cumulative distribution function $F_n(x_1,x_2) =  P_n(Y_n^{0,1} \le x_1, Y_n^{0,2} \le x_2)$.
By Lemma~\ref{lemma:type_phi(1{:}2)}, 
$F(x_1,x_2) = \lim_{n\to\infty} F_n(x_1,x_2)$
exists and is a nonsingular bivarate normal distribution centred at the origin with covariance matrix
\[
M = \begin{bmatrix} 1 & a \\ a & 1 \end{bmatrix},
\quad\text{where}\quad
a = \frac{\den(\TQ_i)+\den(\TQ_{i+1})^2-2\den(\TQ_i)\den(\TQ_{i+1})}{\den(\TQ_i)+\den(\TQ_{i+1})-2\den(\TQ_i)\den(\TQ_{i+1})}  = \biggl( \frac{\sigma_{i;1,2}}{\sigma_i} \biggr)^2.
\]
Thus by Proposition~\ref{prop:min_max_distributions} and Lemma~\ref{lemma:bivariate_min_max},
\begin{align*}
\lim_{n\to\infty} P_n(I_n \le x) &= \lim_{n\to\infty} P_n(\min\{Y_n^{0,1},Y_n^{0,2}\} \le x) \\
&= \lim_{n\to\infty} \Mn_{F_n}(x) = \Mn_F(x) = \Phi\biggl(x, -\sqrt{\frac{1-a}{1+a}}\biggr) .
\end{align*}
This equality establishes the theorem, again after an algebraic confirmation that the expression $\sqrt{(1-a)/(1+a)}$ equals $\nu_{i+1}/\sqrt{\sigma_i^2+\sigma_{i;1,2}^2}$.
\end{proof}

The most complicated situation is when both~$\sif_{i-1}$ and~$\sif_{i+1}$ are doubletons.

\begin{proof}[Proof of Theorem~\ref{theorem:dist_22}]
Under the assumptions $\#\sif_{i-1} = \#\sif_{i+1} = 2$, we know that
$q_i(m) \in \{\tq_{i-1},\tq_i\}$ and $q_{i+1}(m) \in \{\tq_{i+1}, \tq_{i+2}\}$.
As in the previous two proofs, Proposition~\ref{prop:y_typical} implies
\begin{align*}
\prm(m;q_i(m)) &= \min\{\prm(m;\tq_{i-1}),\prm(m;\tq_i)\} \\
\prm(m,q_{i+1}(m)) &= \max\{\prm(m;\tq_{i+1}),\prm(m;\tq_{i+2}\} . 
\end{align*}
Equation~\eqref{Pn QQ} then becomes
\[
P_n(I_n \le x) = P_n\bigl( \min\{ Y_n^{0,1}, Y_n^{0,2}, Y_n^{1,1}, Y_n^{1,2} \} \le x \bigr) + O_y\biggl( \frac1{\log\log n} \biggr).
\]
Consider the four-dimensional distribution with cumulative distribution function
\begin{equation} \label{still fine for 4}
F_n(x_1,x_2,x_3,x_4) =  P_n(Y_n^{0,1} \le x_1, Y_n^{0,2} \le x_2, Y_n^{1,1} \le x_3, Y_n^{1,2} \le x_4).
\end{equation}
Since $i>3$ by assumption, we do not have $(\tq_{i-1},\tq_i,\tq_{i+1},\tq_{i+2}) = (3,4,5,8)$; thus
by Lemma~\ref{lemma:type_phi(2{:}2)},
$F(x_1,x_2,x_3,x_4) = \lim_{n\to\infty} F_n(x_1,x_2,x_3,x_4)$
exists and is a singular multivarate normal distribution centred at the origin with covariance matrix
\[
M = \begin{bmatrix} 1 & a & b & c \\ a & 1 & c & b \\ b & c & 1 & a \\ c & b & a & 1 \end{bmatrix},
\]
where
\begin{align*}
a &= \frac{\den(\TQ_i)^2+\den(\TQ_{i+1})-2\den(\TQ_i)\den(\TQ_{i+1})}{\den(\TQ_i)+\den(\TQ_{i+1})-2\den(\TQ_i)\den(\TQ_{i+1})}  = \biggl( \frac{\sigma_{i;2,1}}{\sigma_i} \biggr)^2 \\
b &= \frac{\den(\TQ_i)+\den(\TQ_{i+1})^2-2\den(\TQ_i)\den(\TQ_{i+1})}{\den(\TQ_i)+\den(\TQ_{i+1})-2\den(\TQ_i)\den(\TQ_{i+1})}  = \biggl( \frac{\sigma_{i;1,2}}{\sigma_i} \biggr)^2 \\
c &= \frac{\den(\TQ_i)^2+\den(\TQ_{i+1})^2-2\den(\TQ_i)\den(\TQ_{i+1})}{\den(\TQ_i)+\den(\TQ_{i+1})-2\den(\TQ_i)\den(\TQ_{i+1})}  = \biggl( \frac{\sigma_{i;2,2}}{\sigma_i} \biggr)^2.
\end{align*}
Thus by Proposition~\ref{prop:min_max_distributions} and Lemma~\ref{lemma:4_variable_min_max},
\begin{align*}
\lim_{n\to\infty} P_n(I_n \le x) &= \lim_{n\to\infty} P_n( \min\{ Y_n^{0,1}, Y_n^{0,2}, Y_n^{1,1}, Y_n^{1,2} \} \le x \bigr) \\
&= \lim_{n\to\infty} \Mn_{F_n}(x) = \Mn_F(x) \\
&= \Phi(x) + 2\OT\bigg(x ; \sqrt{\frac{1+a-b-c}{2+2b}}\bigg) +2\OT\bigg(x ; \sqrt{\frac{1-a+b-c}{2+2a}}\bigg) \\
& \qquad{} + \SF\bigl(4x ; 2\sqrt{1+a+b+c}, 2\sqrt{1+a-b-c}, 2\sqrt{1-a+b-c}\bigr) .
\end{align*}
Theorem~\ref{theorem:dist_22} now follows from the (tedious) algebraic confirmation that the arguments of the functions on this right-hand side coincide with the arguments given in the statement of the theorem.
\end{proof}

The two-doubleton situation admits one special case which we now address.

\begin{proof}[Proof of Theorem~\ref{theorem:dist_22alt}]
When $i=3$ we have $(\tq_{i-1},\tq_i,\tq_{i+1},\tq_{i+2}) = (3,4,5,8)$. The previous proof remains valid through equation~\eqref{still fine for 4}, at which point we invoke
Lemma~\ref{lemma:type_phi(3,4,5,8)} to show that
$F(x_1,x_2,x_3,x_4) = \lim_{n\to\infty} F_n(x_1,x_2,x_3,x_4)$
exists and is a singular multivarate normal distribution centred at the origin with covariance matrix
\[
M = \begin{bmatrix} 1 & 5/8 & 1/2 & 1/8 \\ 5/8 & 1 & -1/8 & 1/4 \\ 1/2 & -1/8 & 1 & 3/8 \\ 1/8 & 1/4 & 3/8 & 1/2 \end{bmatrix},
\]
The theorem now follows from Proposition~\ref{prop:min_max_distributions}, since
\begin{equation*}
\lim_{n\to\infty} P_n(I_n \le x) = \lim_{n\to\infty} P_n( \min\{ Y_n^{0,1}, Y_n^{0,2}, Y_n^{1,1}, Y_n^{1,2} \} \le x \bigr) = \lim_{n\to\infty} \Mn_{F_n}(x) = \Mn_F(x).
\qedhere
\end{equation*}
\end{proof}

The last remaining theorem requiring proof deals with the case where the invariant factor in question is itself a member of a doubleton.

\begin{proof}[Proof of Theorem~\ref{theorem:dist_2}]
Write $\sif_i = \{d,d'\}$.
We use a variation of the setup for the previous proofs:
Set $y = \varphi(\tq_{i+2})$ and $N=\iW(y)$. For any $y$-typical integer~$m$, let $q_1(m),\ldots,q_N(m)$ be the ordered sequence of prime powers that exists when
$m$ is $y$-typical (as in Proposition~\ref{prop:y_typical}), and define
\[
Q_j(m) = \bigl\{ p \colon p \equiv 1 \mod{q_j(m)} \bigr\} .
\]

As $\varphi(\tq_i) = \varphi(\tq_{i+1})$, by Proposition~\ref{prop:y_typical}
we have that for any $y$-typical~$m$,
\begin{align} \label{jj'}
\inv(m;d) &= \begin{cases} \prm(m;\tq_i)-\prm(m;\tq_{i+1}), & \text{if } \prm(m;\tq_i) \ge \prm(m;\tq_{i+1}) , \\ 0, & \text{otherwise}  \end{cases} \\
\inv(m;d') &= \begin{cases} \prm(m;\tq_{i+1})-\prm(m;\tq_i), & \text{if } \prm(m;\tq_{i+1}) \ge \prm(m;\tq_i) , \\ 0, & \text{otherwise}, \end{cases} \notag
\end{align}
and thus by Proposition~\ref{prop:y_typical}\eqref{yt6},
\[
\inv(m;d) - \inv(m;d') = \prm(m;\tq_i)-\prm(m;\tq_{i+1}) = \omega(m;\TQ_i-\TQ_{i+1}) + O(1).
\]
Define the random variables
\begin{equation*}
I_n = \frac{\inv(m;d)}{\sigma_i(\log\log n)^{1/2}}, \quad
I_n' = \frac{\inv(m;d')}{\sigma_i(\log\log n)^{1/2}}, \quad
X_n = \frac{\omega(m;\TQ_i-\TQ_{i+1})}{\sigma_i(\log\log n)^{1/2}} .
\end{equation*}
By the same argument as for equation~\eqref{Pn QQ},
\[
P_n(I_n-I_n' \le x) = P_n(X_n \le x) + O_y \biggl( \frac1{\log\log n} \biggr),
\]
and so
\[
\lim_{n\to\infty} P_n(I_n-I_n' \le x) = \lim_{n\to\infty} P_n(X_n \le x) = \Phi(x)
\]
by Lemma~\ref{lemma:type_phi(1{:}1)}. In particular,
\[
\lim_{n\to\infty} P_n(I_n \le I_n') = \tfrac12 = \lim_{n\to\infty} P_n(I_n \ge I_n').
\]
It follows from equation~\eqref{jj'} that $I_n \le I_n'$ is equivalent to $I_n = 0$ and $I_n' \le I_n$ is equivalent to $I_n' = 0$ (for $y$-typical integers~$m$), which implies that
\begin{align*}
\lim_{n\to\infty} P_n(I_n = 0) = P_n(I_n' = 0) = \tfrac12.
\end{align*}
It follows for $x > 0$ that
\begin{align*}
\lim_{n\to\infty} P_n(I_n \le x) &= \lim_{n\to\infty} \bigl( P_n(I_n = 0) + P_n(0 < I_n \le x) \bigr) \\
&= \tfrac12 + \lim_{n\to\infty} P_n(0 < I_n-I_n' \le x) = \tfrac12 + \Phi(x) - \tfrac12 = \Phi(x) ,
\end{align*}
and so $\lim_{n\to\infty} P_n(I_n' \le x) = \frac12$ for $x>0$ by symmetry. Finally, since one of~$I_n$ and~$I_n'$ always equals~$0$ (for $y$-typical integers~$m$), we see for $x>0$ that $I_n+I_n' > x$ if and only if one of the two disjoint events $I_n > x$ or $I_n' >x$ occurs; therefore
\begin{align*}
\lim_{n\to\infty} P_n(I_n+I_n' > x) &= \lim_{n\to\infty} \bigl( P_n(I_n > x) + P_n(I_n' > x) \bigr) \\
&= \bigl( 1 - \Phi(x) \bigr) + \bigl( 1 - \Phi(x) \bigr) = 1 - \Phi_+(x)
\end{align*}
by Definition~\ref{defn:right normal}. In summary, we have shown that
\begin{align*}
\lim_{n\to\infty} P_n(I_n \le x) &= \begin{cases} \Phi(x), & \text{if } x \ge 0, \\ 0, & \text{otherwise} . \end{cases} \\
\lim_{n\to\infty} P_n(I_n' \le x) &= \begin{cases} \Phi(x), & \text{if } x \ge 0, \\ 0, & \text{otherwise} . \end{cases} \\
\lim_{n\to\infty} P_n(I_n+I_n' \le x) &= \Phi_+(x) .
\end{align*}
which completes the proof of the theorem.
\end{proof}

We conclude this section with the proof of our first theorem.

\begin{proof}[Proof of Theorem~\ref{theorem:order_restriction_finite}]
Set $y = \varphi(\tq_D)$ and $N=\iW(y)$, so that $N=D$ if $\varphi(\tq_D) < \varphi(\tq_{D+1})$ while $N=D+1$ if $\varphi(\tq_D) = \varphi(\tq_{D+1})$; it suffices to prove the theorem with~$N$ in place of~$D$. By Proposition~\ref{prop:y_typical_is_typical}, almost all numbers~$m$ are $y$-typical. For such numbers~$m$, Proposition~\ref{prop:y_typical}(\ref{yt5})--(\ref{ytnew}) shows that every invariant factor of $(\Z/m\Z)^\times$ is in one of $\sif_1,\ldots,\sif_N$.

On the other hand, the existence of the limiting distributions in Theorems~\ref{theorem:dist_11}, \ref{theorem:dist_21}, \ref{theorem:dist_12}, \ref{theorem:dist_22}, \ref{theorem:dist_22alt}, and~\ref{theorem:dist_2} implies that for each~$i\in\N$, almost all integers~$m$ have the property that $(\Z/m\Z)^\times$ has an invariant factor in~$\sif_i$. (The more precise reason is the fact that these limiting distributions have ``full mass'': the cumulative distribution function $\Phi_+(x)$ in Theorem~\ref{theorem:dist_2} satisfies $\Phi_+(0)=0$, while the cumulative distribution functions in the other theorems tend to~$0$ as $x\to-\infty$.) In particular, $(\Z/m\Z)^\times$ has invariant factors in each of $\sif_1,\ldots,\sif_N$ for almost all integers~$m$.

These two results together imply that for almost all integers~$m$, the first~$N$ invariant factors $d_1,\dots,d_N$ of $(\Z/m\Z)^\times$ satisfy $d_i \in \sif_i$ for each $1\le i\le N$, establishing the theorem. (The divisibility constraints of invariant factors, and the divisibility relationships among the~$\sif_i$, ensure that these invariant factors must appear in the correct order and that exactly one element of any doubleton~$\sif_i$ is an invariant factor.)

Note that this argument implicitly showed that $(\Z/m\Z)^\times$ has at least~$N$ invariant factors for almost all integers~$m$, as required by the theorem. We could argue this directly, however, from the lower bound $\inv(m) \ge \omega(m) - 1$ in Lemma~\ref{lemma:total_invariant_factors}, since it is well known that almost all integers have at least~$N+1$ distinct prime factors for any fixed~$N$.
\end{proof}

\section{Proofs of the theorems in Section~\ref{expectation intro section}}  \label{sec:implications}

In this section, we derive the limiting expectations in Section~\ref{expectation intro section} from the limiting distributions in Section~\ref{distribution intro section}. Morally speaking the distributional results are stronger, but we need to do a little work to ensure that atypical values do not perturb the expectations from what we presume them to be. Indeed we carry out this work for all the moments, not just the expectation, since the proof is essentially the same.

In this vein, we first we establish a lemma that relates the moments of the limiting distributions of the invariant
counting functions to the moments of the invariant counting functions sampled over finite intervals. (To be clear, we have not used the method of moments to derive those limiting distributions; we have already established the limiting distributions, and we will use them to help us evaluate the moments of the finite samples.) Already this lemma allows us to prove Theorem~\ref{theorem:ex_rare}.
Moreover, we can then use that lemma (and the results in Section~\ref{distribution intro section}) to prove Theorems~\ref{theorem:ex_11},
\ref{theorem:ex_21},
\ref{theorem:ex_12},
\ref{theorem:ex_22}, and
Theorem~\ref{theorem:ex_2}; it is easy to check that these results collectively imply Theorem~\ref{theorem:ex_general}.



Before this is done, we need a technical lemma to bridge the gap between expectations over the
integers $1,\ldots,n$ and expectations restricted to only $y$-typical values.

\begin{defn} \label{defn:indicator boole}
Given any assertion $A(m)$ about positive integers~$n$, define
\[
1[A(m)] = \begin{cases}
1, &\text{if $A(m)$ is true}, \\
0, &\text{if $A(m)$ is false}.
\end{cases}
\]
\end{defn}

\begin{lemma} \label{lemma:expectation_error}
Let $f \colon \N \to \R$ be a function. Suppose there exist nonnegative sequences $(U_n)$ and $(V_n)$ such that $|f(m)| \le U_n\omega(m)+V_n$ whenever $m\le n$.
Then for any $r\in\R$ and $s \in \N$ and $y \ge 1$,
\[
\Ex_n\bigl(f(m)^s\bigr) = \Ex_n\bigl( f(m)^s 1[m\text{ \rm is $y$-typical}] \bigr) + O_{r,s,y}\bigl( (U_n^s+V_n^s)/(\log\log n)^r \bigr).
\]
\end{lemma}

\begin{proof}
Since $\Ex_n\big(f(m)^s\big) = \Ex_n\bigl( f(m)^s 1[m\text{ is $y$-typical}] \bigr) + \Ex_n\bigl( f(m)^s 1[m\text{ isn't $y$-typical}] \bigr)$ trivially, it suffices to bound the latter quantity. Since $\bigl( U_n\omega(m)+V_n \bigr)^s \ll_s \bigl( U_n\omega(m) \bigr)^s+V_n^s$,
\begin{align}
\Ex_n\bigl( f(m)^s & 1[m\text{ isn't $y$-typical}] \bigr) \notag \\
&\ll_s U_n^s \Ex_n\bigl( \omega(m)^s 1[m\text{ isn't $y$-typical}] \bigr) + V_n^s \Ex_n\bigl( 1[m\text{ isn't $y$-typical}] \bigr) \notag \\
	& \ll_{\rho,s,y} U_n^s \Ex_n\bigl( \omega(m)^s 1[m\text{ isn't $y$-typical}] \bigr) + \frac{V_n^s}{(\log\log n)^\rho} \label{pre en81}
\end{align}
for any $\rho\in\R$ by Proposition \ref{prop:y_typical_is_typical}. Using the same proposition,
\begin{align*}
\Ex_n\bigl( \omega(m)^s & 1[m\text{ isn't $y$-typical}] \bigr) \\
&= \Ex_n\bigl( \omega(m)^s 1[m\text{ isn't $y$-typical}] 1[\omega(m) < 3s\log\log n] \bigr) \\
&\qquad{}+ \Ex_n\bigl( \omega(m)^s 1[m\text{ isn't $y$-typical}] 1[\omega(m) \ge 3s\log\log n] \bigr) \\
&\ll_{\rho,s,y} (\log\log n)^s \frac1{(\log\log n)^\rho} + (\log n)^s \Ex_n\bigl( 1[\omega(m) \ge 3s\log\log n] \bigr)
\end{align*}
since $\omega(m) \ll \log n$ for $m\le n$.
On the other hand, we know~\cite[Proposition~4]{en81} that
\[
\#\{ m\le n\colon \omega(m) \ge 3s\log\log n \} \ll n/(\log n)^{3s\log2-1},
\]
and thus
\begin{align*}
\Ex_n\bigl( \omega(m)^s 1[m\text{ isn't $y$-typical}] \bigr) &\ll_{\rho,s,y} \frac{(\log\log n)^s}{(\log\log n)^\rho} +  \frac{(\log n)^s}{(\log n)^{3s\log2-1}} \ll_{\rho,s,y} \frac1{(\log\log n)^{\rho-s}}
\end{align*}
since $3s\log2-1>s$ for all $s\ge1$. Therefore the estimate~\eqref{pre en81} becomes
\[
\Ex_n\bigl( f(m)^s 1[m\text{ isn't $y$-typical}] \bigr) \ll_{\rho,s,y} U_n^s \frac1{(\log\log n)^{\rho-s}} + \frac{V_n^s}{(\log\log n)^\rho},
\]
and the lemma follows by choosing $\rho=r+s$.
\end{proof}

This lemma alone is enough to imply the first two results in Section~\ref{expectation intro section}. We will repeatedly use the crude upper bound $\inv(m;d) \le \inv(m) \le \omega(m) + 1$ that follows from Lemma~\ref{lemma:total_invariant_factors}.

\begin{proof}[Proof of Theorem~\ref{theorem:ex_rare}]
Choose~$y$ to be the smallest number such that $\iV(y) > d$.
Since~$d$ is not a universal factor order, Proposition~\ref{prop:y_typical}(e) implies that $\inv(m;d) = 0$ for $y$-typical $m$.
Since $\inv(m;d) \le \omega(m)+1$, Lemma~\ref{lemma:expectation_error} implies that for any real $r$ and $s \in \N$,
\[
\Ex_n \inv(m;d)^s = \Ex_n\bigl( \inv(m;d)^s 1[m\text{ \rm is $y$-typical}] \bigr) + O_{r,s,y}(1/(\log\log n)^r) \ll_{r,s,d} 1/(\log\log n)^r
\]
as required.
\end{proof}

\begin{proof}[Proof of Theorem~\ref{theorem:ex_11}]
Choose~$y$ to be the smallest number such that $\iV(y) > d$. Since $\inv(m;d) \le \omega(m) + 1$, Lemma~\ref{lemma:expectation_error} with $r=0$ and $s=1$ implies that
\[
\Ex_n \inv(m;d) = \Ex_n\bigl( \inv(m;d) 1[m\text{ is $y$-typical}] \bigr) + O_y(1).
\]
For any $y$-typical integer~$m$, let $q_1(m),\ldots,q_N(m)$ be the ordered sequence of prime powers that exists when
$m$ is $y$-typical (as in Proposition~\ref{prop:y_typical}), and define
\begin{equation} \label{Qjm def maybe again}
Q_j(m) = \bigl\{ p \colon p \equiv 1 \mod{q_j(m)} \bigr\},
\end{equation}
so that $\inv(m;d) = \omega(m;Q_i(m)-Q_{i+1}(m)) + O(1)$ by Proposition~\ref{prop:y_typical}\eqref{yt6} and Lemma~\ref{lemma:prime_multiples}.
Since $m$ is $y$-typical, and since the assumptions $\#\sif_{i-1}=\#\sif_{i+1}=1$ imply that $\varphi(\tq_{i-1}) < \varphi(\tq_i)$ and $\varphi(\tq_{i+1}) < \varphi(\tq_{i+2})$, we must have $q_i(m) = \tq_i$ and $q_{i+1}(m) = \tq_{i+1}$.
We can therefore write $\inv(m;d) = \omega(m;\TQ_i(m)-\TQ_{i+1}(m)) + O(1)$, with $\TQ_j$ as defined in equation~\eqref{Q_j def}. These observations imply that
\begin{align*}
\Ex_n \inv(m;d)
&= \Ex_n\bigl( \bigl( \omega(m;Q_i(m)-Q_{i+1}(m))+O(1) \bigr) 1[m\text{ is $y$-typical}] \bigr) + O_y(1) \\
&= \Ex_n\bigl( \omega(m;\TQ_i-\TQ_{i+1}) 1[m\text{ is $y$-typical}] \bigr) + O_y(1) \\
&= \Ex_n\bigl( \omega(m;\TQ_i-\TQ_{i+1}) \bigr) + O_y(1)
\end{align*}
by Lemma~\ref{lemma:expectation_error} again. Finally, Lemmas~\ref{lemma:omega_finite_mean_variance} and~\ref{lemma:sum_arithmetic_progression_pnt} imply that
\begin{align*}
\Ex_n \inv(m;d) = A_\omega(n;\TQ_i-\TQ_{i+1}) + O_y(1) = \mu_i \log\log n + O_y(1)
\end{align*}
as required.
\end{proof}

The remaining results in Section~\ref{expectation intro section} are more complicated since we no longer have the assumptions $\#\sif_{i-1}=\#\sif_i=\#\sif_{i+1}=1$ and thus cannot conclude that $q_i(m) = \tq_i$ and $q_{i+1}(m) = \tq_{i+1}$. This complication makes it very difficult for us to get asymptotics for the moments. Fortunately, using a tool from probability and the knowledge of the limiting distributions, we can get away with upper bounds for the (normalized) moments.

Recall the definition~\eqref{In def} of~$I_n$, and note that trivially
\begin{equation} \label{In triangle ineq}
|I_n| \le \frac{\inv(m;d)}{\sigma_i(\log\log n)^{1/2}} + \frac{\mu_i(\log\log n)^{1/2}}{\sigma_i}.
\end{equation}

\begin{lemma} \label{lem:exa}
Choose $i\in\N$ and $d\in\sif_i$. For any $\kappa\in\N$, we have $\Ex_n(I_n^\kappa) \ll_{d,\kappa} 1$.
\end{lemma}

\begin{proof}
Choose~$y$ to be the smallest number such that $\iV(y) > d$. Since $\inv(m;d) \le \omega(m) + 1$, the inequality~\eqref{In triangle ineq} implies
\[
|I_n| \le \frac{\omega(m)+1}{\sigma_i(\log\log n)^{1/2}} + \frac{\mu_i(\log\log n)^{1/2}}{\sigma_i}.
\]
We may therefore apply Lemma~\ref{lemma:expectation_error} with $r=\frac\kappa2$ and $s=\kappa$, and with
\[
U_n = \frac1{\sigma_i(\log\log n)^{1/2}} \quad \text{and} \quad V_n = \frac1{\sigma_i(\log\log n)^{1/2}} + \frac{\mu_i(\log\log n)^{1/2}}{\sigma_i},
\]
both of which are $\ll_d (\log\log n)^{1/2}$; the resulting estimate is
\[
\Ex_n(I_n^\kappa) = \Ex_n\bigl( I_n^\kappa 1[m\text{ is $y$-typical}] \bigr) + O_{d,\kappa}(1).
\]

For any $y$-typical integer~$m$, let $q_1(m),\ldots,q_N(m)$ be the ordered sequence of prime powers that exists when
$m$ is $y$-typical (as in Proposition~\ref{prop:y_typical}).
By Proposition~\ref{prop:y_typical}\eqref{yt6} and Lemma~\ref{lemma:prime_multiples}, it follows that $\inv(m;d) = \omega(m;Q_i(m)-Q_{i+1}(m)) + O(1)$, with $Q_j$ as defined in equation~\eqref{Qjm def maybe again}.

Since $m$ is $y$-typical, we have $q_i(m) = \tq_i$ except possibly when $\varphi(\tq_i) = \varphi(\tq_{i-1})$, in which case we might have $q_i(m) = \tq_{i-1}$; similarly we have $q_{i+1}(m) = \tq_{i+1}$ except possibly when $\varphi(\tq_{i+1}) = \varphi(\tq_{i+2})$, in which case we might have $q_{i+1}(m) = \tq_{i+2}$.
Defining $a_m\in\{0,1\}$ and $b_m\in\{1,2\}$ so that $q_i(m) = \tq_{i-a_m}$ and $q_{i+1}(m) = \tq_{i+b_m}$, we can therefore write $\inv(m;d) = \omega(m;\TQ_{i-a_m}(m)-\TQ_{i+b_m}(m)) + O(1)$, with $\TQ_j$ as defined in equation~\eqref{Q_j def}. It follows that
\begin{align*}
\Ex_n \bigl( I_n^\kappa & 1[m\text{ is $y$-typical}] \bigr) \\
&= \sum_{\substack{a\in\{0,1\} \\ \varphi(\tq_{i-a}) = \varphi(\tq_{i})}} \sum_{\substack{b\in\{1,2\} \\ \varphi(\tq_{i+b}) = \varphi(\tq_{i+1})}} \Ex_n\biggl(\frac{\omega(m;\TQ_{i-a}(m)-\TQ_{i+b}(m))+O(1)-\mu_i\log\log n}{\sigma_i(\log\log n)^{1/2}} \\
&\qquad{}\times 1[m\text{ is $y$-typical}] 1[a_m=a] 1[b_m=b] \biggr)^\kappa \\
&\ll \sum_{\substack{a\in\{0,1\} \\ \varphi(\tq_{i-a}) = \varphi(\tq_{i})}} \sum_{\substack{b\in\{1,2\} \\ \varphi(\tq_{i+b}) = \varphi(\tq_{i+1})}} \Ex_n \biggl| \frac{\omega(m;\TQ_{i-a}(m)-\TQ_{i+b}(m))+O(1)-\mu_i\log\log n}{\sigma_i(\log\log n)^{1/2}} \biggr|^\kappa.
\end{align*}
By the results in Section~\ref{omega S minus T section}, in all cases occurring in this double sum we have
\begin{align*}
A_\omega(n;\TQ_{i-a}(m)-\TQ_{i+b}(m)) &= \mu_i\log\log n + O(1) \\
B_\omega(n;\TQ_{i-a}(m)-\TQ_{i+b}(m))^2 &= \sigma_i^2\log\log n + O(1),
\end{align*}
and therefore
\begin{multline*}
\Ex_n \bigl( I_n^\kappa 1[m\text{ is $y$-typical}] \bigr) \\
\ll \sum_{\substack{a\in\{0,1\} \\ \varphi(\tq_{i-a}) = \varphi(\tq_{i})}} \sum_{\substack{b\in\{1,2\} \\ \varphi(\tq_{i+b}) = \varphi(\tq_{i+1})}} \Ex_n \biggl|\frac{\omega(m;\TQ_i-\TQ_{i+1})-A_\omega(n;\TQ_i-\TQ_{i+1}) + O(1)}{B_\omega(n;\TQ_i-\TQ_{i+1}) + O(1)}\biggr|^\kappa.
\end{multline*}
Since the estimate to be proved is trivial for any bounded range of~$n$, we may assume that~$n$ is large enough that each relevant $B_\omega(n;\TQ_i-\TQ_{i+1})$ is at least~$\frac12$. Consequently,
\begin{multline*}
\Ex_n \bigl( I_n^\kappa 1[m\text{ is $y$-typical}] \bigr) \\
\ll \sum_{\substack{a\in\{0,1\} \\ \varphi(\tq_{i-a}) = \varphi(\tq_{i})}} \sum_{\substack{b\in\{1,2\} \\ \varphi(\tq_{i+b}) = \varphi(\tq_{i+1})}} \Ex_n \biggl|\frac{\omega(m;\TQ_i-\TQ_{i+1})-A_\omega(n;\TQ_i-\TQ_{i+1}) + O(1)}{B_\omega(n;\TQ_i-\TQ_{i+1})}\biggr|^\kappa \ll 1
\end{multline*}
by the triangle inequality and Lemma~\ref{lemma:ek_sieve_big_O new}.
\end{proof}

With this observation now complete, we can now show that the moments of the finite samples in the theorems in Section~\ref{distribution intro section} approach their limiting distributions as $n \to \infty$. We continue to use the notation~$I_n$ from equation~\eqref{In def}.

\begin{prop} \label{prop:exa}
Choose $i\in\N$ and $d\in\sif_i$. For any $k\in\N$,
\[
\lim_{n\to\infty} \Ex_n \biggl(\frac{\inv(m;d)-\mu_i\log\log n}{\sigma_i(\log\log n)^{1/2}}\biggr)^k = \int_{-\infty}^\infty x^k \, dF(x)
\]
where $F(x)$ is the appropriate limiting distribution of $F_n = P_n(I_n \le x)$ chosen from among Theorems~\ref{theorem:dist_11}, \ref{theorem:dist_21}, \ref{theorem:dist_12}, \ref{theorem:dist_22}, \ref{theorem:dist_22alt}, and~\ref{theorem:dist_2}.
In other words, the limits of the moments of~$I_n$ are equal to the moments of the limiting distribution.
\end{prop}


\begin{proof}
As $F_n$ converges weakly to $F$, there exists a sequence of random variables $X_n$ with corresponding distributions $F_n$
such that the $X_n$ converge to $X$ almost surely, where $X$ has the same distribution as $F$.
We then apply~\cite[Theorem~1.6.8]{durrett}, using $g(x) = x^{k+1}$ and $h(x) = x^k$ and noting that $\Ex X_n^{k+1} \ll_{d,k} 1$ by Lemma~\ref{lem:exa} applied with $\kappa=k+1$. We conclude that
\[
\lim_{n\to\infty} \Ex_n I_n^k = \lim_{n\to\infty} \Ex X_n^k = \Ex X^k = \int_{-\infty}^\infty x^k \, dF(x) .
\qedhere
\]
\end{proof}

We conclude this section with a derivation of Theorem~\ref{theorem:ex_21} from Theorem~\ref{theorem:dist_21}. In the same manner, one can derive Theorems~\ref{theorem:ex_12}, \ref{theorem:ex_22}, and~\ref{theorem:ex_2} from Theorems~\ref{theorem:dist_12}, \ref{theorem:dist_22}, and~\ref{theorem:dist_2}, respectively (with the caveat that Theorem~\ref{theorem:dist_22alt} is also required for one special case in the proof of Theorem~\ref{theorem:ex_22}); we omit the proofs as no new ideas are needed.

\begin{proof}[Proof of Theorem~\ref{theorem:ex_21}]
Under the hypotheses of Theorem~\ref{theorem:dist_21}, applying Proposition~\ref{prop:exa} with $k=1$ yields
\[
\Ex_n \biggl(\frac{\inv(m;d)-\mu_i\log\log n}{\sigma_i(\log\log n)^{1/2}}\biggr) = (1+o(1)) \int_{-\infty}^\infty x \, dF(x) = (1+o(1)) \Ex \Phi\biggl(x;-\frac{\nu_i}{\sqrt{\sigma_i^2+\sigma_{i;2,1}^2}}\biggr),
\]
where~$\Phi$ denotes
a skew-normal cumulative distribution function as in Definition~\ref{defn:skew-normal}. As noted in that definition, the right-hand side equals $\alpha\sqrt{2/(\alpha^2+1)\pi}$ where $\alpha = -{\nu_i}/{\sqrt{\sigma_i^2+\sigma_{i;2,1}^2}}$. By linearity of expectation,
\[
\Ex_n \inv(m;d) = \mu_i\log\log n + \sigma_i(\log\log n)^{1/2} (1+o(1)) \alpha \sqrt{\frac2{(\alpha^2+1)\pi}}.
\]
To finish the proof of Theorem~\ref{theorem:ex_21}, it remains only to check that
\[
\sigma_i \alpha \sqrt{\frac2{(\alpha^2+1)\pi}} = -\frac{\sigma_{i-1}}{\sqrt{2\pi}}
\]
using Definitions~\ref{defn:mu_sigma_sequences} and~\ref{defn:nu}---a tedious but straightforward calculation which uses the fact that $\varphi(\tq_i) = \varphi(\tq_{i-1})$ under the hypothesis $\#\sif_{i-1} = 2$.
\end{proof}

\appendix

\section{Consolidated notation}
\label{appendix:notation}

\subsection*{Totients and invariant factors}

\definitionphisequence
\definitiontwototient

\definitionUFOsets The first several universal factor order sets are
\begin{multline*}
\firsttenUFOsets.
\end{multline*}

\definitionmusigma
\definitionnu

\definitioncountingfunctionsinv
\definitioncountingfunctionspri
\definitioncountingfunctionsprm
\definitionVW

\definitiontypical

\subsection*{Additive functions and expectations}

\definitionABCD

\definitionDd

\Enotation
\definitionEk
\definitionPn

\definitionomegafunctions

\subsection*{Probability distributions}

Several functions related to common probability distributions are
\begin{align*}
\phi(x) &= \textstyle\frac{e^{-{x^2}/2}}{\sqrt{2\pi}} & \text{(normal probability density function)} , \\
\Phi(x) &= \textstyle\int_{-\infty}^x \phi(t) \, dt & \text{(normal cumulative distribution function)} , \\
\phi_+(x) &= 2\phi(x)1_{[0,\infty)}(x) & \text{(truncated right normal probability density function)} , \\
\Phi_+(x) &= \max\{\Phi(x)-\Phi(-x),0\} & \text{(truncated right normal cumulative distribution function)} , \\
\phi(x;\alpha) &= 2\phi(x)\Phi(\alpha x) & \text{(skew normal probability density function \cite{skew_normal_class})} , \\
\Phi(x;\alpha) &= \Phi(x) - 2\OT(x,\alpha) & \text{(skew normal cumulative distribution function \cite{skew_normal_class})}.
\end{align*}
\definitionS
We also define Owen's $T$-function \cite{owen_T}.
\[
\OT(h,a) = \frac1{2\pi} \int_0^a e^{-(1+t^2)h^2/2} \frac{dt}{1+t^2}.
\]
\definitionU

\definitionMnMx

\section{Consolidated theorems}
\label{appendix:theorems}

The introduction splits the main theorems of this paper into small pieces for the purpose of facilitating the exposition.
In this section, the theorems are consolidated for ease of use when being applied or referenced.

\begin{theorem}
\theoremoverall
\end{theorem}

The first theorem concerns asymptotic expectations of $\inv(m;d)$:

\begin{theorem}
Let $d \in \N$. If $d\in\siF$ is a universal factor order,
the limiting distribution of $\inv(m;d)$ under appropriate rescaling is
fully characterized by the number of elements in $\sif_{i-1}$, $\sif_i$, and $\sif_{i+1}$,
except for the special cases $d=2\in\sif_1$.
If $d\notin\siF$ is a rare factor order,
then the limiting distribution is fully characterized by this fact.

More precisely, any $d \in \N$, exactly one of the following cases applies:
\begin{enumerate}
\item \theoremexA
\item \theoremexB
\item \theoremexE
\item \theoremexD
\item \theoremexC
\item \theoremexrare
\end{enumerate}
\end{theorem}

The second theorem concerns the limiting distribution of $\inv(m;d)$:

\begin{theorem}
\theoremdistgeneral

More precisely, for any $d \in \N$, exactly one of the following cases applies:
\begin{enumerate}
\item \theoremdistA
\item \theoremdistB
\item \theoremdistE
\item \theoremdistD
\item \theoremdistDalt
\item \theoremdistC
\item \theoremdistrare
\end{enumerate}
\end{theorem}

\section{Probability distributions}
\label{appendix:distributions}

The two lemmas in this section are proved in the dissertation of second author~\cite[Appendix~B]{simpsonthesis}.

First, the sum of a random variable with a normal distribution and a truncated normal distribution
yields a skew normal distribution as described in Definition~\ref{defn:skew-normal}.

\begin{lemma} 
\label{lemma:normal+trunc=skew}
Suppose that $X,Y$ are independent random variables with $P(X \le x) = \Phi(x/\sigma_1)$ and $P(Y \le x) = \Phi_+(x/\sigma_2)$, then
\begin{align*}
P(X+Y \le x) = \Phi\Big(\frac x{\sqrt{\sigma_1^2+\sigma_2^2}}\Big)
- 2\OT\Big(\frac x{\sqrt{\sigma_1^2+\sigma_2^2}}, \frac{\sigma_2}{\sigma_1}\Big) , \\
P(X-Y \le x) = \Phi\Big(\frac x{\sqrt{\sigma_1^2+\sigma_2^2}}\Big)
- 2\OT\Big(\frac x{\sqrt{\sigma_1^2+\sigma_2^2}}, -\frac{\sigma_2}{\sigma_1}\Big) ,
\end{align*}
with characteristic functions
\begin{align*}
\chi_{X+Y}(t) &= e^{-(\sigma_1^2+\sigma_2^2)t^2/2}\bigg(1+\eta\Big(\frac{\sigma_2 t}{\sqrt 2}\Big)\bigg) \\
\chi_{X-Y}(t) &= e^{-(\sigma_1^2+\sigma_2^2)t^2/2}\bigg(1-\eta\Big(\frac{\sigma_2 t}{\sqrt 2}\Big)\bigg).
\end{align*}
\end{lemma}

We can similarly characterize sums of three independent random variables where one is has a normal distribution
and the other two have truncated normal distributions.

\begin{lemma}
\label{lemma:triple_distribution}
Suppose $X,Y,Z$ are independent random variables with distribution functions
\begin{align*}
P(X \le x) &= \Phi\bigg(\frac x{\sigma_1}\bigg), \\
P(Y \le x) &= \Phi_+\Big(\frac x{\sigma_2}\Big), \\
P(Z \le x) &= \Phi_+\Big(\frac x{\sigma_3}\Big).
\end{align*}

Let $\Sigma = \sqrt{\sigma_1^2+\sigma_2^2+\sigma_3^2}$. Then $X+Y+Z$ and $X-Y-Z$ have distributions
\begin{align*}
P(X+Y+Z \le x) &= \Phi\Big(\frac x\Sigma\Big) 
- 2\OT\Big(\frac x\Sigma, \frac{\sigma_2}{\sqrt{\sigma_1^2+\sigma_3^2}}\Big)
- 2\OT\Big(\frac x\Sigma, \frac{\sigma_3}{\sqrt{\sigma_1^2+\sigma_2^2}}\Big)
+ \SF(x;\sigma_1,\sigma_2,\sigma_3) , \\
P(X-Y-Z \le x) &= \Phi\Big(\frac x\Sigma\Big)
- 2\OT\Big(\frac x\Sigma, -\frac{\sigma_2}{\sqrt{\sigma_1^2+\sigma_3^2}}\Big)
- 2\OT\Big(\frac x\Sigma, -\frac{\sigma_3}{\sqrt{\sigma_1^2+\sigma_2^2}}\Big)
+ \SF(x;\sigma_1,\sigma_2,\sigma_3) ,
\end{align*}
with characteristic functions
\begin{align*}
\chi_{X+Y+Z}(t) &= e^{-\frac{\Sigma^2 t^2}2} \biggl(1+\eta\Big(\frac{\sigma_2t}{\sqrt 2}\Big) \biggr) \biggl( 1+\eta\Big(\frac{\sigma_3t}{\sqrt 2}\Big) \biggr) , \\
\chi_{X-Y-Z}(t) &= e^{-\frac{\Sigma^2 t^2}2} \bigg(1-\eta\Big(\frac{\sigma_2t}{\sqrt 2}\Big) \biggr) \biggl( 1-\eta\Big(\frac{\sigma_3t}{\sqrt 2}\Big) \biggr).
\end{align*}
\end{lemma}

\section*{Acknowledgments}

The authors thank Paul P\'eringuey for helpful remarks.
The first author was supported in part by a Natural Sciences and Engineering Council of Canada Discovery Grant.

\addcontentsline{toc}{chapter}{Bibliography}
\bibliography{invariants}{}

\begin{thebibliography}{10}

\bibitem{AGS}
B.~C. Arnold, H.~W. Gómez, and H.~S. Salinas.
\newblock A doubly skewed normal distribution.
\newblock {\em Statistics}, 49(4):842--858, 2015.

\bibitem{skew_normal_class}
A.~Azzalini.
\newblock A class of distributions which includes the normal ones.
\newblock {\em Scand. J. Statist.}, 12(2):171--178, 1985.

\bibitem{SIFMG}
B.~Chang and G.~Martin.
\newblock The smallest invariant factor of the multiplicative group.
\newblock {\em Int. J. Number Theory}, 16(6):1377--1405, 2020.

\bibitem{durrett}
R.~Durrett.
\newblock {\em Probability: Theory and Examples}.
\newblock Cambridge University Press, 2010.

\bibitem{en81}
P.~Erd\H{o}s and J.-L. Nicolas.
\newblock Sur la fonction: nombre de facteurs premiers de {$N$}.
\newblock {\em Enseign. Math. (2)}, 27(1-2):3--27, 1981.

\bibitem{gs07}
A.~Granville and K.~Soundararajan.
\newblock Sieving and the {E}rd{\H o}s--{K}ac theorem.
\newblock In {\em Equidistribution in number theory, an introduction}, volume
  237 of {\em NATO Sci. Ser. II Math. Phys. Chem.}, pages 15--27. Springer,
  Dordrecht, 2007.

\bibitem{hannesson_martin}
M.~Hannesson and G.~Martin.
\newblock Multiplicative groups avoiding a fixed group.
\newblock {\em Int. J. Number Theory}, page (to appear).

\bibitem{kubilius}
J.~Kubilius.
\newblock {\em Probabalistic Methods in the Theory of Numbers}, volume~11 of
  {\em Translations of Mathematical Monographs}.
\newblock American Mathematical Society, 1964.

\bibitem{ljunggren}
W.~Ljunggren.
\newblock On the irreducibility of certain trinomials and quadrinomials.
\newblock {\em Math. Scand.}, 9:65--70, 1960.

\bibitem{DNSMG}
G.~Martin and L.~Troupe.
\newblock The distribution of the number of subgroups of the multiplicative
  group.
\newblock {\em J. Aust. Math. Soc.}, 108:46--97, 2022.

\bibitem{skew_normal_sum}
S.~Nadarajah and R.~Li.
\newblock The exact density of the sum of independent skew normal random
  variables.
\newblock {\em J. Comput. Appl. Math.}, 311:1--10, 2017.

\bibitem{owen_T}
D.~B. Owen.
\newblock Tables for computing bivariate normal probabilities.
\newblock {\em Ann. Math. Statist.}, 27:1075--1090, 1956.

\bibitem{p77}
C.~Pomerance.
\newblock On the distribution of amicable numbers.
\newblock {\em J. Reine Angew. Math.}, 293/294:217--222, 1977.

\bibitem{simpsonthesis}
R.~Simpson.
\newblock The distribution of the invariant factors of multiplicative groups,
  to appear.
\newblock Thesis (Ph.D.)--University of British Columbia.

\bibitem{Chebyshev}
P.~Tchebychev.
\newblock M\'emoire sur les nombres premiers.
\newblock {\em Journal de mathématiques pures et appliquées, Série 1},
  17:366--390, 1852.

\bibitem{multivariate_normal}
Y.~L. Tong.
\newblock {\em The multivariate normal distribution}.
\newblock Springer Series in Statistics. Springer-Verlag, New York, 1990.

\end{thebibliography}
\bibliographystyle{plain}

\end{document}